\documentclass[12pt, a4paper]{amsart}
\usepackage{amsfonts,amsmath, amsthm, amssymb, amscd}
\usepackage{latexsym}
\input{amssym.def}
\input{amssym}

\usepackage[dvips]{graphicx}
\usepackage{psfrag}

\newtheorem{theorem}{Theorem}[section]
\newtheorem{lemma}[theorem]{Lemma}
\newtheorem{cor}[theorem]{Corollary}
\newtheorem{prop}[theorem]{Proposition}
\newtheorem{dfn}[theorem]{{Definition}}

\newtheorem*{rmk}{{Remark}}

\numberwithin{equation}{section}

\newcommand {\N}{\mathbb{N}} 
\newcommand {\R}{\mathbb{R}} 
\newcommand {\C}{\mathbb{C}} 

\DeclareMathOperator{\id}{id}
\DeclareMathOperator{\vol}{vol}

\DeclareMathOperator{\supp}{supp}

\DeclareMathOperator{\End}{End}
\DeclareMathOperator{\grad}{grad}

\DeclareMathOperator{\tr}{tr}
\DeclareMathOperator{\Span}{span}
\DeclareMathOperator{\Jac}{Jac}

\begin{document}
\title[Noncompact harmonic Manifolds]{Noncompact harmonic manifolds}

\author{Gerhard Knieper and Norbert Peyerimhoff }
\date{\today}
\address{Faculty of Mathematics,
Ruhr University Bochum, 44780 Bochum, Germany}
\email{gerhard.knieper@rub.de}
\address{Department of Mathematical Sciences, Durham University, Durham DH1 3LE, UK}
\email{norbert.peyerimhoff@dur.ac.uk}


\begin{abstract}
  The Lichnerowicz conjecture asserts that all harmonic manifolds are
  either flat or locally symmetric spaces of rank~1.  This conjecture
  has been proved by Z.I. ~Szab\'{o} \cite{Sz} for harmonic manifolds
  with compact universal cover.  E.~Damek and F.~Ricci \cite{DR}
  provided examples showing that in the noncompact case the conjecture
  is wrong. However, such manifolds do not admit a compact quotient.
  The classification of all noncompact harmonic spaces is still a very
  difficult open problem. 

  In this paper we provide a survey on recent results on noncompact
  simply connected harmonic manifolds, and we also prove many new
  results, both for general noncompact harmonic manifolds and for
  noncompact harmonic manifolds with purely exponential volume growth.
\end{abstract}


\maketitle

\tableofcontents

\section{Introduction}
A complete Riemannian manifold $X$ is called harmonic if the harmonic
functions satisfy the mean value property, that is, the average on any
sphere coincides with its value in the center.  Equivalently, for any
$p \in X$ the volume density $\theta_p(q) = \sqrt{\det g_{ij}(q)}$ in
normal coordinates, centered at any point $p \in X$, is a radial
function. In particular, if $c: [0, \infty) \to X$ is a normal
geodesic with $c(0) =p$, the function $f(t) := \theta_p(c(t))$ is
independent of $c$. It is easy to see that all rank~1 symmetric spaces
and Euclidean spaces (model spaces) are harmonic. In 1944,
A.~Lichnerowicz conjectured that conversely every complete harmonic
manifold is a model space. He confirmed the conjecture up to
dimension~4 \cite{Lic}. It was not before the beginning of the 1990's
that general results where obtained. In 1990 Z.I.~Szab\'{o} \cite{Sz}
proved the Lichnerowicz conjecture for compact simply connected
spaces. However, not much later, in 1992, E.~Damek and F.~Ricci
\cite{DR} showed that in the noncompact case the conjecture is
wrong. They provided examples of homogeneous harmonic spaces which are
not symmetric.  Nevertheless, in 1995 G.~Besson, G.~Courtois and
S.~Gallot \cite{BCG} confirmed the conjecture for manifolds of
negative curvature admitting a compact quotient. The proof consisted
in a combination of deep rigidity results from hyperbolic dynamics and
used besides \cite{BCG} the work of Y.~Benoist, P.~Foulon and
F.~Labourie \cite{BFL} and P.~Foulon and F.~Labourie \cite{FL}.

In 2002, A.~Ranjan and H.~Shah showed \cite{RSh2} that noncompact,
simply connected harmonic manifolds of polynomial volume growth are
flat. Using a result by Y.~Nikolayevski \cite{Ni} showing that the
density function $f$ is an exponential polynomial, subexponential
volume growth of noncompact simply connected harmonic manifolds
implies flatness as well.  In 2006, J.~Heber \cite{He} proved that
among the homogeneous harmonic spaces only the model spaces and the
Damek-Ricci spaces occur. Therefore, it remains to study
nonhomogeneous harmonic manifolds of exponential volume growth. In
particular, these are spaces without conjugate points and their
horospheres have constant mean curvature $h >0$.\\

As has been recently observed by the first author \cite{Kn3} the
asymptotic nature of the volume growth has a crucial impact on the
geometry of the harmonic space.  In particular, in \cite{Kn3} it has
been proved that purely exponential volume growth, geometric rank $1$,
Gromov hyperbolicity and the Anosov-property of the geodesic flow
(with respect to the Sasaki-metric) are equivalent for noncompact
harmonic spaces with $h >0$.  Note that the volume growth is called
purely exponential if there exists a constant $c \ge 1$ such that for
the volume density $f$ the estimate
$$
\frac{1}{c} \le \frac{f(t)}{e^{ht}} \le c
$$
holds for all $t \ge 1$. In \cite{Kn3} it is also proved that
nonpositive curvature or more generally no focal points imply the
above conditions.  We also note that all examples of noncompact
harmonic spaces including the Damek-Ricci spaces have nonpositive
curvature. Therefore, the following questions are fundamental in the
study of noncompact harmonic spaces:

\begin{itemize}
\item[(A)] Has every non-flat simply connected noncompact harmonic
  manifold {\em purely} exponential volume growth?
\item[(B)] Has every non-flat simply connected noncompact harmonic
  manifold  nonpositive curvature?
\item[(C)] Are there {\em nonhomogeneous} simply connected harmonic manifolds?
\end{itemize}

In particular, a negative answer to Question (C) implies a positive
answer to Question (A) and Question (B). If Question (B) has a
positive answer, Question (A) has a positive answer as well but not
necessarily vice versa.  If Question (A) has an affirmative answer its
proof could be considered as a first step into the direction of the
classification of all harmonic manifolds. But we like to mention that
even under the additional assumption that a non-flat simply connected
noncompact manifold $(X,g)$ admits a compact quotient, there is at
present no proof that $(X,g)$ has purely exponential volume growth
without a further assumption.  However, in case a compact quotient
exists and the volume growth is purely exponential one can deduce
using the above mentioned rigidity results that $(X,g)$ is a symmetric
space of negative curvature.  In particular, if the answer to question
(A) is yes it would solve the classification of harmonic spaces
admitting a compact quotient (see  \cite{Kn3} for details).\\

We refer the reader to \cite{Sz}, \cite{Krey} and \cite{Be} for well
known classical results on harmonic spaces, mostly related to the case
of simply connected {\em compact} harmonic spaces. For informative surveys
on {\em Damek-Ricci spaces}, we refer the reader to \cite{BTV} and \cite{Rou}.

The article does not cover recent results on asymtotically harmonic
spaces given in \cite{He}, \cite{Zi1,Zi2} and \cite{CaSam}, nor does it
present the integral geometric results for general noncompact spaces
given in \cite{PS}. Moreover, the article is complementary to
\cite{Kn3}. One of the aims of this article is to present many
important other recent results by several authors in a self-contained
way. Another aim is the presentation of a number of new results. An
overview over these results is given at the beginning of each of the
two parts of this article.

\bigskip
\bigskip

\noindent {\bf Acknowledgement:} This research was mainly carried out
during a Research in Pairs Programme (RiP) in August 2012 at the
Mathematisches Forschungsinstitut Oberwolfach. The authors want to
express their gratitude for this opportunity. We also like to thank
U. Dzwigoll for her support in typing this article.

\newpage

\part{General noncompact harmonic manifolds}

In this part, we give mostly self-contained presentations of the
following topics:
\begin{enumerate}
\item Nikolayevsky's result \cite{Ni} that the density function of a noncompact
  harmonic space is a exponential polynomial,
\item Ranjan/Shah's result \cite{RSh2} that harmonic spaces with
  polynomial volume growth are flat, and their integral formula for
  harmonic functions,
\item Zimmer's result \cite{Zi3} that Busemann and Martin boundary
  coincide on nonflat harmonic spaces, and the calculation of the
  Radon-Nykodym derivative of the visibility measures.
\end{enumerate}
The material is presented in our own framework and notation, and is
often very different from the articles mentioned above. We also
present several other new results, amongst them a uniform divergence
result for geodesic rays emanating from the same point (Chapter
\ref{chap:unifdifgeod}), various curvature properties of spheres and
horospheres (Chapter \ref{chap:curvspheres}), and a differential
inequality (see \eqref{eq:box}) for the density function $f$ of every
noncompact harmonic space. Moreover, we give an explicit formula for
the Green's kernel in terms of the density function, our treatment of
visibility measures and their Radon-Nykodym derivative differs
considerably from \cite{Zi3}, and we discuss representations of bounded
harmonic functions in connection with the Martin boundary.


\section{The density function of harmonic manifolds}

The main goal of this chapter is Nikolayevsky's result \cite{Ni} that
the density functions of harmonic manifolds are exponential
polynomials.

We start with the more general assumption that $(X,g)$ is a complete,
simply connected manifold without conjugate points. By a theorem of
Cartan-Hadamard, the exponential map $\exp_p: T_pX \to X$ is a
diffeomorphism. 

Let us first briefly recall some basic facts on the calculus of Jacobi
tensors (see e.g. \cite{Es}, \cite{Gre}, \cite{Kn2} and \cite{Kn} for
more details). Let $c: I \to X$ be a unit speed geodesic and let
$N(c)$ denote the normal bundle of $c$ given by a disjoint union
$$
N_t(c) := \{w \in T_{c(t)}X \mid \langle w, \dot{c}(t) \rangle = 0 \}.
$$
A $(1,1)$-tensor along $c$ is a differentiable section
$$
Y: I \to \End{N(c)}= \bigcup_{t \in I}\End( N_t(c)),
$$
i.e., for all orthogonal parallel vector fields $x_t$ along $c$ the
covariant derivative of $t \to Y(t) x_t$ exists. The derivative $Y'(t)
\in \End( N_t(c))$ is defined by
$$
Y'(t)(x_t) = \frac{D}{dt} \left(Y(t)x_t \right).
$$
$Y$ is called parallel if $Y'(t) = 0$ for all $t$.  If $Y$ is parallel
we have $Y(t)x_t = (Y(0)x)_t$ and, therefore, $\langle Y(t)x_t , y_t
\rangle$ is constant for all parallel vector fields $x_t, y_t$ along
$c$ .  In particular, $Y$ is parallel if and only if $Y$ is a constant
matrix with respect to parallel frame field in the normal bundle of
$c$. Therefore, parallel $(1,1)$-tensors are also called constant.

The curvature tensor $R$ induces a symmetric $(1,1)$-tensor along $c$
given by
$$
R(t) w : = R(w, \dot{c}(t)) \dot{c}(t).
$$
A $(1,1)$-tensor $Y$ along $c$ is called a Jacobi tensor if it solves
the Jacobi equation
$$
Y''(t) + R(t) Y(t) = 0.
$$
If $Y, Z$ are two Jacobi tensors along $c$ the derivative of the
Wronskian
$$
W(Y, Z)(t) := Y'^{\ast}(t) Z(t) - Y^{\ast}(t)Z'(t)
$$
is zero and thus, $W(Y, Z)$ defines a parallel $(1,1)$-tensor.  A
Jacobi tensor $Y$ along a geodesic $c: I \to X$ is called Lagrange
tensor if $W(Y, Y) =0$.  The importance of Lagrange tensors comes from
the following proposition.

\begin{prop} (see, e.g, \cite[Prop. 2.1]{Kn3}) \label{jac1} Let $Y: I
  \to \End{N(c)} $ be a Lagrange tensor along a geodesic $c: I \to X$
  which is nonsingular for all $t \in I$. Then for $t_0 \in I$ and any
  other Jacobi tensor $Z$ along $c$, there exist constant tensors
  $C_1$ and $C_2$ such that
  \begin{equation} \label{eq:jactensrep}
  Z(t) = Y(t) \left( \int\limits_{t_0}^{t}( Y^{\ast} Y)^{-1}(s) ds \
    C_1 + C_2 \right)
  \end{equation}
  for all $t \in I$. Conversely, every tensor of the form \eqref{eq:jactensrep}
  with $Y, C_1, C_2$ as above is a Jacobi tensor. 
\end{prop}

Let $SX$ denote the unit tangent bundle of $X$ with fibres $S_pX$, $p
\in X$ and $\pi: SX \to X$ be the canonical footpoint projection. For
every $v \in SX$, let $c_v: \R \to X$ denote the unique geodesic
satisfying $\dot{c}_v(0) = v$.  Define $A_v$ to be the Jacobi tensor
along $c_v$ with $A_v(0) = 0$ and $A_v'(0) = \id$.  Then the volume of
a geodesic sphere $S(p,r)$ of radius $r$ about $p$ is given by
$$
\vol S(p,r) = \int\limits_{S_pX}\det A_v(r) d \theta_p(v),
$$
where $d \theta_p(v)$ is the volume element of $S_pX$ induced by the
Riemannian metric. 

\begin{dfn} Let $(X,g)$ be a complete, simply connected manifold without
  conjugate points. $X$ is a {\em harmonic manifold} if the volume density
  $\det A_v(t)$ does not depend on $v \in SX$. We call the function
  $$ f(t) = \det A_v(t) $$
  the {\em density function} of the harmonic space $X$.
\end{dfn}

\begin{rmk} If $(X,g)$ is a harmonic space, we have
  $$
  \vol S(p,r) = \omega_n f(r),
  $$
  where $\omega_n$ is the volume of the sphere in the Euclidean
  space $\mathbb{R}^n$. Since
  $$
  \frac{f'(r)}{f(r)} = \frac{(\det A_v(r))'}{\det A_v(r)} = \tr
  (A_v'(r) A_v(r)^{-1})
  $$
  is the mean curvature of the geodesic sphere of radius $r >0$ about
  $\pi(v)$ in $c_v(r)$, $X$ is harmonic if and only if the mean
  curvature of all spheres is a function depending only on the radius.
\end{rmk}

\begin{prop}
  Let $(X,g)$ be a simply connected manifold without conjugate points
  and $v \in SX$. Let $B_v$ be the Jacobi tensor along $c_v$ with
  $B_v(0) = \id$ and $B_v'(0) = 0$. For $t \neq 0$, the tensor
  $$Q_v(t) = A_v^{-1} (t) B_v(t)$$
  is well defined, since $(X,g)$ has no conjugate points. The tensor
  $Q_v$ satisfies
  $$
  Q_v(t) - Q_v(s) = - \int\limits_s^t (A_v^* A_v)^{-1} (u) du
  $$
  for $0 < s \le t$.
\end{prop}

\begin{proof}
  Since $A_v$ is a Lagrange tensor along $c_v$ and nonsigular for $t
  \neq 0$ ($(X,g)$ has no conjugate points), Proposition \ref{jac1}
  yields
  \begin{equation} \label{BvAv}
  B_v(t) = A_v(t) (\int\limits_s^t (A_v^* A_v)^{-1} (u) du\, C_1 + C_2).
  \end{equation}
  For $t = s$, we obtain $B_v(s) = A_v(s) C_2$, and therefore $C_2 = Q_v(s)$.
  Differentiation at $t = s$ yields
  \begin{eqnarray*}
    B_v'(s) & = & A_v' (s) C_2 + A_v(s)(A_v^* A_v)^{-1} (s) C_1\\
    & = & A_v' (s) Q_v(s) + (A_v^*)^{-1}(s) C_1\\
    & = & (A_v' (s) A_v^{-1} (s)) B_v(s) + (A_v^*)^{-1} (s) C_1\\
    & = & (A_v^*)^{-1}(s) (A_v')^{*} (s) B_v(s) + (A_v^*)^{-1} (s) C_1.
  \end{eqnarray*}
  Here we used that $A_v' (s) A_v^{-1} (s)$ is symmetric since it is
  the second fundamental form of a geodesic sphere with radius $s$
  around $p$. This implies that
  $$
  A_v^*(s) B_v'(s) = (A_v')^*(s) B_v(s) + C_1,
  $$
  i.e., the constant tensor $C_1$ satisfies 
  $C_1 = A_v^*(s) B_v'(s) - (A_v')^*(s) B_v(s) = -W(A_v,B_v)(s)$. 
  At $s = 0$, we conclude $C_1(0) = - \id \cdot \id = - \id$. Plugging this
  into \eqref{BvAv}, we obtain 
  $$ Q_v(t) = A_v^{-1}(t) B_v(t) = - \int\limits_s^t (A_v^* A_v)^{-1}(u) du + 
  Q_v(s).
  $$
\end{proof}

\begin{prop} \label{prop:Avsjac}
  Let $A_{v,s}$ be Jacobitensor along $c_v$ with $A_{v,s}(s) = 0$ and
  $A_{v,s}'(s) = \id$. Then
  \begin{equation} \label{eq:Avsrep}
  A_{v,s}(t) = A_v(t) \int\limits_s^t (A_v^* A_v)^{-1} (u) du\, A_v^*(s).
  \end{equation}
  Note that $A_v^*(s)$ in the above formula is a constant tensor.
\end{prop}

\begin{proof}
  $A_{v,s}$ and the right hand side of \eqref{eq:Avsrep} are both
  Jacobitensors, because of Proposition \ref{jac1}. Since
  $$
  A_{v,s}(s) = 0 = A_v(s) \int\limits_s^s (A_v^* A_v)^{-1} (u) du\, A_v^*(s) = 0
  $$
  and
  $$
  A_{v,s}'(s) = \id = A_v'(s) \cdot 0 + A_v(s) (A_v^* A_v)^{-1}(s) A_v^*(s),
  $$
  both Jacobitensors agree.
\end{proof}

\begin{cor}
  We have
  $$
  Q_v(s) - Q_v(t) = A_v^{-1}(t) A_{v,s}(t) (A_v^*)^{-1}(s)
  $$
\end{cor}

\begin{proof}
  Using Proposition \ref{jac1} and \ref{prop:Avsjac}, we conclude
  $$
  Q_v(s) - Q_v(t) = \int\limits_s^t (A_v^* A_v)^{-1}(u) du = 
  A_v^{-1}(t) A_{v,s}(t) (A_v^*)^{-1}(s).
  $$
\end{proof}

Proposition \ref{prop:Avsjac} leads to the following important result
for manifolds without conjugate points:

\vspace*{.5cm}

\begin{center}
\fbox{\parbox{5cm} {\[\det (Q_v(s) - Q_v(t)) = \frac{\det
      A_{v,s}(t)}{\det A_v(t) \cdot \det A_v(s)} \qquad \text{for}\ 0 < s \leq
    t.\qquad \]}}
\end{center}

\vspace*{.5cm}

We assume now that $(X,g)$ is a harmonic manifold and are about to
present the main result of this chapter. We have $f(t) = \det A_v(t)$
and $\det A_{v,s}(t) = f(t-s)$. Therefore, we obtain
$$
\frac{f(t-s)}{f(t) f(s)} = \det (Q_v(s) - Q_v(t))\quad 
\text{where}\; Q_v(s) = A_v^{-1} (s) B_v(s).
$$

Our main result (Theorem \ref{thm:Ni} below) follows easily from the
following useful fact.

\begin{prop} \label{prop:finexppol} For a function $\varphi \in
  C^\infty({\R})$, we denote its translation to $s \in \R$ by
  $\varphi_s$ and define this function as
  $\varphi_s(t)=\varphi(t-s)$. Assume that the vector space, spanned by
  all translates of $\varphi$, is finite dimensional. Then $\varphi$
  is an exponential polynomial, i.e., we have
  $$
  \varphi(t) = \sum\limits_{i=1}^k (p_i(t) \sin(\beta_i t) + q_i(t)
  \cos (\beta_i t)) e^{\alpha_i t}
  $$
  where $p_i, q_i$ polynomials and $\beta_i, \alpha_i \in {\R}$.
\end{prop}

\begin{proof}
  Let $V$ be the finite dimensional vector space defined by
  $$
  V : = {\rm span}\; \{\varphi_s \mid\; s \in {\R}\}.
  $$
  
  The derivative of $\varphi$ can be expressed as the following limit
  of functions:
  \begin{eqnarray*}
    \varphi'(t) & = & \lim\limits_{s \to 0} \frac{\varphi(t) - \varphi(t-s)}{s}\\
    & = & \lim\limits_{s \to 0} \frac{\varphi(t) - \varphi_s(t)}{s}\\
    & = & \big(\lim\limits_{s \to 0} \frac{\varphi - \varphi_s}{s} \big) (t).
  \end{eqnarray*}
  We have $\frac{1}{s} (\varphi - \varphi_s) \in V$ for all $s > 0$
  and $\varphi' = \lim\limits_{s \to 0} \frac{1}{s} (\varphi -
  \varphi_s)$. Since $V$ is finite dimensional, it is closed and we
  have $\varphi' \in V$.

  Now, the $(\dim V) + 1$ functions
  $\varphi,\varphi',\dots,\varphi^{(\dim V)} \in V$ must be linear
  dependent over ${\R}$. This shows that $\varphi$ satisfies a linear
  ordinary differential equation with constant coefficients and is,
  therefore, an exponential polynomial.
\end{proof}

We are now ready to prove our main result.

\begin{theorem} (see \cite[Theorem 2]{Ni}) \label{thm:Ni} Let $(X,g)$
  be a harmonic manifold. Then the density function $f(t)$ is an
  exponential polynomial.
\end{theorem}

\begin{proof}
  As introduced earlier, we define the translations $f_s : {\R} \to
  {\R}$ by $f_s(t) = f(t-s)$, where $f$ is the density function. Since
  $Q_v(s) - Q_v(t)$ is a matrix with entries of the form $q(s)-q(t)$,
  we have
  \begin{eqnarray*}
  f_s(t) & = & \det (Q_v(s) - Q_v(t) f(t) f(s)\\
         & = & \sum\limits_{\alpha = 1}^N b_\alpha(s) c_\alpha(t),
  \end{eqnarray*}
  for some integer $N$ and with suitable smooth functions
  $b_\alpha,c_\alpha$.

  In particular, the vector space
  $$
  V : = {\rm span}\; \{f_s \mid\; s \in {\R}\} \subset {\rm span}\;
  \{c_\alpha \mid\; 1 \leq \alpha \leq N\}
  $$
  has dimension $\le N$. Applying Proposition \ref{prop:finexppol}, we
  see that $f(t)$ is an exponential polynomial.
\end{proof}

\begin{rmk} Let $(X,g)$ be a harmonic manifold. 

  (a) The volume of a geodesic ball $B_t(p)$ of radius $t>0$ in
  $(X,g)$ can be expressed by
  \begin{eqnarray*}
    \vol B_t(p) & = & \int\limits_0^t \vol S_u(p) du\\
    & = & \int\limits_0^t \int\limits_{S_pX} f(u) du = \omega_n 
    \int\limits_0^t f(u) du,
  \end{eqnarray*}
  where $\omega_n$ is the volume of the $(n-1)$-dimensional
  Euclidean unit sphere.

  (b) We have $f(t) = (-1)^{\dim X-1} f(-t)$, since
  $$
  A_v(t) = -A_{-v}(-t) = : C(t)
  $$
  which follows from the fact that both sides are Jacobi-Tensors along
  $c_v$ with $A_v(0) = C(0) = 0$ and $A_v'(0) = C'(0) = \id$. This
  implies
  \begin{eqnarray*}
    f(t) = \det A_v(t) = \det [-A_{-v}(-t)] & = & (-1)^{\dim X-1} \det A_{-v}(-t)\\
    & = & (-1)^{\dim X-1} f(-t).
  \end{eqnarray*}

  (c) The quotient $\frac{f'}{f}(r) \ge 0$ is the mean curvature of
  sphere $S_r(p)$ (with respect to the outward normal vector) in
  $(X,g)$, and $\frac{f'}{f}(r)$ is monotone decreasing to the
  (constant) mean curvature $h \ge 0$ of the horospheres of $(X,g)$.
\end{rmk}

\begin{cor} \label{cor:h0poly} Let $(X,g)$ be a harmonic
  manifold. Then the following properties are equivalent:
  \begin{itemize}
  \item $h=0$,
  \item $X$ has polynomial volume growth,
  \item $X$ has subexponential volume growth.
  \end{itemize}
\end{cor}

\begin{proof}
  Assume $h=0$. Then we have $\lim_{r \to \infty} \frac{f'(r)}{f(r)} =
  0$. Using l'H{\^o}pital, we obtain
  $$
  \lim\limits_{r \to \infty} \frac{\log f(r)}{r} = \lim\limits_{r \to
    \infty} \frac{f'(r)}{f(r)} = 0.
  $$
  This shows that $f(r)$ has subexponential volume growth. 

  Assume that $f(r)$ has subexponential volume growth. Since $f(r)$ is
  an exponential polynomial, this implies
  $$
  |f(r)| \leq C (1+r)^k \qquad \forall\; r \geq 0,
  $$
  with a suitable $C > 0$ and for some $k \geq 0$. Therefore, $(X,g)$
  has polynomial volume growth.

  Finally, assume that $(X,g)$ has polynomial volume growth. This implies that
  $$ \lim_{r \to \infty} \frac{\log f(r)}{r} = 0. $$
  On the other hand, we have
  $$ \lim_{r \to \infty} \frac{f'(r)}{f(r)} = h. $$
  Using l'H{\^o}pital, again, we conclude that
  $$ 0 = \lim_{r \to \infty} \frac{\log f(r)}{r} = 
  \lim_{r \to \infty} \frac{f'(r)}{f(r)} = h. $$ 
  This shows that $h=0$.
\end{proof}

\begin{rmk}
  In chapter \ref{minhorflat}, we will see that $h=0$ has an even much
  stronger implication: A harmonic manifold $(X,g)$ with $h=0$ is
  flat.
\end{rmk}

\section{Uniform divergence of geodesics}
\label{chap:unifdifgeod}

In this chapter, we prove that for every distance $d > 0$ and any
angle $\alpha > 0$ there exists a $t_0 > 0$, such that any two unit
speed geodesics $c_1,c_2$ starting at the same point and differing by
an angle $\ge \alpha$ will diverge uniformly in the sense that
$d(c_1(t),c_2(t)) \ge d$ for all $t \ge t_0$. For the proof, we start
with the following lemma.

\begin{lemma} \label{lem:jacrep}
  Let $(X,g)$ be a manifold without conjugate points. Denote by $S_v$
  the stable Jacobi tensor along the geodesic $c_v$ defined by $S_v =\
  \lim_{r \to \infty} S_{v,r}$, where $S_{v,r}$ is the Jacobi tensor
  along $c_v$ defined by the boundary conditions $S_{v,r}(0) = \id $ and
  $S_{v,r}(r) = 0$. Then we have
 \begin{itemize}
 \item[(i)] $S_v(t) = A_v(t) \int\limits_t^\infty (A_v^* A_v)^{-1}(u) du$,
 \item[(ii)] $ S_v'(0) -S_{v,r}'(0) = \int\limits_r^\infty (A_v^* A_v)^{-1}(u) du$.
 \end{itemize}
\end{lemma}

For a proof see \cite[pp. 227]{EOS}. 

\begin{cor} \label{cor:jacid} With the notation in Lemma \ref{lem:jacrep},
we have
$$
A_v'(t)A_v^{-1}(t) -S_v'(t)S_v^{-1}(t) = A_v^{-1}(t)^*(S_v'(0)
-S_{v,t}'(0))^{-1}A_v^{-1}(t).
$$
\end{cor}

\begin{proof}
Evaluating the Wronskian $W(A_v, S_v)(t)$ at $t= 0$ we obtain
$$
W(A_v, S_v)(t) =(A_v')^*(t) S_v(t) - A_v^*(t)S_v'(t) = \id
$$
which implies
$$
(A_v^*)^{-1}(t)(A_v')^*(t)- S_v'(t) S_v^{-1}(t)= (A_v^*)^{-1}(t)S_v^{-1}(t).
$$
Lemma \ref{lem:jacrep} yields
$$
S_v^{-1}(t)= (S_v'(0) -S_{v,t}'(0))^{-1}A_v^{-1}(t) .
$$
Since $B_v=A_v'A_v^{-1}$ is the second fundamental form of spheres
centered at $c_v(0)$, $B_v$ is a symmetric operator, i.e., 
$$
A_v'(t)A_v^{-1}(t) = (A_v'(t)A_v^{-1}(t))^* =(A_v^*)^{-1}(t)(A_v')^*(t).
$$
Combining these results yields the asserted identity.
\end{proof}

\begin{prop}
  Let $(X,g)$ be a noncompact harmonic manifold and $f$ its density
  function $f$. Then there exist constants $t_0 >0$ and $a>0$ such
  that
  $$
  \|A_v(t)\| \ge \frac{a}{\sqrt{\frac{f'(t)}{f(t)} -h}}
  $$
  for all $t \ge t_0$. 
\end{prop}

\begin{proof}
  We know that $(X,g)$ has no conjugate points and that the sectional
  curvature of $X$ is bounded. From the lower bound $-b^2 \le K_X$
  follows (see \cite[Cor. 2.12 in Section 1.2]{Kn}) that
  $$
  -b \le A_v'(t)A_v^{-1}(t) \le b \coth bt 
  $$ 
  for all $ t > 0 $, and
  $$
  -b \le S_v'(t) S_v^{-1}(t) \le b
  $$
  for all $t$. Choose $t_0 > 0$ such that $b \coth bt_0 =2$. Using the
  identity in Corollary \ref{cor:jacid}, we obtain for $t\ge t_0$
  \begin{eqnarray*}
    \| (S_v'(0) -S_{v,t}'(0))^{-1}\| &= &
    \|A_v(t)^*(A_v'(t)A_v^{-1}(t) -S_v'(t)S_v^{-1}(t))A_v(t)\| \\
    &\le &\|A_v(t) \|^2 3b.
  \end{eqnarray*}
  Since $S_v'(0) -S_{v,t}'(0)$ is positive definite by Lemma
  \ref{lem:jacrep}(ii), and
  $$
  \tr(S_v'(0) -S_{v,t}'(0)) =\frac{f'(t)}{f(t)} -h
  $$
  we obtain
  \begin{eqnarray*}
    \| (S_v'(0) -S_{v,t}'(0))^{-1}\| &\ge & 
    \frac{1}{\| (S_v'(0) -S_{v,t}'(0))\| }\\
    &\ge&  \frac{1}{\tr(S_v'(0) -S_{v,t}'(0))} = 
    \frac{1}{ \frac{f'(t)}{f(t)} -h }.
  \end{eqnarray*}
  This yields the required estimate.
\end{proof}

Using this we derive the uniform divergence of geodesics described above.

\begin{cor}\label{cor:uniformdiv}
  Let $c_v : [0, \infty) \to X$ and $c_w : [0, \infty) \to X$ be two
  geodesics with $v,w \in S_pX$. Then
  $$
  d(c_v(t), c_w(t)) \ge a(t)  \angle(v.w)
  $$
  where $a: [0,\infty) \to [0, \infty)$ is a function (not depending
  on $p \in X$) with $\lim\limits_{t \to \infty} a(t) = \infty$.
\end{cor}

\begin{proof}
 Let $c: [0,1] \to X$ be a geodesic connecting $c_v(t)$ with
  $c_w(t)$. Then $c$ is given by
  $$
  c(s) = \exp_p r(s) v(s),
  $$
  where  $v(s) \in S_p X$  such that $v(0) = v$ , $v(1) =w$ for
  all $0 \leq s \leq1$ and $r(0) = r(1) = t$. Then
  \begin{eqnarray*}
  \left.\frac{d}{ds} \right|_{s = s_0} c(s) &=& D \exp_p (
  r(s_0) v(s_0) )(r'(s_0) v(s_0) +  r(s_0) v'(s_0))\\
  &=& r'(s_0) c_{v(s_0)}' (r(s_0)) + A_{v(s_0)} ( r(s_0)) (v'(s_0)).
  \end{eqnarray*}
  Since $c_{v(s_0)}' (r(s_0)) \perp A_{v(s_0)} (r(s_0))
  (v'(s_0))$, we obtain
  \begin{eqnarray*}
    \left \| \left.\frac{d}{ds} \right |_{s=s_0} c \right
    \|^2 &= &(r'(s_0))^2 + ||A_{v(s_0)} ( r(s_0)) v'(s_0)||^2\\
   & \geq& ||A_{v(s_0)} ( r(s_0)) v'(s_0)||^2.
  \end{eqnarray*}
  If there exists $s_0 \in [0, 1]$ such that $r(s_0) \le \frac{t}{2}$
  then using the triangle inequality we have $d(c_v(t), c_w(t)) \ge
  t \ge (t/\pi) \angle(v,w)$. If this is not the case, we obtain for all $t > 0$
  \begin{eqnarray*}
    d(c_v(t), c_w(t)) = {\rm length}(c) &\ge& \int\limits_0^1 ||A_{v(s)} ( r(s)) v'(s)|| ds \\
    &\ge&\frac{a}{\sqrt{\frac{f'(t/2)}{f(t/2)} -h}} \int\limits_0^1||v'(s)|| ds\\
    &\ge&\frac{a}{\sqrt{\frac{f'(t/2)}{f(t/2)} -h}} \angle(v,w).
  \end{eqnarray*}
  The corollary follows now with the choice
  $$ a(t) = \min \left\{ \frac{t}{\pi}, 
  \frac{a}{\sqrt{\frac{f'(t/2)}{f(t/2)} -h}} \right\}. $$ 
\end{proof}

\section{Curvature properties of spheres and horospheres}
\label{chap:curvspheres}

This chapter is devoted to the proof of the following facts about spheres
and horospheres:

\begin{prop} \label{prop:curvsphhoro}
  Let $(X,g)$ be a noncompact harmonic manifold. Then
  \begin{enumerate}
  \item[(A)] There exist constants $C(R_0) > 0$ such that all spheres
    of radius $r \geq R_0 > 0$ (and horospheres) have sectional
    curvatures in $[-C(R_0),C(R_0)]$.
  \item[(B)] All spheres have constant scalar curvature where the
    constant depends only on the radius.
  \item[(C)] The horospheres have constant non-positive scalar
    curvature and the constant does not depend on the horosphere. The
    constant is zero if and only if $X$ has constant sectional
    curvature.
\end{enumerate}
\end{prop}

Recall that if $A_{v_0}$ is the Jacobi tensor along the geodesic $c_{v_0}$
with $A_{v_0}(0) = 0$ and $A_{v_0}'(0) = {\rm id}$, then $B_{v_0} = A_{v_0}' \cdot
A_{v_0}^{-1}$ is the second fundamental form of the sphere centered at
$c_{v_0}(0)$. $B_{v_0}$ is symmetric and satisfies the Ricatti equation $B_{v_0}'
+ B_{v_0}^2 + R = 0$. In the proof below we need another notation of the
second fundamental form: Let $S_r(p)$ be the sphere of radius $r > 0$
around $p$, $q \in S_r(p)$, and $v \in T_qX$ be the outward unit
normal vector of the sphere $S_r(p)$ at $q$. Let $v_0 \in T_pX$ be the
unit vector such that $v = c_{v_0}'(r)$. We also denote the second
fundamental form of $S_r(p)$ at $q$ by $B_{q,v}(r)$, that is, we have
$B_{q,v} = B_{v_0}(r)$ (see Figure \ref{fig_sff}).

\begin{figure}[h]
  \begin{center}
    \psfrag{p}{$p$} 
    \psfrag{q}{$q$}
    \psfrag{v0}{$v_0$}
    \psfrag{v}{$v$}
    \psfrag{Sr(p)}{$S_r(p)$}
    \psfrag{Bqv(r)=Bv0(r)}{$B_{q,v}(r)=B_{v_0}(r)$}
    \includegraphics[width=8cm]{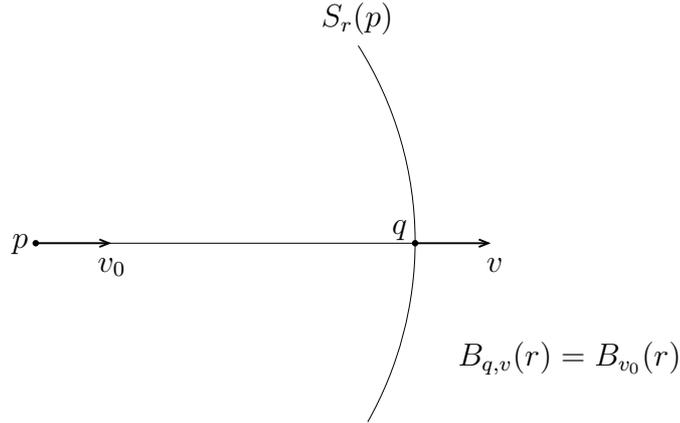}
  \end{center}
  \caption{Second fundamental forms of spheres}
  \label{fig_sff}
\end{figure}

\begin{proof}
  The Gauss equation for geodesic spheres implies
  \begin{multline*}
    \langle R^X(u,w)w,u \rangle = \langle R^{S_r(p)} (u,w)w,u \rangle \\
    + \langle u,B_{q,v}(r)w \rangle^2 - \langle B_{q,v}(r)u,u \rangle
    \cdot \langle B_{q,v}(r)w,w \rangle \qquad \forall u,w \in T_q S_r(p). 
  \end{multline*}
  The sectional curvature of a harmonic space is bounded. Therefore,
  there exists a constant $C_X > 0$ such that we have for all
  orthonormal vectors $u,w$:
  $$
  |K(\Span \{u,w\})| = | \langle R(u,w)w,u \rangle | \leq C_X.
  $$
  Since there exists $b > 0$ such that the curvature tensor along
  $c_{v_0}$ is bounded from below by $-b^2 \le R(r)$, we have (see
  \cite[Cor. 2.12 in Section 1.2]{Kn})
  $$ -b \leq B_{q,v}(r) \leq b \coth br.$$
  From \cite[Prop. 2.1]{RSh3} we also know that $B_{q,v}(r') \leq
  B_{q,v}(r)$ if $r' \geq r$. Therefore, we can find a constant $
  C_0(R_0)>0$ such that $\|B_{q.v}(r)\| \leq C_0(R_0)$ for all $r \geq
  R_0$. Hence, for $r \geq R_0$, we have
  \begin{eqnarray*}
  |K^{S_r(p)} (\Span \{u,w\})| &\leq& |K^X(\Span \{u,w\})| + 2 \|B_{q,v}(r)\|^2 \\
  &\leq& C_X + 2 C_0(R_0)^2 = C(R_0).
  \end{eqnarray*}
  This shows (A).

  \medskip

  For the proof of (B), we consider $u,w \in T_q S_r(p_0)$.\\
  Starting with the Gauss equation above for the geodesic sphere $S_r(p)$ and
  taking trace with respect to $u$ in $T_qS_r(p)$, we obtain
  \begin{multline*}
    {\rm Ric}^X(w,w) - \langle R^X(v,w) w,v \rangle = \\
    {\rm Ric}^{S_r(p)}(w,w) + \|B_{q,v}(r)w\|^2 - {\rm tr}[B_{q,v}(r)] \cdot
    \langle B_{q,v}(r)w,w \rangle\\
    = {\rm Ric}^{S_r(p)}(w,w) + \langle B_{q,v}(r)^2 w,w \rangle - {\rm tr}
    [B_{q,v}(r)] \cdot \langle B_{q,v}(r) w,w \rangle.
  \end{multline*}
  Taking trace again, now with respect to $w$ in $T_qS_r(p)$, we obtain
  $$
  (n-1) {\rm Ric}^X - {\rm Ric}^X = {\rm scal}^{S_r(p)}(q) + 
  {\rm tr}[B_{q,v}(r)^2] - ({\rm tr}[B_{q,v}(r)])^2,
  $$
  and therefore
  \begin{equation} \label{eq:spherric}
  (n-2) {\rm Ric}^X = {\rm scal}^{S_r(p)}(q) + \tr[B_{q,v}(r)^2] -
  (\tr[B_{q,v}(r)])^2
  \end{equation}
  The Ricatti equation 
  $$
  B_{v_0}'(r) + B_{v_0}^2(r) + R(r) = 0
  $$
  with $R(r)w = R(w,c_{v_0}'(r)) c_{v_0}'(r)$ yields, after applying the trace
  $$
  \tr [B_{v_0}'(r)] + \tr [B_{v_0}^2(r)] + {\rm Ric}^X = 0.
  $$
  From this and \eqref{eq:spherric} we conclude
  $$
  \tr[B_{v_0}'(r)] + (n-1) {\rm Ric}^X - {\rm scal^{S_r(p)}}(q) +
  (\tr[B_{v_0}(r)])^2 = 0.
  $$
  Since $\tr B_{v_0}(r) = \frac{f'}{f}(r)$, we obtain
  $$
  \left( \frac{f'}{f} \right)'(r) + (n-1) {\rm Ric}^X - {\rm
    scal}^{S_r(p)}(q) + \left( \frac{f'}{f}(r) \right)^2 = 0.
  $$
  Furthermore,
  $$
  \left( \frac{f'}{f} \right)'(r) = \frac{f'' f-(f')^2}{f^2}(r)
  $$
  implies
  \begin{equation} \label{eq:fricscal}
  \frac{f''}{f}(r) + (n-1) {\rm Ric}^X - {\rm scal}^{S_r(p)}(q) = 0,
  \end{equation}
  that is
  $$
  {\rm scal}^{S_r(p)}(q) = \frac{f''}{f}(r) + (n-1) {\rm Ric}^X.
  $$
  Hence ${\rm scal}^{S_r(p)}$ is constant, and the constant depends
  only on the radius $r$ but not on center $p$ of $S_r(p)$. Therefore,
  we can write it as ${\rm scal}^{S_r}$. This proves (B).
 
  \medskip

  Since the values ${\rm scal}^{S_r}$ converge to the scalar curvature
  ${\rm scal}^{\mathcal{H}}$ of any horosphere $\mathcal{H}$ as $r \to
  \infty$, we know that the scalar curvature of horospheres is
  constant and the constant does not depend on the
  horosphere. Moreover, we conclude that $\lim\limits_{r \to \infty}
  \frac{f''}{f}(r)$ exists and
  $$
  - \lim\limits_{r \to \infty} \frac{f''}{f}(r) = (n-1) {\rm Ric}^X -
  {\rm scal}^{\mathcal{H}}.
  $$
  Let us denote this limit by $\alpha$. Our next goal is to determine
  $\alpha$. We derive from \eqref{eq:fricscal} that $f$ satisfies the
  differential equation
  $$
  f''(r) + ((n-1) {\rm Ric}^X - {\rm scal}^{S_r})f(r) = 0.
  $$
  Note that we have  
  $$ \lim_{r \to \infty} (n-1) {\rm Ric}^X - {\rm scal}^{S_r} = \alpha. $$
  Using the result in \cite{Per}, we conclude that $\lim\limits_{r \to
    \infty} \frac{f'(r)}{f(r)} = \lim\limits_{r \to \infty}
  \frac{f_0'(r)}{f_0(r)}$, where $f_0$ is the solution of
  $$
  f_0''(r) + \alpha f_0(r) = 0. 
  $$
  On the other hand, we know that $\lim\limits_{r \to
    \infty} \frac{f'(r)}{f(r)} = h$. This shows that $\lim\limits_{r \to
    \infty} \frac{f_0'(r)}{f_0(r)} = h$, that is $\alpha = -h^2$. This shows
  that
  \begin{equation} \label{eq:ricscalhor} 
  (n-1) {\rm Ric}^X - {\rm scal}^{\mathcal{H}} = -h^2. 
  \end{equation}
  The Ricatti equation for orthogonal horospheres along $c_{v_0}$ yields
  $$
  \tr[(B_{v_0}^{\mathcal{H}})'(r)] + \tr [B_{v_0}^{\mathcal{H}}(r)^2]
  + {\rm Ric^X} = 0.
  $$
  Since the trace is linear and $\tr[B_{v_0}^{\mathcal{H}}(r)] = h$,
  the derivative vanishes and we end up with
  $$
  \tr [B_{v_0}^{\mathcal{H}}(r)^2] + {\rm Ric}^X = 0.
  $$
  Using $(n-1)\sum_{j=1}^{n-1} a_j^2 \ge \left(\sum_{j=1}^{n-1} a_j\right)^2$,
  we conclude
  \begin{equation} \label{eq:sqtrsq} 
  (n-1) \tr [B_{v_0}^{\mathcal{H}}(r)^2] \ge 
  \left( \tr[B_{v_0}^{\mathcal{H}}(r)] \right)^2 = h^2,
  \end{equation}
  which implies
  \begin{equation} \label{eq:richh} 
  (n-1) {\rm Ric}^X = -(n-1) \tr [B_{v_0}^{\mathcal{H}}(r)^2] \leq -h^2
  \end{equation}
  Combining \eqref{eq:ricscalhor} and \eqref{eq:richh}, we obtain
  $$
  {\rm scal}^{\mathcal{H}} = h^2 + (n-1) {\rm Ric}^X \leq h^2 - h^2 = 0.
  $$
  This shows the first part of (C), namely that the scalar curvature
  of horospheres is a non-positive constant.

  For the second part, we can assume that $h > 0$. (In the case $h=0$,
  the manifold $X$ is isometric to the flat $\R^n$, and all its
  horospheres are isometric to the flat $\R^{n-1}$.) The scalar curvature
  of the horospheres vanishes if and only if $(n-1){\rm Ric}^X = -h^2$, i.e.,
  if the inequality \eqref{eq:sqtrsq} holds with equality. The inequality
  \eqref{eq:sqtrsq} holds with equality if and only if
  $B_{v_0}^{\mathcal{H}}(r)$ is a multiple of the identity for every
  $r$. Since $\tr[B_{v_0}^{\mathcal{H}}(r)] = h$, this multiple does
  not depend on $v_0$ and $r$, and we have
  $$
  B_{v_0}^{\mathcal{H}}(r) = \frac{h}{n-1} \id.
  $$
  Looking at the Ricatti equation again
  $$
  (B_{v_0}^{\mathcal{H}})'(r) + (B_{v_0}^{\mathcal{H}}(r))^2 + R(r), 
  = 0,
  $$
  we see that
  $$
  R(r) = - \frac{h^2}{(n-1)^2} \id,
  $$
  i.e., the Jacobi operator along $c_{v_0}$ is a fixed constant
  multiple of the identity, which means that the sectional curvature
  of $X$ is constant.

  Therefore the scalar curvature of horospheres vanishes if and only
  if $X$ has constant sectional curvature. This yields (C).
\end{proof}

\section{Harmonic functions with polynomial growth}

For the reader's convenience, let us first recall the notion of sub-
and superharmonicity.

\begin{dfn}
  Let $(X,g)$ be a Riemannian manifold, $\Delta = {\rm div} \circ {\rm
    grad}$ be its Laplacian and $u \in C^2(X)$. $u$ is called {\em
    subharmonic} if $\Delta u \geq 0$. $u$ is called {\em
    superharmonic} if $\Delta u \leq 0$. Obviously, if $u$ is
  superharmonic then $-u$ is subharmonic, and vice versa. Moreover, a
  harmonic function is both sub- and superharmonic.
\end{dfn}

It is easy to see that subharmonic functions on arbitrary harmonic
manifolds satisfy the mean value inequality.

\begin{prop} (Mean Value Inequality) \label{prop:mvi} Let $(X,g)$ be a
  harmonic manifold and $u \in C^2(X)$ be subharmonic. Then we have for
  all $p \in X$ and $r > 0$,
  $$ u(p) \le \frac{1}{\vol S_r(p)} \int_{S_r(p)} u(q) d\mu_r(q). $$ 
\end{prop}

\begin{proof}
  Let $\pi_p: C(X) \to C(X)$ be the radialisation, i.e.,
  $(\pi_p u)(p) = u(p)$ and  
  $$ (\pi_p u)(q) = \frac{1}{\vol S_r(p)} \int_{S_r(p)} u(q) d\mu_r(q) $$
  for all $q \in X$ with $d(p,q) = r$. Let $u$ be subharmonic, i.e.,
  $\Delta u \ge 0$. It is well known that $\pi_p$ and $\Delta$
  commute. Therefore, we obtain
  $$ \Delta (\pi_p u) = \pi_p (\Delta u) \ge 0, $$
  i.e., $u_0 = \pi_p u$ is again subharmonic. Since $u_0$ is radial around $p$, 
  we have $u_0 = F \circ d_p$ and, using \eqref{eq:laplpolar}, we have
  \begin{equation} \label{eq:laplF} 
  0 \le \Delta u_0(q) = F'' \circ d_p(q) + (\frac{f'}{f}F') \circ d_p(q). 
  \end{equation}
  Note that $F$ is an even function with $F(0) = u(p)$ and $F'(0) = 0$. Since
  $$ \frac{1}{\vol S_r(p)} \int_{S_r(p)} u(q) d\mu_r(q) = u_0(q) = 
  F \circ d_p(q), $$ 
  it suffices to show that $F(r) \ge F(0) = u(p)$ for all $r \ge
  0$. From \eqref{eq:laplF} we conclude
  $$ (f F')'(r) = f(r) F''(r) + f'(r) F'(r) \ge 0. $$
  Integrating over $[0,r]$ and using $f(0) = 0 = F'(0)$, we obtain
  $$ f(r) F'(r) \ge 0, $$
  and since $f(r) > 0$ for $r > 0$, we see that $F$ is monotone increasing on
  $[0,\infty)$. This shows $F(r) \ge F(0)$, finishing the proof.   
\end{proof}

In the remainder of this chapter, we consider a harmonic manifold
$(X,g)$ of polynomial volume growth. The main goal is to prove that
the vector space of all harmonic functions of polynomial growth of
order $\leq D$ is finite dimensional. A main tool in the proof is the
following result for general Riemannian manifolds.

\begin{lemma} (compare with \cite[Lemma 28.3]{Li}) \label{lem:PLi} Let
  $(X,g)$ be a Riemannian manifold, $p \in X$, and $M \ge 0$ be a
  non-negative number. Let $K$ be a finite dimensional linear space of
  functions on $X$ such that for each $u \in K$, there exists a
  constant $C_u > 0$ such that
  \begin{equation} \label{eq:growthcond} 
  \int_{B_\rho(p)} |u(x)|^2 dx \le C_u (1+\rho)^M \qquad \forall \rho
  \ge 0.
  \end{equation}
  Then, for every $\beta > 1, \delta > 0, \rho_0 > 1$, there exists $\rho >
  \rho_0$ such that the following holds: if $\{u_i\}_{i=1}^k$ with
  $k=\dim K$ is an orthonormal basis of $K$ with respect to the
  quadratic form $A_{\beta \rho}(u,v) : = \int\limits_{B_{\beta
      \rho}(p)} u(x) v(x) dx$, then
  $$
  \sum\limits_{i=1}^k \int\limits_{B_\rho(p)} |u_i(x)|^2 dx \geq 
  k \beta^{-(M + \delta)}.
  $$
\end{lemma}

\begin{proof}
  We assume that the lemma is wrong. Then, for all $\rho > \rho_0$,
  there exists an orthonormal basis $\{u_i\}$ with respect to
  $A_{\beta \rho}$ such that
  $$
  \tr_{\beta \rho} A_\rho = \sum\limits_{i=1}^k A_\rho(u_i,u_i) =
  \sum\limits_{i=1}^k \int\limits_{B_\rho(p)} |u_i(x)|^2 dx < k
  \beta^{-(M+\delta)}.
  $$
  Note that the trace does not depend on the choice of orthonormal
  basis $\{u_i\}$.  The arithmetic-geometric mean states that
  $(\det_{\beta \rho} A_\rho)^{\frac{1}{k}} \leq \frac{1}{k} tr_{\beta
    \rho} A_\rho$. This implies that
  \begin{equation} \label{eq:det0} 
  0 \le {\rm det}_{\beta \rho} A_\rho
  \leq \beta^{-k(M+\delta)} \qquad \forall\; \rho \ge \rho_0.
  \end{equation}
  Note that for $\rho,\rho'$ we have $\det_{\rho'} A_\rho=\det
  U_{\rho,\rho'}$, where the endomorphism $U_{\rho,\rho'}$ is defined
  by $A_\rho(u,v) = A_{\rho'}(U_{\rho,\rho'}u,v)$.  This implies that
  $\det_\rho A_\rho = 1$ and
  \begin{equation} \label{eq:detprod} {\rm det}_{\rho''} A_\rho =
    \left( {\rm det}_{\rho''} A_{\rho'} \right) \left( {\rm
        det}_{\rho'} A_\rho\right). 
  \end{equation} 
  Applying this identity iteratively to \eqref{eq:det0}, we obtain
  \begin{equation} \label{eq:detit} {\rm det}_{\beta^j \rho} A_\rho
    \leq \beta^{-jk(M+\delta)} \qquad \forall\; \rho \ge \rho_0,
  \end{equation}
  for all positive integers $j$.

  For a fixed orthonormal basis $\{g_i\}$ of $K$ with respect to
  $A_\rho$, the growth condition \eqref{eq:growthcond} implies that there
  exists a $C > 0$ such that
  $$
  \int\limits_{B_{\rho'(p)}} |g_i(x)|^2 dx\leq C (1 + \rho')^M
  \quad \forall\; \rho' \geq 0.
  $$
  This implies that
  \begin{eqnarray*}
    0 \le {\rm det}_\rho A_{{\beta^j \rho}} & = & 
    \det \left(\int\limits_{B_{\beta^j \rho}(p)} g_i(x) g_k(x) dx\right)\\
    & \leq & C' \sum\limits_{\sigma \in S_k} (1+\beta^j \rho)^{kM} 
    \le k!\; C''\; \beta^{jkM}\; \rho^{kM},
  \end{eqnarray*}
  with suitably chosen constants $C',C'' > 0$.  From this we conclude,
  using again \eqref{eq:detprod},
  \begin{equation} \label{eq:det2}
  {\rm det}_{\beta^j \rho} A_\rho = ({\rm det}_\rho A_{\beta^j \rho})^{-1} \geq 
  \frac{1}{k!\; C''}\; \beta^{-jkM}\; \rho^{-kM}.
  \end{equation}
  Combining \eqref{eq:detit} and \eqref{eq:det2}, we obtain for all
  fixed $\rho \ge \rho_0$ and all positive integers $j$,
  $$ \frac{1}{k!\; C''}\; \beta^{jk\delta} \leq \rho^{kM}. $$
  Since $\beta > 1$, the left hand side tends to infinity as $j \to
  \infty$, whereas the right hand side is a constant. This is the
  desired contradiction.
\end{proof}

\begin{rmk}
  The case considered in \cite[Lemma 28.3]{Li} is a Riemannian
  manifold $(X,g)$ of polynomial volume growth, i.e.,
  $$ \vol(B_\rho(p)) \le C_X(1+\rho)^\mu \qquad \forall \rho \ge 0, $$
  and a vector space $K$ such that each $u \in K$ has polynomial
  growth of degree at most $D$, i.e., there exists a constant $c_u >
  0$ such that
  $$ |u(x)| \le c_u(1+d_p(x))^D \qquad \forall x \in X. $$
  This implies that
  $$ \int_{B_\rho(p)} |u(x)|^2 dx \le C_X c_u (1+\rho)^{2D+\mu} \qquad \forall
  \rho' \ge 0,$$
  and Lemma \ref{lem:PLi} is applicable in this situation with $M = 2D+\mu$.
\end{rmk}

Now we present the main result. 

\begin{theorem} (see also \cite[Theorem 4.2]{LiW} for a more general
  situation) \label{thm:LiW}
  Let $(X,g)$ be a harmonic manifold of polynomial growth,
  $p_0 \in X$ and $D$ a positive integer. Let
  $$
  H_D(p_0) := \{\varphi \in C^2(X) \mid\; \Delta \varphi = 0,
  \exists\; c_1,c_2 > 0: |\varphi(x)| \leq c_1 + c_2 d_{p_0}(p)^D \}
  $$
  Then $\dim H_D(p_0) < \infty$.
\end{theorem}

\begin{proof}
  Assume that $(X,g)$ is of polynomial volume growth of degree $\leq \mu$. 
  Henceforth, we fix the constants $\beta, \delta,
  \rho_0$ in Lemma \ref{lem:PLi}, in particular $\beta = 2$.

  Let $\mathcal{H} \subset H_D(p_0)$ be an arbitrary finite
  dimensional subspace, and denote its dimension by $k$. We will
  derive an upper bound on $k$. Recall that
  $$
  (f,g)_\rho := \int\limits_{B_\rho(p_0)} f(x)g(x) d \vol(x)
  $$
  is a proper inner product for every $\rho > 0$, since a function $f
  \in \mathcal{H}$ with $(f,f)_\rho = 0$ would have to vanish on
  $B_\rho(p_0)$. By the unique continuation principle for
  eigenfunctions (see \cite{Aro}), this would mean $f \equiv 0$. 

  The remark above shows that inequality \eqref{eq:growthcond} is
  satisfied with $M = 2D+\mu$. By Lemma \ref{lem:PLi}, we can then
  find an $r > \rho_0$ such that we have for all orthonormal bases
  $\{\varphi_i\} \subset \mathcal{H}$ with respect to
  $(\cdot,\cdot)_{2r}$:
  \begin{equation} \label{eq:on1} C k = C \sum\limits_{j=1}^k
    \int\limits_{B_{2r}(p_0)} \varphi_j^2(x) dx \leq
    \sum\limits_{j=1}^k \int\limits_{B_r (p_0)} \varphi_j^2(x) dx,
  \end{equation}
  with $C = 2^{-(2D+\mu+\delta)}$.

  Since $\sum_{j=1}^k \varphi_j^2$ is subharmonic, i.e., $\Delta
  (\sum\limits_{j=1}^k \varphi_j^2) \geq 0$, we can apply the maximum
  principle and obtain
  $$
  \int\limits_{B_r(p_0)} \sum\limits_{j=1}^k \varphi_j^2(x) dx \leq
  \sum\limits_{j=1}^k \varphi_j^2 (q) \cdot \vol (B_r(p_0)),
  $$
  for some $q \in \partial B_r(p_0)$. Since $(X,g)$ is harmonic, we
  have $\vol (B_r(p_0)) = \vol (B_r(q))$, and we obtain
  $$
  \int\limits_{B_r(p_0)} \sum\limits_{j=1}^k \varphi_j^2(x) dx \leq
  \sum\limits_{j=1}^k \varphi_j^2(q) \cdot \vol (B_r(q)).
  $$
  Choose an orthogonal transformation $A \in 0(k)$ such that the
  functions $\varphi_l = \sum\limits_{s=1}^k a_{ls} \varphi_s$ satisfy
  $\psi_2(q)=\cdots=\psi_k(q)=0$. Then $\sum\limits_{j=1}^k
  \varphi_j^2(x) = \sum\limits_{j=1}^k \psi_j^2(x)$ for all $x \in X$,
  and
  \begin{eqnarray}
    \int\limits_{B_r(p_0)} \sum\limits_{j=1}^k \varphi_j^2(x) dx &\leq& 
    \sum\limits_{j=1}^k \psi_j^2(q) \cdot \vol (B_r(p_0))
    = \vol (B_r(p_0)) \cdot \psi_1^2(q) \nonumber \\
    &=& \vol (B_r(q)) \cdot \psi_1^2(q) \le \int_{B_r(q)} \psi_1^2(x) dx 
    \label{eq:mviappl} \\
    &\le& \int_{B_{2r}(p_0)} \psi_1^2(x) dx \nonumber \\
    &=& \sum_{j,l=1}^k \int_{B_{2r}(p_0)} a_{1j} a_{1l} 
    \varphi_j(x) \varphi_l(x) dx \nonumber \\
    &=& \sum_{j=1}^k a_{1j}^2 \int_{B_{2r}(p_0)} \varphi_j^2(x) dx =  
    \sum\limits_{j=1}^k a_{1j}^2 = 1. \label{eq:orthappl}
  \end{eqnarray}
  Here, we used the Mean Value Inequality (Proposition \ref{prop:mvi})
  in \eqref{eq:mviappl}, and orthogonality of the functions
  $\varphi_j$ with respect to $(\cdot,\cdot)_{2r}$ in
  \eqref{eq:orthappl}.  Combining \eqref{eq:on1} with the inequality
  $\sum\limits_{j=1}^k \int\limits_{B_r (p_0)} \varphi_j^2(x) dx \le
  1$ just derived, we conclude that $C k \le 1$, i.e., $k = \dim
  \mathcal{H} \le 1/C = 2^{2D+\mu+\delta}$.

  Since $\mathcal{H} \subset H_D(p_0)$ was an arbitrary finite
  dimensional subspace, $H_D(p_0)$ itself must also be finite
  dimensional with dimension $\le 2^{2D+\mu+\delta}$.
\end{proof}

\section{Harmonic manifolds with $h=0$ are flat}
\label{minhorflat}

Recall from Corollary \ref{cor:h0poly} that harmonic manifolds $(X,g)$
with minimal horospheres, i.e., with $h=0$ ($h$ denotes the mean
curvature of the horospheres), must have polynomial volume growth. In
this chapter, we present the proof of the much stronger result, due to
\cite{RSh2}, that $h=0$ already implies that $(X,g)$ is flat.

The function $\mu$, which we introduce in the following proposition,
will play a crucial role in this proof. Let us first collect some
general properties of this function.

\begin{prop} \label{prop:mu} Let $(X,g)$ be a general harmonic
  manifold with density function $f(r)$. Then the function $\mu(r) =
  \frac{\int_0^r f(s)ds}{f(r)}$ satisfies the following properties:
  \begin{itemize}
  \item[(a)] We have $\mu(0) =0$, $\mu'(0) = \frac{1}{n}$, $\mu''(0) =
    0$, and $\mu'''(0) = \frac{2 {\rm Ric}^X}{n(n+2)}$.
  \item[(b)] We have $\mu(r) \ge 0$, for all $r \ge 0$.
  \item[(c)] We have $0 \le \mu'(r) \le 1$ for all $r \ge 0$.
  \item[(d)] We have $-\mu''(r) \mu(r) \le \frac{1}{4}$, for all $r \ge 0$. 
  \end{itemize}
\end{prop}

\begin{proof}

  {\bf (a)} Recall from \cite[Chapter 3.6]{Wi} that
  $$ f(r) = r^{n-1}(1 - \frac{\rm{Ric}^X}{6n}r^2 + O(r^4)). $$
  By integration, we obtain
  $$
  \int\limits_0^r f(t) dt = \frac{r^n}{n} (1 - \frac{{\rm
      Ric}^X}{6(n+2)} r^2 + O(r^4)).
  $$
  This implies that
  \begin{eqnarray*}
    \mu(r) = \frac{\int\limits_0^r f(s)ds}{f(r)} & = & \frac{r^r}{nr^{n-1}} 
    \frac{1 - \frac{{\rm Ric}^X}{6(n+2)} r^2 + O(r^4)}
    {1 - \frac{{\rm Ric}^X}{6n} r^2 + O(r^4)}\\
    & = & \frac{r}{n} \left(1 - \frac{{\rm Ric}^X}{6(n+2)} r^2 + O(r^4)\right)
    \left(1 + \frac{{\rm Ric}^X}{6n} r^2 + O(r^4)\right)\\
    & = & \frac{r}{n} \left(1 + \frac{{\rm Ric}^X}{6} r^2 \left(\frac{1}{n} - 
    \frac{1}{n+2}\right) + O(r^4)\right)\\
    & = & \frac{r}{n} + \frac{{\rm Ric}^X}{3n(n+2)} r^3 + O(r^5),
  \end{eqnarray*}
  which allows us to read off the results of (a). 

  {\bf (b)} This follows immediately from $f(r) > 0$ for all $r \ge 0$.

  {\bf (c)} Choose a point $p_0 \in X$. It is straightforward to see
  that $\mu$ satisfies the following differential equation
  \begin{equation} \label{eq:mudiffeqagain}
    \mu'(r) + \frac{f'}{f}(r) \mu(r) = 1, 
  \end{equation}
  i.e., $\Delta (\mu \circ d_{p_0}) = 1$, which shows that $\mu \circ
  d_{p_0}$ is subharmonic. Applying the maximum principle to $\mu
  \circ d_{p_0}$, we see that the restriction of this function to any
  closed ball around $p_0$ assumes its maximum at the boundary of this
  ball. But $\mu \circ d_{p_0}$ is constant along the boundary of any
  of these balls, and we conclude that
  $$ \mu'(r) \ge 0 \qquad \forall r > 0. $$
  On the other hand, using the above differential equation for $\mu$
  again, as well as $(f'/f)(r) \ge 0$ and $\mu(r) \ge 0$, we obtain
  $$ \mu'(r) = 1 - \frac{f'}{f}(r)\mu(r) \le 1. $$

  {\bf (d)} Rewriting \eqref{eq:mudiffeqagain}, we obtain
  $$
  \frac{f'}{f}(r) = \frac{1 - \mu'(r)}{\mu(r)},
  $$
  and, consequently, using the fact that $(f'/f)(r)$ converges
  monotonely decreasing to $h$,
  $$
  0 \geq \big(\frac{f'}{f} \big)'(r) = \frac{-\mu''(r)
    \mu(r)-(1-\mu'(r))\mu'(r)}{\mu^2(r)}.
  $$
  This implies that
  $$
  -\mu''(r) \mu(r) \leq (1 - \mu'(r)) \mu'(r).
  $$
  Since $0 \le \mu'(r) \le 1$, we conclude that 
  $$
  -\mu''(r) \mu(r) \leq \frac{1}{4}.
  $$
\end{proof}

We will also need the following general fact.

\begin{prop} \label{prop:gmu} Let $(X,g)$ be a general harmonic
  manifold and $p_0 \in X$. \ Let $\varphi \in C^2(X)$ be a function
  satisfying
  $$ \Delta \varphi = c, \qquad \varphi(p_0) = 0, $$
  for some constant $c \in \R$. Let $g \in C^2(\R)$ such that $g \circ
  d_{p_0} = \pi_{p_0}(\varphi) = \frac{1}{\omega_n}
  \int\limits_{S_{p_0}X} \varphi (c_w(r)) d \theta_{p_0}(w)$. Then we
  have
  $$ g'(r) = c \mu(r). $$
\end{prop}

\begin{proof}
  Note that $g(0) = \varphi(p_0) = 0$ and that $g$ is even, since
  $\pi_{p_0}(\varphi)$ is radial around $p_0$. Therefore, we have
  $g'(0)=0$. The function $g$ satisfies the differential equation
  $$ g''(r) + \frac{f'(r)}{f(r)} g'(r) = c, $$
  i.e,
  $$ f(r)g''(r) + f'(r)g'(r) = cf(r), $$
  which, in turn, transforms into
  $$ (fg')'(r) = cf(r). $$
  Integrating over $[0,r]$, and using $f(0)g'(0) = 0$, leads to
  $$ g'(r) = c \frac{\int_0^r f(t)dt}{f(r)} = c \mu(r). $$
\end{proof}

From now on, we assume that $(X,g)$ is a harmonic manifold with $h=0$,
i.e., all its horospheres are minimal. The main step in the proof of
flatness is to prove that $(X,g)$ is Ricci flat. We will show that
${\rm Ric}^X < 0$ would imply unboundedness of $-\mu'' \mu$ from
above, in contradiction to Proposition \ref{prop:mu}(d). On the other
hand, we cannot have ${\rm Ric}^X > 0$, because this would imply
compactness of $(X,g)$, by the Bonnet-Myers Theorem. But we only consider
noncompact harmonic manifolds.

Let $v \in S_{p_0}X$, and $b_v$ be the associated Busemann function
with $b_v(p_0) = 0$. Then
$$
\Delta b_v = h = 0,
$$
i.e., $b_v$ is harmonic. This implies that
\begin{equation} \label{eq:Db2}
\Delta b_v^2 = 2 \Vert \grad b_v \Vert^2 = 2,
\end{equation}
and we can apply Proposition \ref{prop:gmu} with $\varphi=b_v^2$. Then
the function $g \circ d_{p_0} := \pi_{p_0} (b^2)$ satisfies $g'(r) = 2
\mu(r)$. Our next goal is to prove that the function $g$ is a very
special exponential polynomial.

\begin{prop} \label{prop:g} $g$ has the following properties:
  \begin{itemize}
  \item[(a)] $g$ is an even function.
  \item[(b)] We have $0 \le g(r) \le r^2$, for all $r \ge 0$.
  \item[(c)] We have $0 \le g''(r) \le 2$, for all $r \ge 0$.
  \end{itemize}
\end{prop}

\begin{proof}
  {\bf (a)} Since $\pi_{p_0}(b^2)$ is radial around $p_0$, $g$ must be even.

  {\bf (b)} $0 \le g(r)$ follows from $g(0) = b_v(p_0)^2 = 0$, $g(r) =
  \frac{1}{2} \int_0^r \mu(t) dt$ and $\mu(r) \ge 0$, for $r \ge
  0$. Moreover,
  $$ g(r) = \frac{1}{\omega_n} 
  \int_{S_{p_0}X} \underbrace{b_v^2(c_w(r))}_{\le r^2} \le r^2. $$
  
  {\bf (c)} This follows from Proposition \ref{prop:mu}(c).
 \end{proof}  

 The next proposition is a key ingredient for the proof that $g$ is an
 exponential polynomial.

\begin{prop} \label{prop:F}
  Let $(X,g)$ be a harmonic manifold with minimal
  horospheres. Consider the vector space
  $$
  \mathcal{F} = \{\phi : X \to {\R} \mid \exists\; c,c_1,c_2>0\;
  \text{with}\; \Delta \phi = c, |\phi(x)| \leq c_1+c_2 d_{p_0}(x)^2 \}.
  $$
  Then $\mathcal{F}$ is finite dimensional.
\end{prop}

\begin{proof}
  We know from Theorem \ref{thm:LiW} that
  $$
  H_2(p_0) = \{\phi : X \to {\R} \mid \Delta \phi=0, \exists\;
  c_1,c_2>0\; \text{with}\; |\phi(x)| \leq c_1+c_2 d_{p_0}(x)^2 \}
  $$
  is finite dimensional. The map $\Phi : \mathcal{F} \to {\R},
  \Phi(\phi) = \Delta \phi$ is linear and $\ker \Phi =
  H_2(p_0)$. Therefore
  $$
  \dim \mathcal{F} \leq \dim H_2(p_0) + 1 < \infty.
  $$
\end{proof}

Next, we introduce the concept of translation of a radial function. 

\begin{dfn}
  Let $\phi \in C(X)$ be a radial function around $p_0$ with $\phi(q) =
  F(d_{p_0}(q))$. The translation of $\phi$ to another point $p \in X$
  is denoted by $\phi_p$ and defined by
  $$
  \phi_p(q) = F(d_{p_0}(q)).
  $$
\end{dfn}

We now consider the radial function $\phi = \pi_{p_0}(b_v^2)$. Then
$\phi = g \circ d_{p_0}$ and $\phi_p = g \circ d_p$. From
\eqref{eq:Db2} we conclude that
$$ \Delta \phi = \Delta \pi_{p_0} (b_v^2) = \pi_{p_0} (\Delta b_v^2) = 2. $$
The translation of $\phi$ to any point $p \in X$ satisfies also
$$ \Delta \phi_p = 2, $$
since $\Delta$, applied to a radial function, is in radial coordinates
independent of the centre. Using \ref{prop:g}(b), we have, for all $p
\in X$,
$$ |\phi_p(x)| \le d_p(x)^2 \le (d_p(p_0) + d_{p_0}(x))^2 \le 
2 d_p(p_0)^2 + 2 d_{p_0}(x)^2, $$ 
which shows that $\phi_p \in {\mathcal F}$, for all $p \in X$.

Let $\gamma: \R \to X$ be a geodesic with $\gamma(0) = p_0$. For a
function $\psi \in C^\infty(X)$, we define $\gamma^* \psi \in
C^\infty({\R})$ via $(\gamma^* \psi)(t) = \psi(\gamma(t))$. Let
${\mathcal F}$ be the finite dimensional vector space introduced in
Proposition \ref{prop:F}. Then the vector space $\tilde{\mathcal{F}} =
\gamma^* \mathcal{F} \subset C^\infty ({\R})$ has also finite
dimension.

Note that $g = \gamma^* \phi \in \tilde{\mathcal{F}}$. Let $g_s(t) :=
g(t-s)$. Then we have, for all $s \in \R$, $g_s = \gamma^*
\phi_{\gamma(s)} \in \tilde{\mathcal{F}}$.  Applying Proposition
\ref{prop:finexppol}, we see that $g$ is an exponential polynomial.
From Proposition \ref{prop:g}, we know that $0 \le g(t) \leq t^2$ and that
$g$ is even. Therefore, we must have
\begin{eqnarray*}
  g(t) = \sum\limits_{i=1}^{N_1} a_i \cos(\alpha_i t) & + & 
  t \sum\limits_{i=1}^{N_2} b_i \sin(\beta_i t)\\
  & + & t^2 \sum\limits_{i=1}^{N_3} c_i \cos (\gamma_i t),
\end{eqnarray*}
with suitable constants $a_i, b_i, c_i, \alpha_i, \beta_i, \gamma_i \in \R$.

Since $g''(t) \leq 2$ by Proposition \ref{prop:g}(c), this simplifies
to
$$
g(t) = \sum\limits_{i=1}^N a_i \cos(\alpha_i t) + c t^2.
$$
Differentiating $g$ trice, we obtain
$$
\mu''(t) = \frac{1}{2}g'''(t) = \frac{1}{2} \sum\limits_{i=1}^N a_i
\alpha_i^3 \sin(\alpha_i t).
$$

Now, we assume that ${\rm Ric}^X < 0$. Proposition \ref{prop:mu}(a)
then tells us that $\mu'''(0) = \frac{2{\rm Ric}^X}{n(n+2)} < 0$,
which implies that there exist $\delta, r_0 > 0$ such that $\mu''(r_0)
= - \delta < 0$. Since $\mu''$ is a finite sum of sines, $\mu''$
is almost periodic, and there exists a sequence $r_k \to \infty$ such that 
$\mu''(r_k) = - \delta$ (see, e.g., \cite{Bohr}). This implies that 
$-\mu(r_k) \mu''(r_k) = \delta \mu(r_k) \to \infty$, because of  
$$ \lim_{r \to \infty} \frac{1}{\mu(r)} = 
\lim_{r \to \infty} \frac{f(r)}{\int_0^r f(t)dt} = \lim_{r \to \infty}
\frac{f'(r)}{f(r)} = h = 0. $$ 
But $\-\mu(r_k) \mu''(r_k) \to \infty$ is in contradiction to
\ref{prop:mu}(d). This implies that $X$ must be Ricci flat.

Next, we consider the Ricatti equation
$$ U_v'(r) + U_v^2(r) + R_v(r) = 0, $$  
where $U_v(r)$ (a self-adjoint endomorphism on $c_v'(r)^\bot \subset
T_{c_v(r)}X$) denotes the second fundamental form of the horosphere
through $c_v(r)$, and centered at $c_v(-\infty)$, and $R_v(r)(w) =
R(w,c_v'(r))c_v'(r)$ is the Jacobi operator. Note that $\tr U_v(r) =
(\Delta b_v)(c_v(r)) = h = 0$ and, therefore, also $\tr U_v'(r) = 
(\tr U_v)'(r) = 0$. This implies that, after taking traces,
$$
tr U_v^2(r) = -\tr R_v(r) = -{\rm Ric}^X = 0,
$$
i.e., $U_v^2(r) = 0$, which implies that $U_v(r) = 0$. Inserting this
back into the Riccati equation, we end up with $R_v(r) = 0$, which
shows that $(X,g)$ is flat, finishing the proof.

Finally, we like to present another (very strong) criterion, which
implies flatness. For non-unit tangent vectors $v \in TX$, let us
define the corresponding Busemann functions by $b_v : = \Vert v \Vert
b_{v/\Vert v \Vert}$.

\begin{prop} \label{prop:bndim}
  Let $(X,g)$ be a harmonic manifold of dimension $n$. If the vector
  space
  $$
  {\rm span} \{b_v \mid v \in T_pX \}
  $$
  is $n$-dimensional, then the map $F: X \to \R^n$,
  $$
  F(x) = (b_{e_1}(x), \ldots, b_{e_n}(x))
  $$
  is an isometry. In particular, $(X,g)$ is flat.
\end{prop}

\begin{proof}
  Let ${\mathcal B} := {\rm span} \{b_v \mid\; v \in T_p X \}$. Note
  that the map $F: {\mathcal B} \to T_pX$, $F(b) = {\rm grad}\;b(p)$ is
  a bijection, since both vectors spaces ${\mathcal B}$ and $T_pM$
  have the same dimension and $F$ is surjective, because of ${\rm
    grad}\;b_v(p) = -v$. Let $v,w \in T_pX$. Since
  $$ F(b_v+b_w) = -(v+w) = F(b_{v+w}), $$
  we conclude that $b_{v+w} = b_v + b_w$. 

  Let $q \in X$. It is sufficient to show that $DF(q): T_qX \to \R^n$
  is a linear isometry. Note that 
  $$
  DF(q)(w) = \left( \langle \grad b_{e_1}(q),w \rangle, \ldots, 
  \langle \grad b_{e_n}(q),w \rangle \right).
  $$
  We first show that
  \begin{equation} \label{eq:bv1v2}
  \langle \grad b_{v_1}(q), \grad b_{v_2}(q) \rangle = \langle v_1,v_2
  \rangle \quad \forall\; v_1,v_2 \in T_p X,
  \end{equation}
  using $b_{v_1+v_2} = b_{v_1} + b_{v_2}$. Note that $\Vert b_v(q)
  \Vert = \Vert v \Vert$ for all $v \in TX$. Using both facts, we obtain
  \begin{eqnarray*}
    \Vert v_1+v_2 \Vert^2 &=& \langle \grad b_{v_1+v_2}(q), 
    \grad b_{v_1+v_2}(q) \rangle\\
    &=& \langle \grad b_{v_1}(q), \grad b_{v_1}(q) \rangle
    + \langle \grad b_{v_2}(q), \grad b_{v_2}(q) \rangle\\
    && + \; 2 \langle \grad b_{v_1}(q), \grad b_{v_2}(q) \rangle\\
    &=& \Vert v_1 \Vert^2 + \Vert v_2 \Vert^2 + 
    2 \langle \grad b_{v_1}(q), \grad b_{v_2}(q) \rangle.
  \end{eqnarray*}
  This shows \eqref{eq:bv1v2}. Therefore $\{\grad b_{e_i}(q) \mid 1
  \leq i \leq n\}$ are an orthonormal basis of $T_q X$. This implies
  that 
  $$
  \Vert DF(q)(w)\Vert^2 = \sum\limits_{i=1}^n \langle \grad
  b_{e_i}(q),w \rangle^2 = \Vert w \Vert^2.
  $$
\end{proof}

\begin{rmk}
  It is tempting to use Proposition \ref{prop:bndim}, to cook up an
  alternative proof of the fact that $h=0$ implies flatness of
  $(X,g)$. But our attempt to do this, falls short. Nevertheless, let
  us see, how far we get: In the case $h=0$, the harmonic manifold
  $(X,g)$ has polynomial growth, by Corollary
  \ref{cor:h0poly}. Moreover, we have $\Delta b_v = 0$, and therefore
  $$ {\rm span} \{ b_v \mid v \in TX \} \subset H_1(p_0). $$
  Theorem \ref{thm:LiW} tells us that this span is finite
  dimensional. In view of Proposition \ref{prop:bndim}, we would like
  to show that
  $$ {\rm dim} \left( {\rm span}  \{ b_v \mid v \in T_pX \} \right) = n. $$
  But we do not know how to derive such a precise result on the dimension
  of this space. 
\end{rmk}

\section{Special eigenfunctions of geodesic spheres}

Let $(X,g)$ be an arbitrary harmonic manifold with reference point
$p_0$. Ranjan/Shah introduce in \cite{RSh2} an interesting family of
functions $\varphi_v$ on $X \setminus \{ p_0 \}$. It turns out that
these functions, restricted to the geodesic spheres $S_r(p_0)$, are
eigenfunctions of the Laplacian $\Delta^{S_r(p_0)}$ for all radii $r >
0$. These eigenfunctions have just two nodal domains, both with half
the volume of $S_r(p_0)$, and it is natural to assume that $\varphi_v$
are eigenfunctions to the smallest non-trivial eigenvalue of
$\Delta^{S_r(p_0)}$ (see \cite[p. 690, Remark (iii)]{RSh2}). But there
is currently no proof of this assumption.

We first define $w_{p_0}(q) \in S_{p_0}X$ by
$$
exp_{p_0}(d_{p_0}(q) w_{p_0}(q)) = q,
$$
where $d_{p_0}(q) = d(p_0,q)$.

\begin{prop} (see \cite[Section 4]{RSh2})
  For each $v \in T_{p_0}X$, let $\varphi_v: X \setminus \{p_0\} \to
  \R$ be the function
  $$
    \varphi_v(q) = \langle v, w_p(q) \rangle.
  $$
  Then the restriction of $\varphi$ to the geodesic sphere $S_r(p_0)
  \subset X$ is an eigenfunction of the Laplacian $\Delta^{S_r(p)}$.
  The corresponding eigenvalue is $- \left(\frac{f'}{f}\right)' (r) > 0$.
\end{prop}

\begin{proof}
  Let $\gamma : (- \epsilon, \epsilon) \to X$ be a smooth curve with
  $\gamma'(0) = v \in T_{p_0}X$. Then we have
  \begin{eqnarray*}
    \frac{d}{ds} \bigg|_{s=0} d_q (\gamma(s)) & = & \langle \grad d_q(p), 
    \gamma'(0) \rangle\\
    & = & - \langle w_{p_0}(q), v \rangle = - \varphi_v(q).
  \end{eqnarray*}
  Using
  \begin{equation} \label{eq:laplpolar}
  \Delta_q^X (F\circ d_p)(q) = 
  F''\circ d_p(q) + \left(\frac{f'}{f}\; F'\right)\circ d_p(q)
  \end{equation}
  and the chain rule, we conclude
  \begin{eqnarray*}
    \Delta^X \varphi_v(q) & = & - \Delta_q \bigg( \frac{d}{ds} \bigg|_{s=0} d_q 
    (\gamma(s)) \bigg) = - \frac{d}{ds} \bigg|_{s=0} (\Delta_q d_{\gamma(s)})(q)\\
    & = & - \frac{d}{ds} \bigg|_{s=0} \frac{f'}{f} (d(\gamma(s),q)) = 
    - \bigg( \frac{f'}{f} \bigg)'(d_{p_0}(q))(- \varphi_v(q)).
  \end{eqnarray*}
  This implies 
  $$ \Delta^X \varphi_v(q) = \bigg( \frac{f'}{f} \bigg)' (d_{p_0}(q))
  \varphi_v(q). $$
  Since the radial derivatives of $\varphi_v(q)$ with respect to the the centre
  $p_0$ vanish, we have for all $q \in S_r(p_0)$ 
  $$ \Delta^X \varphi_v(q) = \Delta^{S_r(p_0)} \varphi_v(q), $$
  i.e.,
  $$ \Delta^{S_r(p_0)} \varphi_v(q) - \left(\frac{f'}{f}\right)' (r) 
  \varphi_v(q) = 0. $$
\end{proof}

\begin{rmk}
  Let $(X,g)$ be a harmonic manifold of dimension $n$ with reference
  point $p_0$.

  (a) The eigenspace ${\mathcal E}$ of $\Delta^{S_r(p_0))}$ to the
  eigenvalue $-\left(\frac{f'}{f}\right)'(r)$ has dimension $\ge n
  $. This follows from the fact that the map $T_{p_0}X \to {\mathcal
    E}$, $v \mapsto \varphi_v$ is linear and injective. 

  (b) We have the asymptotics $-(\frac{f'}{f})'(r) \to 0$ as $r \to
  \infty$, see Chapter \ref{chap:curvspheres}. It would be interesting
  to find out whether $-(\frac{f'}{f})' (r) > 0$ is the smallest non-zero
  eigenvalue of $S_r(p_0)$.

  (c) After the canonical identification of $S_r(p_0)$ with $S_{p_0}
  X$ via the exponential map and pullback of the Riemannian metric on
  $S_r(p_0) \subset X$, we obtain a family of Riemannian manifolds
  $(S_{p_0}X, g_r)$ such that the Laplace eigenfunction $\varphi_v$ is
  a spherical harmonic of degree $1$. In general, it is unlikely that
  spherical harmonics of higher degree are also Laplace
  eigenfunctions, but these spherical harmonics can be used to obtain
  orthonormal bases of the Hilbert spaces $L^2(S_{p_0}X, g_r)$, $r >
  0$, since the Riemannian measures of $(S_{p_0}X,g_r)$ are multiples
  of the Euclidean measure of $S_{p_0}X \subset T_{p_0}X$ (viewed as a
  round unit sphere).
\end{rmk}

Let $(X,g)$ be a harmonic manifold and $v \in T_p X$. Consider the
function $\varphi_v: X \to \mathbb{R}$ defined by
$$
\varphi_v(q) = \langle v,w_p q \rangle
$$
where $q = \exp_p (d(p,q)w_p(q))$. Since in the radial direction the
function is constant we have for $q = c_w(r)$: $\grad \varphi_v(q) \in
T_q S_{d(p,q)}(p)$. We want to calculate the gradient of $\varphi_v$:

Take a curve $\alpha : (- \epsilon, \epsilon) \to S_p X$ with
$\alpha(0) = w$. Then
$$
\varphi_v (c_{\alpha(s)}(r)) = \langle v, \alpha(s) \rangle
$$
This implies
$$
\left\langle \grad \varphi_v(q), \left.\frac{\partial}{\partial s}
\right|_{s=0} c_{\alpha(s)}(r) \right\rangle_q = \left\langle v,
\left.\frac{\partial}{\partial s} \right|_{s=0} \alpha(s) \right\rangle_p
$$

Let $A_w$ be orthonormal Jacobitensor along $c_w$ with $A_w(0) = 0$
and $A_w'(0) = \id$ . Then $\det(A_w(r)) = f(r)$ and
$$
\left\langle \grad \varphi_v(q), A_w(r) \left.\frac{\partial}{\partial s}
\right|_{s=0} \alpha(s) \right\rangle_q = \left\langle v,
\left.\frac{\partial}{\partial s} \right|_{s=0} \alpha(s) \right\rangle_p
$$
and therefore
\begin{eqnarray*}
\left\langle A_w^*(r) \grad \varphi_v(q), \left.\frac{\partial}{\partial s}
\right|_{s=0}\alpha(s) \right\rangle_p &=& \left\langle v,
\left.\frac{\partial}{\partial s} \right|_{s=0} \alpha(s) \right\rangle_p\\
&=& \left\langle v - \langle v,w \rangle w, \left.\frac{\partial}{\partial s}
\right|_{s=0} \alpha(s) \right\rangle_p.
\end{eqnarray*}

Since this holds for all $\left.\frac{\partial}{\partial s}
\right|_{s=0} \alpha(s)$ we have
\begin{equation} \label{eq:gradphivq}
A_w^*(r) \grad \varphi_v(q) = v -\langle v,w \rangle w
\end{equation}
and
$$
\grad \varphi_v(q) = (A_w^*(r))^{-1} (v -\langle v,w \rangle w).
$$
This yields
$$
\|\grad \varphi_v(q)\|^2 = \langle A_w^{-1} (r)(A_w^* (r))^{-1} (v -
\langle v,w \rangle w), v - \langle v,w \rangle w \rangle.
$$
Since $\varphi_v$ is eigenfunction with $\Delta^{S_r(p)} \varphi_v =
(\frac{f'}{f}(r)) \varphi_v$, we have
$$
\frac{\int\limits_{S_r(p)} \|\grad \varphi_v\|^2 d
  \mu_r}{\int\limits_{S_r(p)} \varphi_v^2 d \mu_r} = -
\frac{\int\limits_{S_r(p)} \Delta^{S_r(p)} \varphi_v(q) \varphi_v(q) d
  \mu_r(q)}{\int\limits_{S_r(p)} \varphi_v^2(q) d \mu_r(q)} = -\left(
  \frac{f'}{f} \right)'(r).
$$
On the other hand
$$
\int\limits_{S_r(p)}\varphi_v^2 d \mu_r = \int\limits_{S_pX} \langle
v,w \rangle^2 f(r) d \theta_p(w) = f(r) \frac{w_{n-1}}{n} \langle v,v
\rangle
$$
and
$$
\int\limits_{S_r(p)} \|\grad \varphi_v\|^2 d \mu_r = f(r)
\int\limits_{S_pX} \langle A_w^* A_w)^{-1} (r) \langle v - \langle v,w
\rangle w ), v - \langle v,w \rangle w \rangle d \theta_p(r)
$$
implies
$$
- \left( \frac{f'}{f} \right)'(r) \langle v,v \rangle =
\frac{n}{w_{n-1}} \int\limits_{S_pX} \langle (A_w^* A_w)^{-1} (r)
\langle v - \langle v,w \rangle w), v - \langle v,w \rangle w \rangle
d \theta_p
$$
for all $v \in T_p X$. Define the symmetric endomorphism
$$
H_w(r) u =
\begin{cases} (A_w^* (r) A_w(r))^{-1} (u) & \text{if $u \perp w$,}\\
0 & \text{if $u \in \Span\{w\}$.}
\end{cases}
$$
We obtain the identity
$$ 
\frac{n}{\omega_n} \int\limits_{S_pX} \langle H_w(r) v,v \rangle d
\theta_p(w) = - \left( \frac{f'}{f} \right)'(r) \langle v,v \rangle.
$$
Since both sides are symmetric linear forms, we conclude that
$$
\frac{n}{\omega_n} \int\limits_{S_pX} \langle H_w(r) u,v \rangle d \theta_p(w) = - \left( \frac{f'}{f} \right)'(r) \langle u,v \rangle
$$
and
\begin{equation}\label{eq:star}
  \frac{n}{\omega_n} \int\limits_{S_pX} H_w(r)(\cdot) d \theta_p(w) = 
  - \left( \frac{f'}{f} \right)'(r) \id_{T_pX}.
\end{equation}
Using the arithmetric-geometric mean, we have
\begin{equation} \label{eq:starstar} \left( \frac{1}{f^2(r)}
  \right)^{\frac{1}{n-1}} = \det \left(H_w(r) \big|_{w^ \perp}
  \right)^{\frac{1}{n-1}} \leq \frac{1}{n-1} tr H_w(r).
\end{equation}
Taking traces in \eqref{eq:star}
$$
\frac{n}{\omega_n} \int\limits_{S_pX} tr H_w(r)(\cdot) d
\theta_p(w) = -n \left( \frac{f'}{f} \right)' (r),
$$
and using \eqref{eq:starstar} yields the inequality:
$$
\frac{n(n-1)}{\omega_n} \int\limits_{S_pX} \left( \frac{1}{f^2(r)}
\right)^{\frac{1}{n-1}} d \theta_p(w) \leq -n \left( \frac{f'}{f}
\right)' (r).
$$
Therefore
$$
 (n-1) \frac{1}{f^{ \frac{2}{n-1}}(r)}  \leq - \left( \frac{f'}{f} \right)' (r).
$$
or equivalently

\begin{center}
  \fbox{\parbox{12cm} {
      \begin{equation} \label{eq:box} (n-1) \leq -f ^{\frac{2}{n-1}}
        (r) \left( \frac{f'}{f} \right)' (r).
      \end{equation}\\
      The density function $f(r)$ of every noncompact harmonic space 
      satisfies the differential inequality \eqref{eq:box}. In cases 
      $f(r) = r^{n-1}$ or $f(r) = \sin h^{n-1} (r)$ we have equality 
      in \eqref{eq:box}. }}
\end{center}

\begin{prop}
  The eigenvalue $-\left(\frac{f'}{f} \right)'(r)$ of the Laplacian on
  geodesic spheres of radius $r$ tends to zero as $r$ tends to
  $\infty$. If $-\left(\frac{f'}{f} \right)'(r)$ tends to zero on an
  exponential rate, the volume growth of $(X,g)$ is purely
  exponential.
\end{prop}

\begin{proof}
  The first assertion follows from \eqref{eq:ricscalhor} since
  $$
  \lim\limits_{r \to \infty} \left( -\frac{f'}{f} \right)'(r) =
  \lim\limits_{r \to \infty} \frac{f''}{ f}(r) - \frac{(f')^2}{f^2}(r)
  =0.
  $$
  To prove the second claim consider the function $a: [0, \infty) \to
  [0, \infty)$ defined by $f(r) = e^{hr} a(r)$. Since
  $$
  \frac{f'}{f} (r) = \frac{e^{hr}(ha(r) +a'(r))}{e^{hr}a(r)} = h +
  \frac{a'}{a} (r),
  $$
  we have
  $$
  \left( -\frac{f'}{f} \right)'(r) = \left( -\frac{a'}{a} \right)'(r).
  $$
  Since $\frac{f'}{f} (r)$ is monotonically decreasing and converging
  to $h$, $\frac{a'}{a} (r)$ is monotonically decreasing and
  converging to $0$. If $-\left(\frac{f'}{f} \right)'(r)$ tends to
  zero at an exponential rate we have constants $c >0$ and $r_0>0$
  such that
  $$
  \left( -\frac{f'}{f} \right)'(r) \le e^{-cr}
  $$
  for all $r \ge r_0$. Hence, for $r \ge r_0$ we obtain
  $$
  \frac{a'}{a} (r)= \int\limits_r ^\infty e^{-cs} ds =
  \frac{1}{c}e^{-cr},
  $$
  and therefore
  $$
  \frac{a(r)}{a(r_0)} \le \frac{1}{c^2}(e^{-cr_0} -e^{-cr}) <
  \frac{1}{c^2}e^{-cr_0},
  $$
  which yields the second assertion.
\end{proof}

\section{An integral formula for subharmonic functions}

Let $(X,g)$ be a harmonic manifold. Our main goal in this chapter is 
an explicit integral formal for the derivative of harmonic
functions which was first presented by Ranjan and Shah in \cite[Theorem
2.1]{RSh2}. Their formula generalises to sub- and superharmonic
functions. We present a derivation in this more general setting. Our
derivation differs from the proof in \cite{RSh2}.

Assume $u \in C^2(X)$ is subharmonic (i.e., $\Delta u \ge 0$). We
present the derivation in this case. The subharmonic case is derived
by replacing $u$ by $-u$. Let $\psi = u \cdot \varphi_v$ and $q =
c_w(r)$ for a suitable choice of $w \in S_pX$ and $r > 0$. Then
\begin{eqnarray*}
  \Delta \psi(q) &=& \frac{d^2}{d r^2} u (c_w(r)) \langle v,w \rangle +
  \frac{f'}{f} (r) \frac{d}{dr} u (c_w(r)) \langle v,w \rangle +
  \Delta ^{S_r(p)} (u \cdot \varphi_v)(q) \\
  &\ge& - (\Delta^{S_r(p)} u)(q) \varphi_v(q) + 
  \Delta^{S_r(p)} (u \cdot \varphi_v)(q).
\end{eqnarray*}
Integrations over $S_r(p)$ yields
\begin{eqnarray*}
  \int\limits_{S_r(p)} \Delta \psi(q) d\mu_r(q) &=& f(r) \frac{d^2}{d r^2} 
  \int\limits_{S_p X} u (c_w (r)) \langle v,w \rangle d \theta_p(w) + \\
  &+& f'(r) \frac{d}{dr} \int\limits_{S_p X} u(c_w(r)) \langle v,w \rangle 
  d \theta_p(w)\\
  &\ge& \int\limits_{S_r(p)} 
  (\Delta^{S_r(p)} u)(q) \cdot \varphi_v(q) d \mu_r(q)\\
  &=& - \int\limits_{S_r(p)} u(q) \bigg( \frac{f'}{f} \bigg)'(r) \varphi_v(q) 
  d\mu_r(q)\\ 
  &=& - f(r) \big( \frac{f'}{f} \big)'(r) \int\limits_{S_p X} u(c_w(r)) 
  \langle v,w \rangle d \theta_p(w).
\end{eqnarray*}
Introducing $g(r) = \int\limits_{S_p X} u(c_w(r)) \langle v,w \rangle
d \theta_p(w)$, we obtain
$$ f(r) g''(r) + f'(r) g'(r) + f(r) \left(\frac{f'}{f} \right)'(r) g(r) 
\ge 0.
$$
Since $f > 0$, we can divide by $f$
$$ g''(r) + \frac{f'}{f}(r) g'(r) + f(r) \left(\frac{f'}{f} \right)'(r) g(r) 
\ge 0, $$
and simplify the result to
\begin{equation} \label{eq:gdiffeq}
g''(r) + \left( \frac{f'}{f} g \right)'(r) \ge 0. 
\end{equation}
We have the initial conditions
$$
g(0) = u(p) \int\limits_{S_p X} \langle v,w \rangle d \theta_p(w) = 0
$$
and
$$
g'(0) = \int\limits_{S_p X} \langle \grad u(p), w \rangle \langle
v,w \rangle d \theta_p(w) = \frac{\omega_n}{n} \langle \grad
u(p),v \rangle.
$$
Next, we integrate \eqref{eq:gdiffeq} over $[0,r]$ and obtain
$$
g'(r) - g'(0) + \left( \frac{f'}{f} g \right)(r) - \left( \frac{f'}{f} g
\right)(0) \geq 0.
$$
Since 
\begin{eqnarray*} 
\lim\limits_{t \to 0} \left( \frac{f'}{f} g \right)(t) &=& 
\lim\limits_{t \to 0} \frac{n-1}{t} g(t) = (n-1) 
\lim\limits_{t \to 0} \frac{g(t) - g(0)}{t}\\ 
&=& (n-1) g'(0) = \frac{n-1}{n} \omega_n \langle \grad u(p),v \rangle,
\end{eqnarray*}
we conclude
$$
g'(r) + \left( \frac{f'}{f} g \right)(r) \geq n g'(0) = \omega_n
\langle \grad u(p),v \rangle.
$$
Multiplying by $f > 0$, we obtain
$$ (fg)'(r) = (fg')(r) + (f'g)(r) \ge \omega_n
\langle \grad u(p),v \rangle f(r). $$ 
Integrating, again, over $[0,r]$ and using $f(0)=0=g(0)$ leads to
$$ (fg)(r) = (fg)(r) - (fg)(0) \geq 
\omega_n \langle \grad u(p),v \rangle \int\limits_0^r f(t) dt, $$
i.e.,
\begin{eqnarray*}
  \langle \grad u(p),v \rangle &\leq& 
  \frac{f(r)}{\omega_n \int\limits_0^r f(t)dt} 
  \int\limits_{S_p X} u(c_w(r)) \langle v,w \rangle d \theta_p(w) \\
  &=& \frac{1}{\vol(B_r(p))} \int\limits_{S_r(p)} u(q) 
  \varphi_v(q) d \mu_r(q).
\end{eqnarray*}
This proves the following result:

\begin{theorem} \label{thm:intform}
  Let $(X,g)$ a harmonic manifold. If $u$ is subharmonic on $X$ (i.e.,
  $\Delta u \geq 0$) and $v \in T_pX$, then
  $$ \langle \grad u(p), v \rangle \le \frac{1}{\vol (B_r(p))} 
  \int\limits_{S_r(p)} u(q) \varphi_v(q) d \mu_r(q)
  $$
  for all $r > 0$.
\end{theorem}

\begin{cor} \label{cor:subpolharmconst}
  Let $(X,g)$ be a harmonic manifold with subexponential volume
  growth. If $u$ is subharmonic with sublinear growth, then $u$ is
  constant.
\end{cor}

\begin{proof}
  We conclude from Theorem \ref{thm:Ni} and the fact that $f(r) > 0$ that
  $$ f(r) = a r^n + o(r^n) \quad \text{as $r \to \infty$}, $$
  with an appropriate $a > 0$. This implies that we have for a
  suitably large $R_0 > 0$,
  $$ 
  \int\limits_0^r f(s) d s \geq \frac{a}{2} r^{n+1} \quad \forall\; r \geq R_0,
  $$
  and
  $$  \left| \frac{f(r)r}{\int_0^r f(s) d s} \right| \leq C 
  \quad \forall\; r \geq R_0, $$ 
  with an appropriate $C > 0$. Subexponential growth of $u$ implies
  $$ \frac{|u(x)|}{d(p,x)} \to 0\; \text{as}\; d(p,x) \to \infty. $$
  Consequently, we can find for every $\epsilon > 0$ an $R_1 > 0$ such that
  $$ \frac{|u(x)|}{d(p,x)} \le \epsilon, $$
  for all $x \in X$ with $d(p,x) \ge R_1$. Using Theorem
  \ref{thm:intform}, we conclude that, for all $r \ge
  \max\{R_0,R_1\}$ and all $v \in S_pX$,
  $$
  \langle \grad u(p), v \rangle \leq \frac{f(r)}{\int\limits_0^r f(s)
    d s} \epsilon r \leq C \epsilon. $$ 
  Since $\epsilon >0$ and $v \in S_pX$ were arbitrary, this shows that
  $\grad u(p) = 0$ for all $p \in X$, i.e., $u$ is a constant
  function.
\end{proof}

\begin{rmk} Let $(X,g)$ be a harmonic manifold.

  (a) If $u$ is a harmonic function, then $u^2$ is subharmonic:
  $$
  \Delta u^2 = u \cdot \Delta u + \Delta u \cdot u + 2 \langle
  \grad u, \grad u \rangle = 2 \Vert \grad u \Vert^2 \geq 0.
  $$

  (b) A similar argument shows for arbitrary harmonic manifolds
  $(X,g)$ (also those with exponential volume growth) that if $u$ is
  subharmonic (i.e., $\Delta u \geq 0$) with linear growth, then
  $\Vert \grad u \Vert$ is a bounded function. Examples of those
  functions are Busemann functions $b_v(q) = \lim_{t \to \infty}
  d(q,c_v(t)) - t$. In this case we have $\Delta b_v(q) = h \ge 0$
  ($h$ denotes the mean curvature of the horospheres) and $\Vert \grad
  b_v(q) \Vert = 1$ for all $q \in X$.

  (c) If all horospheres of $X$ are minimal, i.e., $h = 0$, then all
  Busemann functions are harmonic functions with linear
  growth. Moreover, $X$ has polynomial volume growth, by Corollary
  \ref{cor:h0poly}. This shows that Corollary
  \ref{cor:subpolharmconst} is sharp.

\end{rmk}

\section{Special harmonic functions}

Based on their eigenfunctions $\varphi_v$ on geodesic spheres of
harmonic manifolds, Ranjan/Shah introduced in \cite[formula
(4.1)]{RSh2} interesting harmonic functions, denoted by $h_v$. In the
case of harmonic manifolds with subexponential growth, these functions
$h_v$ have linear growth. In the case of harmonic manifolds with
exponential growth, the functions $h_v$ are bounded. These functions
can be used in the latter case to construct a diffeomorphism from the
harmonic manifold $(X,g)$ of dimension $n$ to an $n$-dimensional
Euclidean open ball (see Chapter \ref{sec:harmdiffeo}). Moreover, this
diffeomorphism is a harmonic map (see \cite[p. 690, Remarks
(i,ii)]{RSh2}).

\begin{theorem} \label{thm:hvharm}
  Let $(X,g)$ be a harmonic manifold of dimension $n$ and $v \in T_p
  X$. Then the function
  $$
  h_v(q) = \mu (d_p(q)) \varphi_v(q)
  $$
  with $\mu(r) = \frac{\int\limits_0^r f(s)ds}{f(r)}$ is harmonic.
\end{theorem}

We have seen earlier in the proof of flatness of noncompact harmonic
spaces with $h=0$ that the function $\mu$ played a crucial
role.

\begin{proof}
  Using \eqref{eq:laplpolar}, the fact that $\varphi_v$ is a Laplace
  eigenfunction on geodesic spheres, and $\grad (\mu \circ d_p) \perp
  \grad \varphi_v$, we obtain
  \begin{eqnarray*}
    \Delta h_v(q) & = & \left( \mu''(d_p(q)) + \frac{f'}{f} (d_p(q)) 
    \mu' (d_p(q)) \right) \varphi_v(q)\\
    & + & \mu(d_p(q)) \left( \frac{f'}{f} \right)'(d_p(q)) \varphi_v(q),
  \end{eqnarray*}
  i.e.,
  \begin{eqnarray*}
    \Delta h_v(q) & = & \varphi_v(q) \left( \mu'' (d_p(q)) + 
      \left(\frac{f'}{f} \mu \right)' (d_p(q)) \right)\\
    & = & \varphi_v(q) \left( \mu' + \frac{f'}{f} \mu \right)' \circ d_p(q).
  \end{eqnarray*}
  Therefore it suffices to show that $\mu' + \frac{f'}{f} \mu$ is
  constant. This follows immediately from
  \begin{equation} \label{eq:mufirstdiffeq}
  \mu'(r) = \frac{f(r)^2 - \int\limits_0^r f(s)ds
      f'(r)}{f(r)^2} = 1 - \frac{f'}{f} (r) \mu(r).
  \end{equation}
\end{proof}

\begin{rmk}
  (a) If $(X,g)$ is a harmonic manifold with subexponential growth, we
  have $\frac{\log f(r)}{r} \to 0$. Similar arguments as those in the
  proof of Corollary \ref{cor:subpolharmconst} show that then $\mu(r)
  = O(r)$, as $r \to \infty$, i.e., $\frac{\mu(r)}{r}$ is bounded and
  $h_v$ has linear growth. (We have already seen that a harmonic
  manifold with subexponential growth must be flat. But we like to
  stress that the growth behaviour of $h_v$ can be derived without
  this much stronger result.)

  (b) If $(X,g)$ is a harmonic manifold with exponential growth, we have
  \begin{equation} \label{eq:murinf}
  \lim\limits_{r \to \infty} \mu(r) = \lim\limits_{r \to \infty} 
  \frac{\int\limits_0^r f(s)ds}{f(r)} = \lim\limits_{r \to \infty}
  \frac{f(r)}{f'(r)} = \frac{1}{h},
  \end{equation}
  where $h > 0$ denotes the mean curvature of the horospheres of
  $X$. In this case, $h_v$ is a bounded harmonic function.
\end{rmk}

We end this chapter with the following straightforward observation.

\begin{lemma}
  Let $(X,g)$ be a harmonic manifold, $u \in C^\infty(X)$ be a
  harmonic function and $r>0$ with
  \begin{equation} \label{eq:uvv}
  u|_{S_r(p)} \in {\rm span} \{\varphi_v\big\vert_{S_r(p)} \mid v \in T_p X\}.
  \end{equation}
  Then there exists a vector $v \in T_pX$, such that we have $u = h_v$.
\end{lemma}

\begin{proof}
  By the assumption \eqref{eq:uvv}, we can find constants $\alpha_i
  \in \R$ such that $u\big\vert_{S_r(p)} = \sum\limits_{i=1}^n \mu(r)
  \alpha_i \varphi_{e_i} = \mu(r) \varphi_{\sum\limits_{i=1}^n \alpha_i e_i}$.
  Let $v = \sum\limits_{i=1}^n \alpha_i e_i$. Then
  $$ h_v\big\vert_{S_r(p)} - u\big\vert_{S_r(p)} = 0, $$ 
  and $h_v-u$ is harmonic. By the maximum principle, $h_v-u$ vanishes
  on the closed ball $B_r(p)$. By the unique continuation principle
  (see \cite{Aro}), we conclude that $h_v-u = 0$.
\end{proof}

\section{Ball model of a noncompact harmonic space}
\label{sec:harmdiffeo}

Let $(X,g)$ be a simply connected, noncompact harmonic manifold of
dimension $n$ with $h > 0$, $p \in X$ and $v \in T_pX$. Recall from
Theorem \ref{thm:hvharm} that
$$
h_v(q) = \mu(d(p,q)) \cdot \overbrace{\langle v,w_p(q) \rangle}^{\varphi_v(q)},
$$
where $w_p(q) \in S_pX, q=\exp_p(d(p,q) w_p(q))$ and 
$\mu(r) = \frac{\int\limits_0^rf(s)ds}{f(r)}$. Moreover, we have 
$\mu(0) = 0$ and $\mu'(0) = \lim\limits_{r \to 0} \mu'(r) = 1/n$ (see
Proposition \ref{prop:mu}). Since
$$
h_v(\exp_p(rw)) = \mu(r) \langle v,w \rangle \quad \forall\; w \in
S_pX, v \in T_pX,
$$
we have $h_v(p) = 0$, and $h_v$ is differentiable in $p$ with 
\begin{eqnarray*}
\langle \grad h_v(p),w \rangle & = & \frac{d}{dt} \Big|_{t=0} h_v(\exp_p(tw)) 
= \lim\limits_{t \to 0} \frac{\mu(r)}{r} \langle v,w \rangle\\
& = & \mu'(0) \langle v,w \rangle = \frac{1}{n} \langle v,w \rangle,
\end{eqnarray*}
i.e., $\grad h_v(p) = \frac{v}{n}$. For $q \neq p$, we have
$$
\grad h_v(q) = \mu'(d(p,q)) \varphi_v(q) \grad d_p(q) + \mu(d(p,q))
\cdot \grad \varphi_v(q).
$$
From \eqref{eq:gradphivq} we deduce
$$
\grad \varphi_v(q) = A_{w_p(q)}^*(r)^{-1}\left(v- \langle v,w_p q \rangle 
w_p q\right) \in T_q S_r(p),
$$
where $r=d(p,q)$. Moreover, we know from Theorem \ref{thm:hvharm} that
$\Delta h_v(q) = 0$.

\medskip

Let ${\mathcal O}_p$ be the set of all orthonormal bases of $T_pX$ and
$E_p=(e_1,\ldots,e_n) \in {\mathcal O}_p$. Define
$$
F_{E_p}(q) = (h_{e_1}(q),\ldots,h_{e_n}(q)).
$$

\begin{theorem}
  For all $E_p \in {\mathcal O}_p$, the map
  $$
  F_{E_p} : X \to B_{\frac{1}{h}}(0) = \{y \in {\R}^n \mid\; \Vert y
  \Vert < \frac{1}{h}\}
  $$
  is a harmonic map and a diffeomorphism.
\end{theorem}

\begin{proof} $F_{E_p}$ is a harmonic map since all its component
  functions are harmonic functions. The proof that $F_{E_p}$ is a
  diffeomorphism proceeds in two steps.

  \medskip

  \noindent
  {\bf Step 1:} $F_{E_p}$ is a bijection, i.e., for every $y \in
  B_{\frac{1}{h}}(0)$ there exists a unique $q \in X$ with $F_{E_p}(q) =
  y$. We have to solve 
  \begin{eqnarray*}
    F_{E_p}(q) & = & (h_{e_1}(q),\ldots,h_{e_n}(q)) = (y_1,\ldots,y_n)\\
    & = & \mu(d(p,q)) 
    (\langle e_1,w_p(q) \rangle,\ldots,\langle e_n,w_p(q) \rangle)
  \end{eqnarray*}
  This implies $\sum\limits_{i=1}^n y_i^2 = \mu^2(d(p,q))$, i.e.,
  $\Vert y \Vert = \mu(d(p,q))$, which defines $\mu(d(p,q))$ uniquely
  since $\mu : [0,\infty) \to [0,\frac{1}{h})$ is a bijection ($\mu(0)
  = 0$, $\mu'(r) > 0$ by the Strong Maximum Principle (see the proof
  of \cite[Lemma 4.1]{RSh2}), $\mu(r) \to
  \frac{1}{h}$ for $r \to \infty$ (see \eqref{eq:murinf})). Furthermore
  $$
  w_p(q) = \frac{1}{\Vert y \Vert} \sum y_i e_i
  $$
  and $q = \exp_p \left(\underbrace{(\mu^{-1})(\Vert y \Vert)}_{d(p,q)}
  \underbrace{\frac{1}{\Vert y \Vert} \sum y_i e_i}_{\in S_pX}\right)$.

  \medskip

  \noindent
  {\bf Step 2:} $DF_{E(p)}(q) w = 0$ for $w \in T_qX$ implies $w=0$.
  The assumption
  $$
  DF_{E(p)}(q) w = (\langle \grad h_{e_1}(q),w \rangle,\ldots,\langle
  \grad h_{e_n}(q),w \rangle) = 0
  $$
  implies that $\langle \grad h_{e_i}(q),w \rangle = 0$ for all $i$.
  This, in turn, implies for all $(\alpha_1,\ldots,\alpha_n) \in {\R}^n$:

  \begin{eqnarray*}
    0 = \sum\limits_{i=1}^n \alpha_i \langle \grad h_{e_i}(q),w \rangle & = & 
    \left\langle \sum\limits_{i=1}^n \alpha_i \grad h_{e_i}(q),w \right\rangle\\
    & = & \left\langle \grad h_{\sum \alpha_i e_i}(q),w \right\rangle,
  \end{eqnarray*}
  since a straightforward consequence of the definition of $h_v$ is
  $$ \lambda_1 h_{v_1}(q) + \lambda_2 h_{v_2}(q) = \mu(d(p,q)) \cdot 
  \langle \lambda_1 v_1 + \lambda_2 v_2,w_p(q) \rangle = h_{\lambda_1
    v_1 + \lambda_2 v_2}(q).$$ Thus we conclude that $\langle\grad
  h_v(q),w \rangle = 0$ for all $v \in T_pX$. Using the above formula for
  the gradient of $h_v$, we obtain for all $v \in T_pX$:
  $$
  \langle \mu'(d(p,q)) \varphi_v(q) \underbrace{\grad d_p(q)}_{(T_q
    S_r(p))^\perp} + \mu(d(p,q)) \underbrace{\grad \varphi_v(q)}_{T_q
    S_r(p)},w \rangle = 0.
  $$
  Next, we write $w \in T_qX$ as 
  $$ w = \underbrace{\langle w,\grad d_p(q) \rangle 
  \grad d_p(q)}_{(T_q S_r(p))^\perp} + w'$$ 
  with $w' \in T_q S_r(p)$. This implies 
  \begin{eqnarray}
    0 &=& \mu'(d(p,q)) \varphi_v(q) \langle w,\grad d_p(q) \rangle 
    \label{eq:dermu}\\
    &+& \mu(d(p,q)) \langle \grad \varphi_v(q),w' \rangle \quad 
    \forall\; v \in T_pX. \nonumber
  \end{eqnarray}
  Recall that we have
  $$
  \grad \varphi_v(q) = A_{w_p(q)}^*(r)^{-1} (v- \langle v,w_p q
  \rangle w_p q).
  $$
  Choose $v = w_p(q)$, i.e. $v - \langle v,w_p q \rangle w_p q = 0$,
  i.e. $\grad \varphi_v(q) = 0$. Then 
  $$
  0 = \mu'(d(p,q)) \varphi_{w_p(q)}(q) \langle w, \grad d_p(q) \rangle.
  $$
  Since $\varphi_{w_p(q)}(q) = \langle w_p(q),w_p(q) \rangle = 1$ and
  $\mu'(d(p,q)) > 0$, we obtain
  \begin{equation} \label{eq:wdpq}
  \langle w, \grad d_p(q) \rangle = 0.
  \end{equation}
  This implies together with \eqref{eq:dermu} that
  \begin{equation} \label{eq:partgrad}
  0 = \mu(d(p,q)) \langle \grad \varphi_v(q),w' \rangle \quad \forall\; 
  v \in T_pX.
  \end{equation}
  For $v \in (w_p(q))^\perp$ we have $\grad \varphi_v(q) =
  (A_{w_p(q)}^*(r))^{-1}(v)$, and since $(A_{w_p(q)}^*(r))^{-1}:
  (w_p(q))^\perp \to T_q S_r(p)$ is an isomorphism, we can realise
  every vector in $T_q S_r(p)$ as $\grad \varphi_v(q)$, in particular we
  can find $v \in T_pX$ such that $w' = \grad \varphi_v(q)$.
  Putting this into \eqref{eq:partgrad} yields
  $$
  0 = \underbrace{\mu(d(p,q))}_{> 0} \langle w',w' \rangle,
  $$
  i.e., $w'=0$. Equation \eqref{eq:wdpq} and $w'=0$ imply that
  $$
  w = \underbrace{\langle w, \grad d_p(q) \rangle}_{=0} \grad d_p(q) +
  w' = 0,
  $$
  finishing the proof.
\end{proof}

\begin{remark}
  The above proof can be simplified by using the Cartan-Hadamard
  Theorem: Since $X$ is simply connected and has no conjugate points,
  the map $\exp_p: T_pX \to X$ is a diffeomorphism. We can view
  $F_{E_p}$ as a radial rescaling of the inverse exponential map,
  followed by a canonical identification of $T_pX$ with $\R^n$ via the
  basis $e_1,\dots,e_n$.
\end{remark}

\medskip

Next, we calculate $F_{E_p}$ in the case of the hyperbolic plane. We
realise the hyperbolic plane as the Poincar{\'e} unit disk ${\mathbb
  D} = \{ z \in \C \mid |z| < 1 \}$ (see Figure \ref{fig_pud}) with
metric
$$ g = \frac{4(dx^2+dy^2)}{(1-|z|^2)^2}. $$

\begin{figure}[h]
  \begin{center}
    \psfrag{z}{$z$} 
    \psfrag{z0}{$z_0$}
    \psfrag{x}{$x$}
    \psfrag{y}{$y$}
    \psfrag{0}{$0$}
    \psfrag{X=PoincarediskD}{Poincar{\'e} unit disk ${\mathbb D}$}
    \includegraphics[width=7cm]{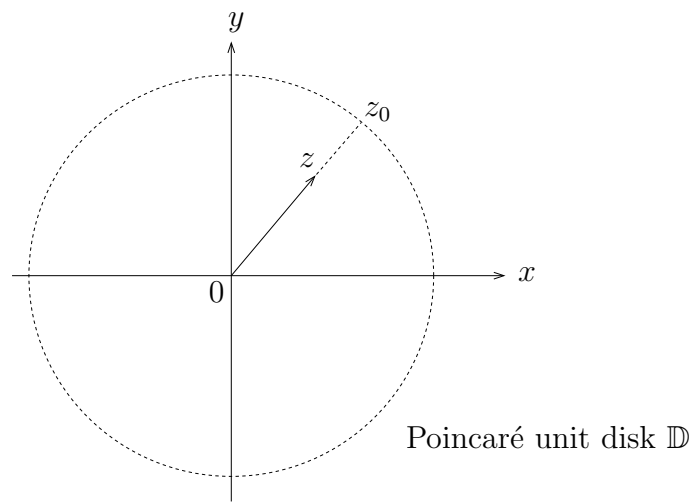}
  \end{center}
  \caption{The Poincar{\'e} unit disk model of the hyperbolic plane}
  \label{fig_pud}
\end{figure}

Let $z \in {\mathbb D}$ and $|z| = r<1$. Let $c : [0,r] \to
\mathbb{D}$ be the curve $c(t) = t \cdot z_0$ with $z_0 = \frac{z}{|z|}
\in S^1$. Then
$$
\Vert c'(t) \Vert_{\mathbb D}^2 = \frac{4}{1-t^2)^2}
$$
and
\begin{eqnarray*}
  d_{\mathbb D}(z,0) = \int\limits_0^r \Vert c'(t) \Vert_X\, dt = 
  2 \int\limits_0^r \frac{1}{1-t^2}\, dt &=& 
  2 \int\limits_0^{{\rm arctanh} (r)} \cosh^2 u(1-\tanh^2 u)\, du\\
  &=& 2 \int\limits_0^{{\rm arctanh}(r)} du = 2 {\rm arctanh}(r),
\end{eqnarray*}
using the substitution $t= \tanh(u)$, which implies $dt=(1-\tanh^2
u)du$ and $1-t^2 = 1 -\tanh^2 u = \frac{1}{\cosh^2 u}$.  Note that
$f(t) = \sinh(t)$, i.e.,
\begin{eqnarray*}
\mu(r) & = & \frac{\int\limits_0^r f(t)dt}{f(r)}\\
& = & \frac{\cosh r-1}{\sinh r}.
\end{eqnarray*}
Let $p = 0 \in {\mathbb D}$ and $E_p = \{ e_1,e_2 \}$, where $e_i$ is
the standard basis in $\R^2 \cong \C$. Note that $h=1$. Then
$$
F = F_{E_p} : {\mathbb D} \to B_{\frac{1}{h}}(0) = B_1(0)
$$
is given by $F(z) = \mu(d_{\mathbb D}(0,z)) \cdot
\frac{z}{|z|}$. Since $d_{\mathbb D}(0,z) = 2 {\rm arctanh}(r)$ for $z \in
{\mathbb D}$ with $|z|=r$, we obtain
\begin{eqnarray*}
  \mu(d_{\mathbb D}(0,z)) & = & \frac{\cosh (2 {\rm arctanh}(r))-1}
  {\sinh (2 {\rm arctanh}(r))}\\
  &  & \frac{\cosh^2({\rm arctanh}(r)) + \sinh^2({\rm arctanh}(r))-1}
  {2  \sinh({\rm arctanh}(r)) \cosh({\rm arctanh}(r))}\\
  & = & \frac{\sinh^2({\rm arctanh}(r))}{\sinh({\rm arctanh}(r)) 
    \cosh({\rm arctanh}(r))} = r.
\end{eqnarray*}
This implies that
$$ 
F(z) = r \cdot \frac{z}{r} = z,
$$
i.e., the model of the hyperbolic plane obtained via the harmonic map $F$ is 
the Poincar\'{e} unit disk, as we would have expected.

\section{The Busemann boundary}

Let $(X,g)$ be a noncompact complete Riemannian manifold and $p\in X$
fixed. Define $B^p: X \to C(X)$, where $B^p(y) =B^p_y: X \to
\mathbb{R}$ is given by
$$
B^p_y(x) = d(y,x) - d(p, y).
$$

\begin{lemma}
  Assume that $C(X)$ carries the topology of uniform convergence on
  compact sets.  Then $B^p: X \to C(X)$ is injective and continuous.
  Furthermore, $\overline {B^p(X)}$ is compact in $C(X)$.
\end{lemma}

\begin{proof}
  Consider $y, y' \in X$ such that $B^p_y = B^p_{y'}$. This implies
  $$
  d(x, y') -d(y', p) = d(x, y)-d(y, p)
  $$
  for all $x \in X$. In particular, for $x= y$ and $x= y'$ we obtain
  $$
  d(y, y') = d(y', p) - d(y, p) \; \text{and} \; -d(y', p) = d(y, y')
  - d(y,p),
  $$ 
  and therefore $d(y, y') =0$, which shows injectivity.

  To show that $B^p$ is continuous, we have to prove that, for each
  sequence $y_n $ converging to $y$, the sequence $B^p_{y_n}$ converges
  uniformly on compact subsets to $B^p_y$. The continuity of $d$
  implies that $B^p_{y_n}$ converges pointwise to $B_y$. Since
  $$
  |B^p_y(x) -B^p_{y}(x')| \le d(x, x')
  $$
  the family $B^p_{y_n}$ is equicontionous and by Arzela-Ascoli
  converges uniformly on compact subsets. Since $\overline {B^p(X)}$
  is a Fr{\'e}chet space, it is metrizable and we only have to check
  sequential compactness. Let $B^p_{y_n}$ be an arbitrary
  sequence. Since $|B^p_y(x)| \le d(x, p)$ is bounded on compact
  subsets independent of $y$, Arzela-Ascoli implies the existence of a
  subsequence converging uniformly on compact sets to a continuous
  function $f$.
\end{proof}

\begin{dfn}
  The set $\overline {B^p(X)}$ is called the {\em Busemann
    compactification} and $\partial_B^p X:=\overline {B^p(X)} \setminus
  B^p(X)$ is called the {\em Busemann boundary} with respect to $p$.
  Using the bijection $B^p: X \to B^p(X)$, a sequence $y_n \in X$
  converges to $\xi \in \partial_B^p X$ (in the {\em Busemann topology})
  iff $B^p_{y_n}$ converges to $\xi$ uniformily on compact subsets.
\end{dfn}
 
Note that the point $p$ is only a normalization (i.e., all functions
$f \in \overline {B^p(X)}$ satisfy $f(p) = 0$), and the convergence of
a sequence is independent of the choice of $p$. More precisely, if
$p'$ is another point, we have
$$
B_y^p(x) - B_y^{p'}(x) = d(y,p') -d(y, p) =B_ {y}^p(p').
$$
Hence, for a given sequence $y_n \in X$, the functions $B^p_{y_n}$ 
converge uniformly on compact sets to $\xi$ if and only if the
functions $ B_{y_n}^{p'} = B_{y_n}^{p} -B_{y_n}^{p}(p') $ converge
uniformly on compact sets to $\xi - \xi(p')$.

Note also that $y_n \to \xi \in \partial_B^p X$ means necessarily that
$d(p,y_n) \to \infty$: If $y_n$ would have a subsequence $y_{n_j}$
with $d(p,y_{n_j}) \le C$ for some $C > 0$, then a subsequence of
$y_{n_j}$ would be convergent to a point $y \in X$. Uniqueness of the
limit would imply $B_{y_n}^p \to B_y^p$, uniformly on compact sets,
but $B_y^p \not\in \partial_B^p X$.
  
If $(X,g)$ has no conjugate points, the Busemann boundary $\partial_B^p
X$ can be represented via Busemann functions. Let $S_v$ be the stable
Jacobi tensor along $c_v$ as defined in Lemma \ref{lem:jacrep}. The
corresponding unstable Jacobi tensor $U_v$ along $c_v$ is given by
$U_v(t) = S_{-v}(-t)$. We say that $(X,g)$ has {\em continuous
  asymptote} if the stable solution $v \mapsto S_v'(0)$ of the Ricatti
equation is continuous. Let us mention some important results for
manifolds without conjugate points (see \cite[Satz 3.5]{Kn2} and
\cite[Lemma 4.3]{Zi3}).

\begin{prop} \label{prop:busprops} Let $(X,g)$ be a complete, simply
  connected Riemannian manifold without conjugate points. For $v \in
  SX$ consider $b_{v,t}: X \to \mathbb{R}$, defined by $b_{v,t}(x) =
  d(x, c_v(t)) -t$. Then $b_v(x) = \lim\limits_{t \to \infty}
  b_{v,t}(x)$ exist and defines a $C^1$ function on $X$. The
   functions $\grad b_{v,t}$ converge to $\grad b_v$ in $C(X)$,
  i.e., the convergence is uniformly on compact sets. Furthermore, 
  $\grad b_v$ is Lipschitz continuous. 

  If $(X,g)$ has continuous asymptote, then the map $b: SX \to C(X)
  $ with $v \mapsto b_v$ is continuous.
\end{prop}

\begin{remark}
  For a proof of the first two properties in Proposition
  \ref{prop:busprops} above under the additional assumption of a
  continuous asymptote see \cite[Prop. 1 and 2]{Es}. Note that
  noncompact harmonic manifolds have continuous asymptote (see Ranjan
  and Shah \cite{RSh3} or Zimmer \cite[Lemma 5.4]{Zi3}).
\end{remark}

The following proposition is due to Zimmer \cite[Prop. 2.11 and Lemma
4.5]{Zi3}. For convenience we provide a proof.

\begin{prop} \label{prop:busboundprops}
  Assume that $(X,g)$ is a complete, simply connected manifold without
  conjugate points such that the map $b: SX \to C(X)$ is
  continuous. Then the following properties hold.
  \begin{enumerate}
  \item The sequence $b_{v,t}(x)$ converges uniformly on compact
    subsets of $X \times SX$ to $b_v(x)$.
  \item Let $\partial_B^{p_0} X$ be the Busemann boundary of $X$ with
    respect to $p_0 \in X$. Then, for any $p \in X$ the map
    $\varphi_p^{p_0}: S_pX \to \partial_B^{p_0} X$ with $\varphi_p^{p_0}(v):= b_v -
    b_v(p_0)$ is a homeomorphism.
  \item A sequence $y_n = exp_p(t_n v_n) \in X$ with $v_n \in S_pX$
    and $t_n \ge 0$ converges to a point $\xi \in \partial_B^{p_0} X$ if
    and only if $t_n \to \infty$ and there exists $v \in S_pX$ with
    $v_n \to v$. In particular, $\xi$ is given by $ b_{v} -
    b_{v}(p_0)$.
  \end{enumerate}
\end{prop}

\begin{proof}
  (1) From the triangle inequality, we obtain for all $v \in SX $ and
  $x \in X$ that $b_{v,t}(x) \le b_{v,s}(x)$ if $s \le t$. Since the
  map $(x,v) \to b_v(x)$ is continuous, Dini's theorem implies that
  $b_{v,t}(x)$ converges uniformly on compact subsets
  of $X \times SX$ to $b_v(x)$.\\
  (2) Proposition \ref{prop:busprops} implies that $\varphi_p^{p_0}$ is
  continuous. To show surjectivity assume that $B^{p_0}_{y_n}$
  converges to $\xi \in \partial_B^{p_0} X$. Define $v_n \in S_pX$ and
  $t_n \ge 0$ such that $c_{v_n}(t_n) = y_n$. Then
  $$
  B^{p_0}_{y_n}(x) = d(x, c_{v_n}(t_n)) -d(c_{v_n}(t_n),p_0) =
  b_{v_n,t_n}(x) - b_{v_n,t_n}(p_0).
  $$
  By passing to a subsequence if necessary, we can assume that $v_n$
  converges to $v\in S_pX$. Recall that $t_n \to \infty$. Because of
  (1), the right hand converges to $\xi = b_v - b_v(p_0)$. This shows
  that $\varphi_p^{p_0}$ is surjective.  The map is also injective since
  $\varphi_p^{p_0}(v) =\varphi_p^{p_0}(w) $ implies $-v = \grad b_v(p) = \grad
  b_w(p) = -w$. Therefore, $\varphi_p^{p_0} : S_pX \to \partial_B^{p_0} X$ is
  continuous and bijective and, since $S_pX$ is compact, $\varphi_p^{p_0}$
  is a homeomorphism.\\
  (3) Let $y_n= exp_p(t_n v_n) \in X$ with $v_n \in S_pX$ and $t_n \ge
  0$ be a sequence. Assume that $y_n$ converges to $\xi \in \partial_B^{p_0}
  X$, i.e., $B^{p_0}_{y_n}$ converges to $\xi = b_v - b_v(p_0)$ for some
  $v \in S_pX$. Since $d(y_n,p) \to \infty$, we know that $t_n \to \infty$.
  As above we have
  $$
  B^{p_0}_{y_n}(x) = d(x, c_{v_n}(t_n)) -d(c_{v_n}(t_n),p_0) =
  b_{v_n,t_n}(x) - b_{v_n,t_n}(p_0).
  $$
  If $v_n \not\to v$, we would have a subsequence of $v_n$ converging
  to $w \in S_pX$ with $w \neq v$. But then, following the arguments
  in (2), this subsequence would converge to $b_w - b_w(p_0)$,
  violating the injectivity of the map $\varphi_p^{p_0}$. Conversely, if
  $v_n \to v$ and $t_n \to \infty$, (1) implies $B^{p_0}_{y_n} =
  b_{v_n,t_n} - b_{v_n,t_n}(p_0) \to b_{v} - b_{v}(p_0) =: \xi$,
  uniformily on compact subsets of $X$.
\end{proof}

\section{Visibility measures and their  Radon-Nykodym derivative}
\label{sect:visib}

Let $(X,g)$ be a noncompact, simply connected harmonic manifold of
dimension $n$. We choose a reference point $p_0 \in X$ and define
$\partial_B X := \partial_B^{p_0} X$. For any other point $p \in X$,
we know from Proposition \ref{prop:busboundprops} that the map
$\varphi_p^{p_0} \circ (\varphi_p^p)^{-1}: \partial_B X
\to \partial_B^p X$ is a homeomorphism, identifying both Busemann
boundaries in a canonical way.  Moreover, the homeomorphisms
$\varphi_p^{p_0}: S_pX \to \partial_B^{p_0} X$ motivate the following
definition.

\begin{dfn} \label{def:vismeas}
  Let ${\mathcal M}_1(\partial_B X)$ denote the space of Borel
  probability measures on the Busemann boundary $\partial_B X$. For every
  $p \in X$, we define $\mu_p \in {\mathcal M}_1(\partial_B X)$ via
  $$ \int_{\partial_B X} f(\xi)\, d\mu_p(\xi) = \frac{1}{\omega_n}
  \int_{S_pX} f(\varphi_p^{p_0}(v))\, d\theta_p(v) \quad \forall\, f
  \in C(\partial_B X), $$ 
  where $\omega_n$ is the volume of the $(n-1)$-dimensional
  standard unit sphere and $d\theta_p$ is the volume element of $S_pX$
  induced by the Riemannian metric.

  We call $\mu_p$ the {\em visibility measure} of $(X,g)$ at the point
  $p$.
\end{dfn}

We will see that any two visibility measures $\mu_p, \mu_q \in
{\mathcal M}_1(\partial_B X)$ are absolutely continuous, by calculating
their Radon-Nykodym derivative via a limiting process. 
This needs some preparations.

\begin{lemma} \label{lem:bijvismap}
  Let $(X,g)$ be a noncompact simply connected harmonic space. 
  For all $p, q \in X$ there exists a $t(p,q) > 0$ such that for all
  $t \ge t(p,q)$ and all $v \in S_qX$ the geodesic ray $c_v: [0,
  \infty) \to X$ intersects $S_t(p)$ in a unique point $F_t(v)$ (see
  Figure \ref{fig_ftqp}). In particular, the map $F_t: S_q X \to
  S_t(p)$ is bijective.
\end{lemma}

\begin{figure}[h]
  \begin{center}
    \psfrag{SqX}{$S_qX$} 
    \psfrag{q}{$q$}
    \psfrag{v}{$v$}
    \psfrag{Ft(v)}{$F_t(v)$}
    \psfrag{St(p)}{$S_t(p)$}
    \psfrag{t}{$t$}
    \psfrag{p}{$p$}
    \includegraphics[width=8cm]{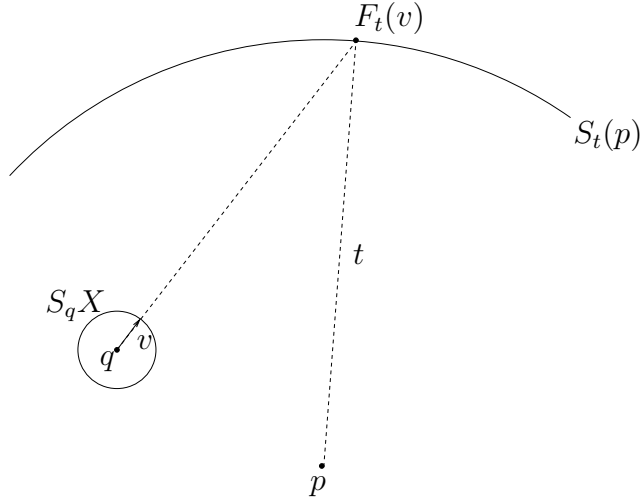}
  \end{center}
  \caption{Illustration of the map $F_t: S_qX \to S_t(p)$}
  \label{fig_ftqp}
\end{figure}

\begin{proof}
  Let $a(t)$ be as in Corollary \ref{cor:uniformdiv} . Choose $t_0$
  such that for all $t \ge t_0$ we have $2 d(p,q) \le a(t)$. Define
  $$
  t(p,q) = \max \{ d(p,q) +1, t_0 \}.
  $$
  In particular, $q$ lies in the ball of radius $t$ around $p$, for
  all $t \ge t(p,q)$, and hence for all $v \in S_qX$ the geodesic ray
  $c_v: [0, \infty) \to X$ intersects $S_t(p)$. Let $t \ge t(p,q)$,
  and assume that $q' = c_v(t_1)$ is the second intersection
  point. Let $w \in S_pX$ be the unique vector such that $c_w(t) =
  q'$. Since $\dot c_v(t_1)$ either points into $B_t(p)$ or is tangent
  to $S_t(p)$ we have
  $$
  \angle (\dot c_v(t_1), \dot c_w(t)) \ge \pi/2.
  $$
  Using the triangle inequality we obtain
  $$
  t- d(p,q) \le t_1 \le t + d(p,q)
  $$
  Using Corollary \ref{cor:uniformdiv}, we obtain for all $s \ge 0$
  $$
  d(c_v(t_1-s), c_w(t-s)) \ge a(s) \pi/2.
  $$
  In particular for $s =t$ this yields 
  $$
  a(t) \pi/2 \le d(c_v(t_1-t), p) \le d(c_v(t_1-t),q ) + d(q,p) \le 2
  d(p,q) \le a(t),
  $$
  which is a contradiction. Hence, a second intersection point cannot
  occur.
\end{proof}

\begin{prop}
  Let $(X,g)$ be a complete, simply connected noncompact manifold
  without conjugate points and $p,q \in X$. Consider the map $F_t: S_q
  X \to S_t(p)$, where $F_t(v)$ is the first intersection point of the
  geodesic ray $c_v: [0,\infty) \to X$ with $S_t(p)$. If $q$ is
  contained in the ball of radius $t$ about $p$, this map is well
  defined. Then the Jacobian of $F_t$ is given by
  $$
  \Jac F_t(v) = \frac{\det A_v(d(q,F_t(v)))}{\langle N_p(F_t(v)),N_q(F_t(v)) 
  \rangle},
  $$
  where $N_x(y) = (\grad d_x)(y)$ and $A_v$ is Jacobitensor along
  $c_v$ with $A_v(0) = 0$ and $A_v'(0) = \id$. Note that we have $\det
  A_v(s) = f(s)$ if $(X,g)$ is harmonic.
\end{prop}

\begin{proof}
  Choose a curve $\gamma : (- \epsilon, \epsilon) \to S_q X$ with
  $\gamma(0) = v \in S_q X$. Then
  $$
  F_t(\gamma(s)) = \exp_q (d(q, F_t(\gamma(s)) \cdot \gamma(s)),
  $$
  and, using the chain rule and the product rule,
  \begin{multline*}
    D F_t(v)(\dot{\gamma}(0)) = \\
    D \exp_q (d(q,F_t(v)) \cdot v)(\langle N_q(F_t(v)), D F_t(v)
    \dot{\gamma}(0) \rangle v + d(q,F_\gamma(v)) \cdot
    \dot{\gamma}(0)).
  \end{multline*}
  Note that $\dot{\gamma}(0) \perp v$.

  We have
  $$
  D \exp_q(t v)(t w) = Y(t)(w) = J(t),
  $$
  where $Y$ is the Jacobi tensor along $c_v$ with $Y(0) = 0$ and
  $Y'(0) = \id$, and therefore $J$ is a Jacobi field along $c$
  satisfying $J(0) =0$ and $J'(0) =w$. Note that $Y$ and $A_v$ are
  related by $A_v = Y \Big\vert_{(c_v')^\perp}$. In particular, we
  have
  $$
  D \exp_q(tv)(tv) = t D \exp_q(tv)(v) = t c_v'(t).
  $$
  This yields
  \begin{eqnarray*}
    D F_t(v)(\dot{\gamma}(0)) & = & 
    \langle N_q(F_t(v)), D F_t(v) \dot{\gamma}(0) \rangle\; 
    D \exp_q(d(q,F_t(v))  v)(v)\\
    & + & A_v(d(q,F_t(v)))(\dot{\gamma}(0))\\
    & = & \langle N_q(F_t(v)), D F_t(v) \dot{\gamma}(0) \rangle\; 
    c_v' (d(q,F_t(v)))\\
    &  &+ A_v(d(q,F_t(v)))( \dot{\gamma}(0)).
  \end{eqnarray*}
 
  Consequently,
  \begin{multline*}
  D F_t(v) (\dot{\gamma}(0)) = \\
  \langle N_q(F_t(v)), D F_t(v) \dot{\gamma}(0) \rangle N_q(F_t(v)) + 
  A_v(d(q,F_t(v))) (\dot{\gamma}(0)).
  \end{multline*}
  Next, we introduce the map
  \begin{eqnarray*}
  L_x: N_p(x)^\perp &\to& N_q(x)^\perp,\\
  L_x(w) &=& w - \langle w, N_q(x) \rangle N_q(x).
  \end{eqnarray*}
  Then we have
  $$
  L_{F_t(v)} (D F_t(v)(\dot{\gamma}(0))) = A_v (d(q,F_t(v)))(\dot{\gamma}(0)).
  $$
 
  The finish the proof of the above Proposition, we need the following lemma.

  \begin{lemma}
  $\Jac L_x = | \langle N_p(x), N_q(x) \rangle |$.
  \end{lemma}

  \begin{proof}
    Consider
    $$
    N_p(x)^\perp \cap N_q(x)^\perp = \{w \in T_x X \mid\; 
    \langle w,N_p(x) \rangle = 0 \; \; \text{and} \; \; 
    \langle w,N_q(x) \rangle = 0 \}.
    $$
    Then $N_p(x)^\perp \cap N_q(x)^\perp$ has co-dimension one in
    $N_p(x)^\perp$ and $L_x$ is the identity on $N_p(x)^\perp \cap
    N_q(x)^\perp$. Let
    $$
    w_0 = N_q(x) - \langle N_q(x), N_p(x) \rangle N_p(x) \in N_p(x)^\perp.
    $$
    The vector $w_0$ is orthogonal to $N_p(x)^\perp \cap
    N_q(x)^\perp$ since for all $w \in N_p(x)^\perp \cap
    N_q(x)^\perp$ we have $\langle w,N_p(x) \rangle = 0$ and $\langle
    w,N_q(x) \rangle = 0$, and therefore
    $$
    \langle w,w_0 \rangle = \langle \underbrace{w,N_q(x)}_{= 0} \rangle - 
    \langle N_q(x), N_p(x) \rangle \langle 
    \underbrace{w, N_p(x)}_{= 0} \rangle = 0.
    $$
    Moreover, $L_x w_0$ is also orthogonal to $N_p(x)^\perp \cap
    N_q(x)^\perp$:
    \begin{eqnarray*}
      L_x w_0& = &w_0 - \langle w_0, N_q(x) \rangle N_q(x)\\
      &= &N_q(x) - \langle N_q(x), N_p(x) \rangle N_p(x) -
      \langle N_q(x), N_q(x) \rangle N_q(x)\\
      &&+ \langle N_q(x), N_p(x))^2 \rangle N_q(x)\\
      &= &\langle N_p(x), N_q(x) \rangle (\langle N_p(x), N_q(x) \rangle N_q(x) 
      - N_p(x)),
    \end{eqnarray*}
    and consequently $\langle w,L_x w_0 \rangle = 0$ for all $w$
    satisfying $\langle w,N_p(x) \rangle= \langle w,N_q(x) \rangle =
    0$.  Consequently:
    $$
    \Jac L_x = \frac{||L_x w_0||}{||w_0||}.
    $$
    Since
    \begin{eqnarray*}
    \|L_x w_0\|^2& = &\langle N_p(x), N_q(x) \rangle^2 
    (\langle N_p(x),N_q(x) \rangle^2 + 1 - 2 \langle N_p(x),N_q(x) \rangle^2)\\
    &= &\langle N_p(x),N_q(x) \rangle^2 (1 - \langle N_p(x),N_q(x) \rangle^2)
    \end{eqnarray*}
    and
    \begin{eqnarray*}
    \|w_0\|^2& =& 1 + \langle N_p(x),N_q(x) \rangle^2 - 
    2 \langle N_p(x),N_q(x) \rangle^2 \\
    &=& 1 - \langle N_p(x),N_q(x) \rangle^2,
    \end{eqnarray*}
    we obtain
    \begin{eqnarray*}
    \Jac L_x &= &\left( \frac{\langle N_p(x),N_q(x) \rangle^2 
    (1 - \langle N_p(x),N_q(x) \rangle^2)}{1 - \langle N_p(x),N_q(x) \rangle^2} 
    \right)^{1/2}\\
    &= &| \langle N_p(x),N_q(x) \rangle |,
    \end{eqnarray*}
    which yields the lemma.
  \end{proof}

  Finally, $L_{F_t(v)} \circ DF_t(v) = A_v (d(q,F_t(v)))$ implies that
  $$
  \Jac F_t(v) = \frac{\det A_v(d(q,F_t(v)))}{\Jac\; L_{F_t}(v)} =
  \frac{\det A_v (d(q,F_t(v)))}{\langle N_p(F_t(v)),N_q(F_t(v))
    \rangle},
  $$
  finishing the proof of the proposition.
\end{proof}

\begin{figure}[h]
  \begin{center}
    \psfrag{SqX}{$S_qX$} 
    \psfrag{q}{$q$}
    \psfrag{v}{$v$}
    \psfrag{Ft(v)}{$F_t(v)$}
    \psfrag{St(p)}{$S_t(p)$}
    \psfrag{t}{$t$}
    \psfrag{p}{$p$}
    \psfrag{Bt(v)}{$B_t(v)$}
    \includegraphics[width=8cm]{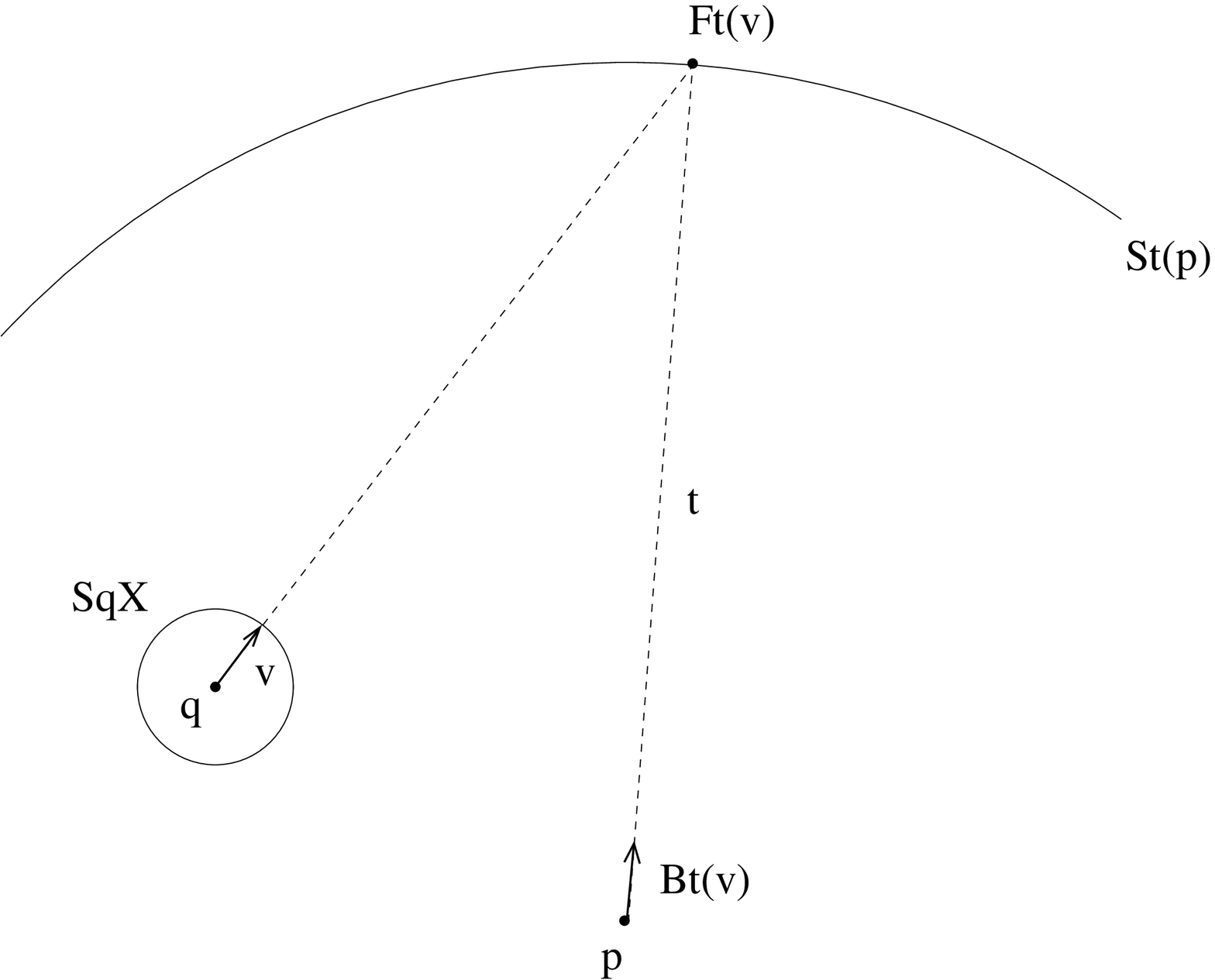}
  \end{center}
  \caption{Illustration of the map $B_t: S_qX \to S_pX$}
  \label{fig_btv}
\end{figure}

\begin{cor} \label{cor:JacBt} Let $(X,g)$ be a noncompact, simply
  connected harmonic space. Let $B_t: S_q X \to S_p X, v \mapsto
  \frac{1}{t} \exp_p^{-1} \circ F_t(v)$ (see Figure
  \ref{fig_btv}). Then we have
  $$ \Jac B_t(v) = \frac{f(d(q,F_t(v))}{f(t)} \cdot \frac{1}{\langle
    N_p(F_t(v)),N_q(F_t(v)) \rangle}.
  $$
\end{cor}

\begin{proof}
  Let $u \in S_p X$. Then $D \exp_p(tu) : u^\perp \to T_{\exp_p(tu)}
  S_t(p)$ is given by $D \exp_p(tu) (w ) = \frac{1}{t} A_u(t)(w)$, and
  therefore with $u = B_t(v)$,
  \begin{eqnarray*}
  \Jac B_t(v) &=& \frac{1}{\det A_u(t)} \cdot \Jac F_t(v)\\ 
  &=& \frac{\det A_v(d(q,F_t(v)))}{\det A_u(t)} \cdot 
  \frac{1}{\langle N_p(F_t(v)),N_q(F_t(v)) \rangle}.
  \end{eqnarray*}
  Since $X$ is harmonic, we have $\det A_v(s) = f(s)$, which finishes
  the proof of the corollary.
\end{proof}

Let $f \in C(\partial_B X)$. We know from Lemma \ref{lem:bijvismap}
that $B_t: S_qX \to S_pX$ is a bijection, for $t > 0$ large
enough. Then we have with $f_1 = f \circ \varphi_p^{p_0}$:
\begin{eqnarray*}
  \int_{\partial_B X} f(\xi)\, d\mu_p(\xi) &=& 
  \frac{1}{\omega_n} \int_{S_pX} f_1(w)\; d\theta_p(w) \\
  &=& \frac{1}{\omega_n} \int_{S_qX} (f_1 \circ B_t)(v) (\Jac B_t)(v)\;
  d\theta_q(v).
\end{eqnarray*}
We will show that 
\begin{itemize}
\item[(i)] $\lim_{t \to \infty} B_t = (\varphi_p^{p_0})^{-1} \circ
  (\varphi_q^{p_0})$,
\item[(ii)] There exist constants $t_0 > 0$ and $C > 0$ such that
$$| \Jac B_t(v) | \le C \quad \forall\, v \in S_qX, \, t \ge t_0. $$  
\item[(iii)] We have, for all $v \in S_qX$,
$$ \lim_{t \to \infty} \Jac B_t(v) = e^{-h b_v(p)}. $$
\end{itemize}
Having these facts, we conclude with Lebesgue's dominated convergence that
\begin{eqnarray*}
  \int_{\partial_B X} f(\xi)\, d\mu_p(\xi) &=& 
  \lim_{t \to \infty} \frac{1}{\omega_n} \int_{S_qX} (f_1 \circ B_t)(v) 
  (\Jac B_t)(v)\; d\theta_q(v) \\
  &=& \frac{1}{\omega_n} \int_{S_qX} (f \circ \varphi_q^{p_0}) e^{-h b_v(p)}
  \; d\theta_q(v) \\
  &=& \int_{\partial_B X} f(\xi) e^{-h b_{q,\xi}(p)}\, d\mu_q(\xi),  
\end{eqnarray*} 
with $b_{q,\xi} = \xi - \xi(q)$ for $q \in X$ and $\xi \in \partial_B
X$. This shows the following fact:

\begin{theorem}
  Let $(X,g)$ be a simply connected noncompact harmonic manifold with
  reference point $p_0 \in X$. Let $(\mu_p)_{p \in X}$ be the
  associated family of visibility measures. Then these measures are
  pairwise absolutely continuous and we have
  $$ \frac{d\mu_p}{d\mu_q}(\xi) = e^{-h b_{q,\xi}(p)}. $$
\end{theorem}

It remains to prove (i), (ii) and (iii) above.

\bigskip

{\bf Proof of (i):} Let $t_n \to \infty$ and $s_n \ge 0$, $w_n = B_{t_n}(v)
\in S_pX$ such that $y_n = \exp_q(s_n v) = \exp_p(t_n w_n)$. We
obviously have $s_n \to \infty$ and $y_n \to b_v - b_v(q)$. Let
$w_{n_j}$ be a convergent subsequence of $w_n = B_{t_n}(v)$ with limit
$w \in S_pX$. Then we have $y_{n_j} \to b_w - b_w(p)$, by
Proposition \ref{prop:busboundprops}(3), and
$$ \varphi_p^{p_0}(v) = b_v - b_v(q) = b_w - b_w(p) = \varphi_q^{p_0}(w). $$
This shows that $\lim_{n \to \infty} B_{t_n}(v) = (\varphi_q^{p_0})^{-1} \circ
\varphi_p^{p_0}(v)$. \qed

\bigskip

For the proof of (ii), we need the following lemma:

\begin{lemma} \label{lem:NpNq}
  For every $\epsilon > 0$, there exists $t_0 > 0$ such that we have
  for all $v \in S_qX$
  $$
  |\langle N_p(F_t(v)),N_q(F_t(v)) \rangle - 1| < \epsilon \quad
  \forall\; t \ge t_0.
  $$
\end{lemma}

\begin{proof}
  This is an easy consequence of corollary \ref{cor:uniformdiv}.
\end{proof}

\bigskip

{\bf Proof of (ii):} Since $f(t) > 0$ is an exponential polynomial,
there exists $m > 0$ such that 
\begin{equation} \label{eq:flim}
\frac{f(t)}{t^m e^{ht}} \to 1\quad \text{as}\; t \to \infty.
\end{equation}
Therefore, there exists $t_0 > 0$ such that
$$ \frac{1}{2} t^m e^{ht} \le f(t) \le \frac{3}{2} t^m e^{ht} $$
for all $t \ge t_0$. Using Lemma \ref{lem:NpNq} and increasing $t_0 >
0$ if necesary, we can also assume that
$$ \langle N_p(F_t(v)),N_q(F_t(v)) \rangle \ge \frac{1}{2} $$
for all $t \ge t_0$. Since $d(q,F_t(v)) \le t + d(p,q)$, we conclude
from Corollary \ref{cor:JacBt} for all $t \ge t_0$ and all $v \in
S_qX$,
$$ |\Jac B_t(v)| \le 6 \frac{(t+d(p,q))^me^{h(t+d(p,q))}}{t^me^{ht}} \le 
6 \left( 1 + \frac{d(p,q)}{t_0} \right)^m e^{h d(p,q)}. \qed $$ 

\bigskip

{\bf Proof of (iii):} This is a immediate consequence of Lemma
\ref{lem:NpNq} and the following Lemma:

\begin{lemma}
  Using the notation above we have that
  $$
  \lim_{t \to \infty} \frac{f(d(q,F_t(v)))}{f(t)} = e^{-hb_v(p)}
  $$
\end{lemma}

\begin{proof}
  Note that $B^q_{F_tv}(x) = d(x, F_t(v)) - d(q,F_t(v))$ converges
  uniformly to $b_v$. Hence, for all $\epsilon > 0$ there exists $t_0
  > 0$ such that for all $t \geq t_0$ we have
  $$
  |b_v(p) - \underbrace{d(p,F_t(v))}_{= t} + d(q,F_t(v))| < \epsilon.
  $$
  This implies
  $$
  - \epsilon + t - b_v(p) \leq d(q,F_t(v)) \leq \epsilon + t - b_v(p).
  $$
  Note that $f(r)$ is monotone, since $\frac{f'}{f}(r) \ge h \geq
  0$. Therefore,
  $$
  f(- \epsilon + t - b_v(p)) \leq f(d(q,F_t(v))) \leq f(\epsilon + t - b_v(p))
  $$
  and
  $$
  \frac{f(- \epsilon + t - b_v(p))}{f(t)} \leq
  \frac{f(d(q,F_t(v)))}{f(t)} \leq \frac{f(\epsilon + t
    -b_v(p))}{f(t)}
  $$
  Using \eqref{eq:flim}, we obtain for $a \in \mathbb{R}$.
  \begin{eqnarray*}
    \lim\limits_{t \to \infty} \frac{f(t-a)}{f(t)} &=& \lim\limits_{t \to
      \infty} \frac{f(t-a)}{(t-a)^m e^{h(t-a)}} \cdot \frac{(t-a)^m
      e^{h(t-a)}}{t^m e^{ht}} \cdot \frac{t^{m} e^{ht}}{f(t)}\\
    &=& \lim\limits_{t \to \infty} (\frac{t-a}{t})^m e^{-ha} = e^{-ha}.
  \end{eqnarray*}
  Hence, for all $ \epsilon > 0$, we have
  $$
  e^{-h(b_v(p) + \epsilon)} \leq \liminf\limits_{t \to \infty}
  \frac{f(d(q,F_t(v)))}{f(t)} \leq \limsup\limits_{t \to \infty}
  \frac{f(d(q,F_t(v)))}{f(t)} \leq e^{-h(b_v(p)-\epsilon)}.
  $$
  This implies the claim.
\end{proof}

\section{The Green's kernel and the Martin boundary}
\label{Greenskernel}

Let $(X,g)$ be a simply connected noncompact harmonic
manifold. In this chapter we calculate explicitly the Green's kernel,
which is the smallest non-negative fundamental solution of the Laplace
equation on $X$, i.e., $G$ is a function defined on $X \times X
\backslash \{(x,x) \mid x \in X\}$, and having the following
properties:
\begin{enumerate}
\item[(a)] $\Delta_x G(x,y) = 0 \quad \forall\; x \neq y$,\\[-.2cm]
\item[(b)] $G(x,y) \ge 0$ for all $x \neq y$,\\[-.2cm]
\item[(c)] For all $y \in X$, we have 
$\inf_{x \in X, x \neq y} G(x,y) = 0$,\\[-.2cm]
\item[(d)] $\int\limits_X G(x,y) \Delta \varphi(y) dy = -\varphi(x)
  \quad \forall\; \varphi \in C_0^\infty(X)$.
\end{enumerate}
$G$ can also be expressed with the help of the smallest positive
fundamental solution of the heat equation $p_t(x,y)$ as
$$ G(x,y) = \int_0^\infty p_t(x,y)dt. $$
This implies that $G(x,y)=G(y,x)$. For $\dim X \ge 3$, the Green's
kernel has a singularity at the diagonal with the asymptotic
$$ G(x,y) = \frac{c_n}{d(x,y)^{n-2}}(1+o(1)) \quad \text{as $x \to y$}, $$
with a fixed constant $c_n > 0$ depending on the dimension.
 
Since the heat kernel $p_t(x,y)$ of a harmonic space $X$ only depends
on $d(x,y)$ (see \cite[Thm 1.1]{Sz}), the same holds true for the
Green's kernel, i.e., there exists a function $\widetilde G$ such that
$G(x,y) = \widetilde G(d(x,y))$.

Property (a) means that for $r=d(y,x)=d_y(x)$ we have
$$
0 = \Delta_x(\widetilde{G} \circ d_y)(x) = \widetilde{G}''(r) +
\frac{f'}{f}(r) \widetilde{G}'(r),
$$
i.e.
$$
(f \widetilde{G}'')(r) + (f' \widetilde{G}')(r) = 0,
$$
i.e., $(f \widetilde{G}')'(r) = 0$. Integration yields $f \tilde{G}' +
\beta =0$. We choose $\beta = \frac{1}{\omega_n} = \frac{1}{\vol
  (S_yX)}$ and obtain $\widetilde{G}'(r) = -\frac{\beta}{f(r)}$.

This leads us to consider $\widetilde{G}(r) = \beta
\int\limits_r^\infty \frac{dt}{f(t)}$. The above derivations show that
$$
G: (X \times X) \backslash \{(x,x) \mid x \in X\} \to {\R},
$$ 
defined by $G(x,y) = \widetilde{G}(d(x,y))$, satisfies $0 = \Delta_x
G(x,y)$ for all $x \neq y$, which shows (a) for this choice of $G$.

\smallskip

The properties (b) und (c) for this choice of $G(x,y)$ are easily
verified.

\smallskip

The following calculation shows (d). We have for all $\varphi \in
C_0^\infty(X)$:
\begin{multline*}
  \langle G(x, \cdot),\Delta \varphi \rangle = \\
  = \int\limits_X G(x,y) \Delta \varphi(y) dy = \int\limits_0^ \infty
  \int\limits_{S_xX} f(r) \widetilde{G}(r)
  (\Delta \varphi)(c_v(r)) d\theta_x(v)\; dr\\
  = \int\limits_0^\infty f(r) \widetilde{G}(r) \int\limits_{S_xX}
  (\varphi \circ c_v)''(r) + \frac{f'}{f}(r) (\varphi \circ c_v)'(r)
  d\theta_x(v)\;
  dr\\
  = \beta \int\limits_0^\infty f(r) \int\limits_r^\infty
  \frac{1}{f(t)} \int\limits_{S_xX} (\varphi \circ c_v)''(r) +
  \frac{f'}{f}(r) (\varphi \circ c_v)'(r)d\theta_x(v)\; dt\; dr\\
  = \beta \int\limits_{S_xX} \int\limits_0^\infty \frac{1}{f(t)}
  \int\limits_0^t \underbrace{f(r)(\varphi \circ c_v)''(r) +
    f'(r)(\varphi \circ c_v)'(r)}_{(f \cdot(\varphi \circ c_v)')'(r)}
  dr\; dt\; d\theta_x(v)\\
  = \beta \int\limits_{S_xX} \int\limits_0^\infty \frac{1}{f(t)} f(t)
  (\varphi \circ c_v)'(t) dt\; d\theta_x(v)\\
  = - \beta \int\limits_{S_xX} (\varphi \circ c_v)(0) d\theta_x(v) =
  - \varphi(x) \cdot \beta \cdot \vol(S_xX) = - \varphi(x).
\end{multline*}

We conclude that the Green's kernel of a noncompact harmonic manifold
$(X,g)$ has the form

\vspace*{.5cm}

\begin{center}
  \fbox{\parbox{12cm} {\[G(x,y) = \widetilde G(d(x,y)) =
      \frac{1}{\omega_n} \int\limits_{d(x,y)}^\infty
      \frac{dt}{f(t)}.\]}}
\end{center}

\vspace*{.5cm}


Let $p_0 \in X$ be a fixed reference point. We define
$$
\Sigma_{p_0} : = \{\sigma = (x_m) \subset X \mid\; d(x_m,p_0) \to
\infty, \, \frac{G(x,x_m)}{G(p_0,x_m)} \to K_\sigma(x)\},
$$
where the convergence $\frac{G(x,x_m)}{G(p_0,x_m)} \to K_\sigma(x)$ is
meant uniformly on all compact subsets of $X$. Let $\sigma = (x_m),
\sigma' = (x_m') \in \Sigma$. We call $\sigma \sim \sigma'$ if and
only if $K_\sigma = K_\sigma'$.

\begin{dfn}
  The {\em Martin boundary} $\partial_\Delta^{p_0} X$ is defined as
  $$
  \partial_\Delta^{p_0} X = \Sigma_{p_0}/_\sim
  $$
  $\partial_\Delta^{p_0} X$ carries a metric, defined by
  $$
  d_\Delta^{p_0}(\sigma,\sigma') : = - \sup\limits_{x \in B(p_0,1)}
  |K_\sigma(x) - K_{\sigma'}(x)|.
  $$
\end{dfn}

\begin{example}(Martin Boundary of Euclidean space)
  Assume that $X = \R^n$ with $n \ge 3$, i.e., $X$ is the flat
  Euclidean space. Then the Green's kernel is given by
  $$ G(x,y) = \frac{1}{(n-2)\; \omega_n\; |x-y|^{n-2}}. $$
  Let $p_0 = 0$. We have 
  $$ \frac{G(x,x_m)}{G(p_0,x_m)} = 
  \left( \frac{|x_m|}{|x-x_m|} \right)^{n-2}. $$
  Now, $d(x_m,p_0) \to \infty$ means that $|x_m| \to \infty$. Since
  $$ |x_m| - |x| \le |x-x_m| \le |x_m| + |x|, $$
  we conclude that, for $|x_m| \to \infty$, 
  $$ \lim \frac{G(x,x_m)}{G(p_0,x_m)} = 1, $$
  i.e., $\partial_\Delta^0 \R^n$ consists of a single point, the constant
  harmonic function $1$. 

  Since $\partial_B^0 \R^n$ is a topological sphere of dimension
  $n-1$, we have $\partial_B^0 \R^n \neq \partial_\Delta^0 \R^n$.
\end{example}

However, Busemann boundary and Martin boundary agree for non-flat noncompact
harmonic spaces.

\begin{theorem} \label{thm:BMboundiso}
  Let $(X,g)$ be a noncompact harmonic manifold with $h > 0$. Then we
  have
  $$
  \partial_B^{p_0} X = \partial_\Delta^{p_0} X
  $$
  as topological spaces.
\end{theorem}

\begin{proof}
  We first show $\partial_{p_0} X = \partial_\Delta X$ as sets.\\

\item[(a)] Our first goal is: If $\xi \in \partial_B^{p_0} X$, and $(x_n)$ is a
  sequence in $X$ such that $x_n \to \xi$ in the Busemann topology,
  then we have the following uniform convergence on compacta: 
  \begin{equation} \label{eq:Greenquot}
  \frac{G(x,x_n)}{G(p_0,x_n)} \to e^{-h\xi(x)}.
  \end{equation}
  This implies that $(x_n)$ is also convergent to a point in the
  Martin boundary, and we obtain a canonical injective map
  $$
  \partial_B^{p_0}(X) \to \partial_\Delta^{p_0}(X),
  $$
  $\xi \mapsto (x_n)$, where $(x_n) $ is any sequence with $x_n \to
  \xi$ in
  the Busemann topology.\\

  Proof of \eqref{eq:Greenquot}: 
  We first prove: Let $I$ be a compact interval. Then
  $$
   \frac{\int_{s+a}^\infty \frac{dt}{f(t)}}{\int_s^\infty \frac{dt}{f(t)}} \to e^{-ha} 
$$
uniformly  on $I$ as $s \to \infty$. To prove this consider
$$
\psi_a(u) = \int_{\frac{1}{u}+a}^\infty \frac{dt}{f(t)}.
 $$ 
 Then $\lim_{u \to 0} \psi_a(u) = 0$ and by the mean value theorem we
 obtain
 $$
(\ast) \quad \frac{\psi_a(u) }{\psi_0(u) } = \frac{\psi_a(0)- \psi_a(u)}{\psi_0(u) - \psi_0(u) } = \frac{\psi_a'(x)}{\psi_0'(x)}
 $$
 for some $x \in (0, u)$.
 Since 
 $$
 \psi_a'(x) = \frac{1}{f(\frac{1}{x}+a) x^2}
 $$
 we obtain
 $$
 \left| \frac{\psi_a(u) }{\psi_0(u) } - e^{-ha} \right| =  \left| \frac{f( \frac{1}{x})}{f( \frac{1}{x} +a)} - e^{-ha} \right|
 $$
 for some $x \in (0, u)$. Using  \ref{eq:flim} we obtain that  $ \frac{\psi_a(u) }{\psi_0(u)} $
 converges on compact intervals uniformly to $e^{-ha}$ as $u \to 0$ which is equivalent to the assertion above.\\
  Let $h_y(x) =
  \frac{G(x,y)}{G(p_0,y)}$. Since $G(x,y) = \frac{1}{\omega_n}
  \int\limits_{d(x,y)}^\infty \frac{dt}{f(t)}$, we have
  $$
  h_y(x) = \frac{\int_{d(x,y)}^\infty \frac{dt}{f(t)}}
  {\int_{d(p_0,y)}^\infty \frac{dt}{f(t)}}.
  $$
  Let $x_n \to \xi \in \partial_B^{p_0} X$ in the Busemann topology. By
  definition this implies that, for each compact subset $K \subset X$
  and all $\epsilon >0$, there exists $n_0(\epsilon,K) > 0$, such that for
  all $n \geq n_0(\epsilon,K)$ we have
  $$
  |d(x,x_n) - d(p_0,x_n) - \xi(x)| \leq \epsilon
  $$
  for all $n \geq n_0(\epsilon,K)$ and $x \in K$. In particular ,we
  have for all $n \geq n_0(\epsilon,K)$ and $x \in K$
  $$
  d(p_0,x_n) + \xi(x) - \epsilon \leq d(x,x_n) \leq d(p_0,x_n) + \xi +
  \epsilon.
  $$

  Since $\frac{1}{f(t)} > 0$, we conclude that
  $$
  \frac{\int_{d(p_0,x_n)+\xi(x)- \epsilon}^\infty
    \frac{dt}{f(t)}}{\int_{d(p_0,x_n)}^\infty \frac{dt}{f(t)}} \leq
  h_{x_n}(x) \leq \frac{\int_{d(p_0,x_n)+ \xi(x)+ \epsilon}^\infty
    \frac{dt}{f(t)}}{\int_{d(p_0,x_n)}^\infty \frac{dt}{f(t)}}
  $$
  Since $\xi$ is bounded on $K$ we can choose for each $\epsilon  >0$ (using ($ \ast$)) a number $n_1(\epsilon, K) $, such that for all $n \ge n_1(\epsilon,K)$,
  $$
e^{-h (\xi(x) + \epsilon)} 
   \leq h_{x_n}(x) \leq e^{-h (\xi(x) - \epsilon)}. 
 $$ 
  This implies, that
  $$
  \lim\limits_{n \to \infty} h_{x_n}(x) = e^{-h\xi(x)},
  $$
  uniformly on each compact set.
  Therefore,
  $$ 
  \lim\limits_{n \to \infty} \frac{G(x_n,x)}{G(x_n,p_0)} =
  e^{-h\xi(x)}
  $$
  converges uniformly on compact sets and
  $(x_n)$ converges to a point in the Martin boundary.\\

\item[(b)] Let $\sigma=(x_n)$ be convergent to a point in the Martin
  boundary, i.e. $d(x_n,p_0) \to \infty$ and
  $\frac{G(x_n,x)}{G(x_n,p_0)} \to K_\sigma$, uniformly on
  compacta. Since $x_n \in X \subset \overline{B^{p_0}(X)}$ and
  $\overline{B^{p_0}(X)}$ is compact, there exists a subsequence
  $x_{n_j} \in X$ such that $x_{n_j} \to \xi \in
  \overline{B^{p_0}(X)}$, in the Busemann topology. Since $d(x_n,p_0)
  \to \infty$, $\xi \in \partial_B^{p_0} X$ we obtain that
  $$
  \frac{G(x_{n_j},x)}{G(x_{n_j},p_0)} \to e^{-h\xi(x)}.
  $$
  Therefore, $ K_\sigma = e^{-h\xi}$. Let $(x_{n_j}')$ be an arbitary
  subsequence of $(x_n)$ which is convergent in the Busemann topology
  to a point $\xi' \in \partial_B^{p_0} X$. Then, from
  $\frac{G(x_n,x)}{G(x_n,x_0)} \to K_\sigma$, we conclude that
  $$
  e^{-h \xi} = e^{-h \xi'},
  $$
  i.e., $\xi_0 = \xi'$. Hence $(x_n)$ is convergent in the Busemann
  topology. This shows that the above map $\partial_B^{p_0}(X)
  \to \partial_\Delta^{p_0}(X)$ is also surjective.\\

\item[(c)] Assume $\xi_n \in \partial_B^{p_0}(X)$ converges to $\xi
  \in \partial_B^{p_0}(X)$, i.e., we have uniform convergence of $ \xi_n
  \to \xi$, on each compact subset $K \subset X$. In particular,
  $$
  \sup\limits_{x \in B_1(x_0)} | \xi_n(x) - \xi(x)| \to 0,
  $$
  and, consequently, $\sup\limits_{x \in B_1(x_0)} |e^{-h \xi_n(x)} -
  e^{-h \xi (x) } |\to 0$. This means that a convergent sequence
  $\xi_n$ in the Busemann topology is also convergent in the
  Martin boundary topology.\\

  Therefore, we obtain that the map $\partial_B^{p_0}(X)
  \to \partial_\Delta^{p_0}(X)$ with $\xi \to (x_n)$ where $(x_n)$ is
  a sequence converging to $\xi$, is a continuous bijection.  Since
  $\partial_B^{p_0}(X)$ is compact, this map is a homeomorphism.
\end{proof}

\section{Representation of bounded harmonic functions}
\label{chp:repbdharmfunc} 

Let $(X,g)$ be a simply connected noncompact harmonic
manifold. One of the merits of the Martin boundary theory, which we like to
recall, is the representation of bounded harmonic functions (see
e.g. \cite{Wo} for more details).

For $ p_0 \in X$ define the set
$$
H_{p_0}(X) = \{ F \in C^2(X) \mid \Delta F = 0, F(p_0) = 1 , F \ge 0
\}
$$
of normalized positive harmonic functions on $X$. This set is convex
and compact with respect to the topology of uniform convergence on
compact subsets. We call an element $F \in H_{p_0}(X)$ {\em minimal},
if for any two function $F_1, F_2 \in H_{p_0}(X)$ with
$$
F = \lambda F_1 + (1- \lambda) F_2
$$
we have $F_1 = F_2$. Note, that the minimal set is contained in the
Martin boundary. Therefore,
$$
\partial_{\Delta, \min}^{p_0}(X) := \{ K_\sigma \mid \sigma
\in \partial_\Delta^{p_0}(X), K_\sigma \; \; \text{is minimal} \}
$$
is the set of minimal elements in $ H_{p_0}(X)$.  Using Choquet theory
we obtain for each $F \in H_{p_0}(X) $ a unique probability measure
$\nu_F$ on $ \partial_{\Delta, \min}^{p_0}(X)$ with
$$
F(x) = \int\limits_{ \partial_{\Delta, \min}^{p_0}(X)} K_\sigma(x) d\nu_F( \sigma).
$$
If we do not insists in probability measures, we can represent in this
way positive harmonic functions which are not normalized as well. If
$F, G$ are positive harmonic functions with $F \le G$ then $G-F$ is a
positive harmonic function and
$$
\mu_G = \nu_F + \nu_{G-F} \le \nu_F.
$$
This implies that $\nu_F $ is absolutely continuous to $\nu_G$.  Of
fundamental importance is the {\em harmonic measure} $\nu_{p_0} :=\nu_1$,
i.e the unique probablity on $ \partial_{\Delta, \min}^{p_0}(X)$
defined by
$$
1 =  \int\limits_{  \partial_{\Delta, \min}^{p_0}(X)} K_\sigma(x)  d\nu_{p_0}( \sigma).
$$
Using this measure, one can represent all bounded harmonic functions on
$X$ in the following way. Consider first a harmonic function $F$ with
$0 \le F \le M$ for some $M \ge 0$. This implies $ \nu_F \le M
\nu_{p_0}$. In particular the Radon-Nikodym derivative
 $$
 \varphi_F :=\frac{d \nu_F}{d \nu_{p_0}}
 $$
 exists and defines an element in $L^1( \partial_{\Delta,
   \min}^{p_0}(X))$ with $ 0 \le \varphi_F \le M$.  Hence,
 $$
 F(x) = \int\limits_{ \partial_{\Delta, \min}^{p_0}(X)} K_\sigma(x)
 d\nu_F( \sigma)= \int\limits_{ \partial_{\Delta, \min}^{p_0}(X)}
 K_\sigma(x) \varphi_F( \sigma)d\nu_{p_0}( \sigma).
 $$
 
 Now, consider an arbitrary bounded harmonic function $F$ and assume
 that $-M \le F \le M$ for some positive number $M$. Then
 $$
 F(x) +M = \int\limits_{ \partial_{\Delta, \min}^{p_0}(X)} K_\sigma(x)
 \varphi_{F+M}( \sigma)d\nu_{p_0}( \sigma),
 $$
 where $0 \le \varphi_{F+M} \le 2M$. Hence,
  \begin{eqnarray*}
    F(x) &=&  
    \int\limits_{ \partial_{\Delta, \min}^{p_0}(X)} K_\sigma(x)  
    \varphi_{F+M}( \sigma)d\nu_{p_0}( \sigma)-
    \int\limits_{  \partial_{\Delta, \min}^{p_0}(X)} K_\sigma(x)M  d\nu_{p_0}( \sigma)\\
    &=& \int\limits_{ \partial_{\Delta, \min}^{p_0}(X)} K_\sigma(x)  
    (\varphi_{F+M}( \sigma) -M)d\nu_{p_0}( \sigma)\\
    &=&\int\limits_{ \partial_{\Delta, \min}^{p_0}(X)} K_\sigma(x) 
    \varphi_{F}( \sigma)d\nu_{p_0}( \sigma),
  \end{eqnarray*}
  where $-M \le \varphi_{F}:=\varphi_{F+M} -M \le M$. In particular,
  the following theorem holds (see also \cite[Theorem (24.12)]{Wo}):

  \begin{theorem} \label{thm:harmoncrep} Denote by
    $L^\infty( \partial_{\Delta, \min}^{p_0}(X))$ the set of bounded
    $L^1$-functions on $\partial_{\Delta, \min}^{p_0}(X)$, and by
    $H^\infty(X)$ the set of bounded harmonic functions on $X$. Then
    the map $H: L^\infty( \partial_{\Delta, \min}^{p_0}(X)) \to
    H^\infty(X)$ with
    $$
    H( \varphi)(x):= \int\limits_{ \partial_{\Delta, \min}^{p_0}(X)}
    K_\sigma(x) \varphi( \sigma)d\nu_{p_0}( \sigma)
    $$
    defines a linear isomorphism.
  \end{theorem}

  Now let $(X,g)$ be a harmonic manifold with $h >0$. Then as we have
  shown above (see Theorem \ref{thm:BMboundiso}), the Martin boundary
  $\partial_{\Delta}^{p_0}(X))$ and the Busemann boundary
  $\partial_{B}^{p_0}(X)$ are isomorphic and, using the identification
  $\sigma \to \xi$, we have $K_\sigma = e^{-h \xi}$.  As has been
  observed by Zimmer \cite{Zi3}, all functions $e^{-h \xi}$ are minimal
  in $ H_{p_0}(X)$.  Let $\mu_{p_0}$ be the visibility measure with
  respect to $p_0$ introduced above. We have that
  $$
  1 = \int\limits_{\partial_{B}^{p_0}(X))} e^{-h \xi} d\mu_{p_0}
  $$
  which by the discussion above implies that $\mu_{p_0}$ is the
  harmonic measure on $\partial_{B}^{p_0}(X))$.

  Related to the Martin representation is the {\em Dirichlet problem
    at infinity}, which deals with the following question.  Let
  $\varphi: \partial_B^{p_0}(X) \to {\R}$ be a continuous function.
  Is there a harmonic function $F$ on $X$ such that
  $$
  \lim_{x \to \xi} F(x) = \varphi(\xi)?
  $$
  By the maximal principle, $F$ is unique if it exists. A natural
  candidate for the solution is the function $H_\varphi$ defined by
  the integral presentation:
  $$
  H_\varphi(x) = \int\limits_{  \partial_B^{p_0}(X) } \varphi(\xi) d \mu_x(\xi) =
  \int\limits_{ \partial_B^{p_0}(X)} \varphi(\xi) e^{-h \xi(x)}
    d \mu_{p_0}(\xi).
  $$
  Obviously, $H_\varphi$ is the solution of the Dirichlet problem if
  and only if
   $$
   \lim_{x \to \xi} \mu_x = \delta_\xi,
   $$
   where $ \delta_\xi$ is the Dirac measure at $\xi
   \in \partial_\Delta^{p_0}(X)$ and the limit is taken in the weak
   topology of Borel probability measure on
   $\partial_\Delta^{p_0}(X)$. We will see in Chapter
   \ref{chp:dirichprob} that this holds if $(X,g)$ has purely
   exponential volume growth.

\newpage

\part{Noncompact harmonic manifolds with 
purely exponential volume growth}

From now on, all harmonic manifolds $(X,g)$ under consideration are
assumed to be noncompact, connected and simply connected, and of purely
exponential volume growth. As shown in the paper \cite{Kn3} by the
first author, purely exponential volume growth, geometric rank $1$,
Gromov hyperbolicity and the Anosov-property of the geodesic flow
(with respect to the Sasaki-metric) are equivalent for harmonic spaces
$(X,g)$. Furthermore, nonpositive curvature or more generally no focal
points implies any of the above conditions.

This part covers the following new results for harmonic manifolds with
purely exponential volume growth:
\begin{enumerate}
\item Agreement of the geometric boundary and the Busemann boundary as
  topological spaces.
\item Solution of the Dirichlet Problem at infinity, and an explicit
  integral presentation using the visibility measures at infinity. The
  solution of the Dirichlet Problem follows also from the general
  theory of \cite{Anc1,Anc2} for Gromov hyperbolic spaces (which is
  related to earlier work for spaces of negative curvature by
  \cite{AnSch}, see also \cite[Chapter II]{SchY}). However, in the
  case of harmonic spaces, these results can be deduced in much more
  direct and geometric way.
\item Polynomial volume growth of all horospheres. As an application,
  we also pove a mean value property of bounded harmonic functions at
  infinity. This latter result follows from a modification of the
  arguments given in \cite{CaSam} for negatively curved asymptotically
  harmonic spaces.
\end{enumerate}

\section{Gromov hyperbolicity}

We start this chapter by introducing the Gromov product. 

\begin{dfn}
  Let $(X,d)$ be a metric space and $x_0 \in X$ a reference point. The
  {\em Gromov product} $(x,y)_{x_0}$ of $x,y \in X$ is defined as
  $$
  (x,y)_{x_0} = \frac{1}{2} (d(x,x_0) + d(y,x_0) - d(x,y))
  $$
\end{dfn}

Note that the Gromov product $(x,y)_{x_0}$ is non-negative, by the
triangle inequality. A metric space $(X,d)$ is called a {\em geodesic
  space}, if any two points $x,y \in X$ can be connected by a
geodesic, i.e., if there exists a curve $\sigma_{xy}: [0,d(x,y)] \to
X$ connecting $x$ and $y$, such that $d(\sigma_{xy}(s),\sigma_{xy}(t))
= |t-s|$ for all $s,t \in [0,d(x,y)]$.

\begin{lemma} \label{lem:gromprod}
  Let $(X,d)$ be a geodesic space. Then we have, for $x,y,x_0 \in X$, 
  $$ (x,y)_{x_0} \leq d(x_0,\sigma_{xy}). $$
\end{lemma}

\begin{proof}
 
 Consider $x' \in \sigma_{xy}$ such that
 $ d(x_0, \sigma_{xy}) = d(x_0, x')$. Then
 \begin{eqnarray*}
      (x,y)_{x_0} &=& \frac{1}{2} (d(x,x_0) + d(y,x_0) - d(x,y)) \\
&=& \frac{1}{2} (d(x,x_0) + d(y,x_0) - d(x,x') - d(x',y))\\
 & \leq&  \frac{1}{2} ((d(x',x_0) + d(x',x_0)) = d(x_0, \sigma_{xy})
 \end{eqnarray*}
 \end{proof}

\begin{dfn}
  A geodesic space $(X,d)$ is called {\em $\delta$-hyperbolic} if
  every geodesic triangle $\Delta$ is $\delta$-thin, i.e., every side
  of $\Delta$ is contained in the union of the $\delta$-neighborhoods
  of the other two sides. If a geodesic space $(X,d)$ is
  $\delta$-hyperbolic for some $\delta \ge 0$, we call $(X,d)$ a
  Gromov hyperbolic space.
\end{dfn}

Let us recall the following two general results for Gromov hyperbolic
spaces.  Note that one of the inequalities in Proposition
\ref{prop:gromprod2} was stated in Lemma \ref{lem:gromprod}.

\begin{prop}(see \cite[Prop. 1.3.6]{CDP}) \label{prop:gromprod1}
  Let $(X,d)$ be a $\delta$-hyperbolic space. Then we have for all
  $x_0,x,y,z \in X$:
  $$
  (x,y)_{x_0} \geq \min \{(x,z)_{x_0}, (y,z)_{x_0}\} - 8 \delta.
  $$
\end{prop}

\begin{prop}(see \cite[Prop. 3.2.7]{CDP}) \label{prop:gromprod2}
  Let $(X,d)$ be a $\delta$-hyperbolic space. Then we have
  $$
  (x,y)_{x_0} \leq d(x_0,\sigma_{xy}) \leq (x,y)_{x_0} + 32 \delta.
  $$
\end{prop}

Now assume that $X$ is a harmonic manifold. We have seen in proposition
\ref{prop:busboundprops} that the Busemann boundary $\partial_B^{p_0} X$ 
of $X$ with
    respect to $p_0 \in X$ can be identified with  $S_pX$. The map
    $\varphi_p^{p_0}: S_pX \to \partial_B^{p_0} X$ with $\varphi_p^{p_0}(v):= b_v -
    b_v(p_0)$ is a homeomorphism. A sequence $x_n = exp_p(t_n v_n) \in X$ with $v_n \in S_pX$
    and $t_n \ge 0$ converges to a point $\xi \in \partial_B^{p_0} X$ if
    and only if $t_n \to \infty$ and there exists $v \in S_pX$ with
    $v_n \to v$. In particular, $\xi$ is given by $ b_{v} -
    b_{v}(p_0)$. Hence, a sequence $x_n$ converges in the Busemann topology to infinity if and only if
     $d(x_n,p) \to \infty$ and $\angle_p(x_n,x_m) \to 0$ for
    $n,m \to \infty$. \\
   Assuming additionally that  $X$ has purely exponential volume growth and therefore is Gromov hyperbolic
   we show that $x_n$ converges to infinity if and only if
    $$ \lim_{n,m \to \infty} (x_n,x_m)_{p} = \infty. $$
    
 We note that this is used for general Gromov hyperbolic manifolds as a definition for convergence to infinity.    (see 
\cite[Section 2.2]{BS}).

The following result is the main result of
this chapter.

\begin{theorem} \label{thm:gromconv}
  Let $X$ be harmonic manifold with purely exponential volume growth,
  $p \in X$ and $\{x_n\}$ be a sequence in $X$. The following are
  equivalent.
  \begin{itemize}
  \item[(a)]
  The sequence $\{x_n\}$ converges in the Busemann topology to infinity, i.e.
   $d(x_n,p) \to \infty$ and $\angle_p(x_n,x_m) \to 0$ for
    $n,m \to \infty$. 
      \item[(b)] $(x_n,x_m)_p \to \infty$ for $n,m \to \infty$. 
  \end{itemize}
\end{theorem}

\begin{proof}
  $(b) \Rightarrow (a)$: It was shown in \cite{Kn3} that $X$ is
  $\delta$-hyperbolic for some $\delta \ge 0$. Let $(x_n,x_m)_p \to
  \infty$.  We know from Lemma \ref{lem:gromprod} that $d(p,x_n),
  d(p,x_m) \ge (x_n,x_m)_p$, which shows that $d(p,x_n) \to \infty$ as
  $n \to \infty$. It remains to show that $\angle_p(x_n,x_m) \to
  0$. Let $U_{px_n}, U_{px_m}$ be $\delta$-tubes around the geodesic
  arcs $\sigma_{px_n}$ and $\sigma_{px_m}$. Then the geodesic
  $\sigma_{x_n x_m}$ must contain a point $p_1 \in U_{px_n} \cap
  U_{px_m}$. We conclude from Lemma \ref{lem:gromprod} that
  $$ d(p_1,p) \ge d(\sigma_{x_nx_m},p) \ge (x_n,x_m)_p. $$
  Let $\gamma_1$ and $\gamma_2$ be the shortest curves connecting
  $p_1$ with $\sigma_{px_m}$ and $\sigma_{px_m}$ at the points $y_n$
  and $y_m$, see Figure \ref{fig:gromthin}. Then
  $d(p_1,y_n),d(p_1,y_m) \leq \delta$, which implies $d(y_n,y_m) \le 2\delta$
  and
  $$ d(y_n,p), d(y_m,p) \ge (x_n,x_m)_p - \delta. $$

  \begin{figure}[h]
  \begin{center}
    \psfrag{Upxn}{$U_{px_n}$} 
    \psfrag{Upxm}{$U_{px_m}$}
    \psfrag{xn}{$x_n$}
    \psfrag{xm}{$x_m$}
    \psfrag{p1}{$p_1$}
    \psfrag{p}{$p$}
    \psfrag{yn}{$y_n$}
    \psfrag{ym}{$y_m$}
         \includegraphics[width=10cm]{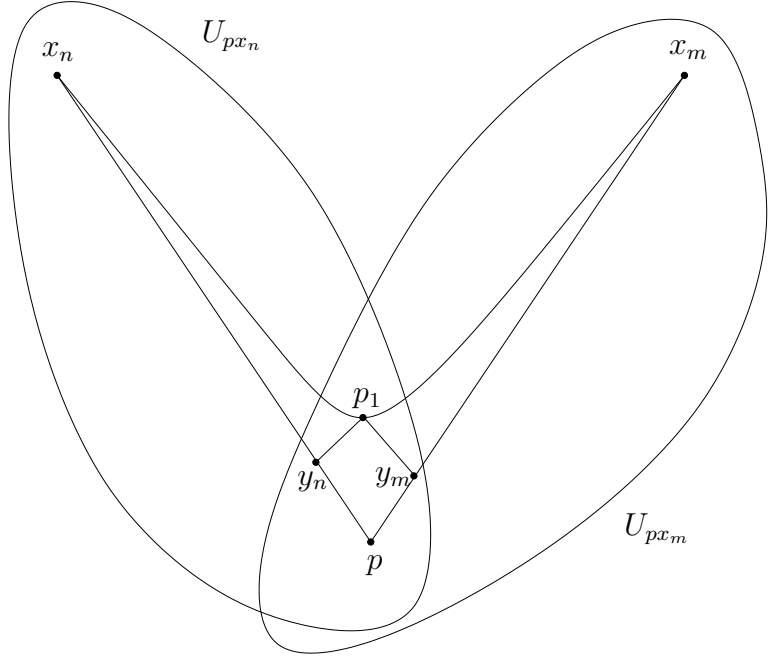}
  \end{center}
  \caption{Illustration of the proof of $(b) \Rightarrow (a)$ in Theorem
    \ref{thm:gromconv}}
  \label{fig:gromthin}
  \end{figure}

  We assume, without loss of generality, that $d(y_n,p) \ge
  d(y_m,p)$. Let $z_m \in \sigma_{px_m}$ be such that $d(p,z_m) =
  d(p,y_n)$. This implies that
  $$ d(y_n,p) = d(z_m,p) \ge (x_n,x_m)_p - \delta. $$
  Since 
  $$ d(y_m,p) \le d(z_m,p) = d(y_n,p) \le d(y_m,p) + d(y_n,y_m)
  \le d(y_m,p) + 2 \delta, $$ 
  and since $y_m, z_m$ lie on the same geodesic arc $\sigma_{px_m}$,
  we have $d(y_m,z_m) \le 2 \delta$. This implies that
  $$ 
  d(z_m,y_n) \le d(y_m,y_n) + d(z_m,y_n) \le 2\delta + 2 \delta = 4 \delta. 
  $$ 
  Using Corollary \ref{cor:uniformdiv}, we conclude that
  $$ 4 \delta \ge {\rm length}(\sigma_{y_nz_m}) \ge a(d(y_n,p)) 
  \angle_p(x_n,x_m). $$ Since $d(y_n,p) \to \infty$, we also have
  $a(d(y_n,p)) \to \infty$, which
  implies that $\angle_p(x_n,x_m) \to 0$.\\

  $(a) \Rightarrow (b)$: Assume $\angle_p(x_n,x_m) \to 0$ and
  $d(x_n,p) \to \infty$ for $n,m \to \infty$. For all $R > 0$, there
  exists $n_0(R) \ge 0$, such that for all $n,m \geq n_0(R)$:
  \begin{equation} \label{eq:smdist}  
  d(p,x_n), d(p,x_m) \ge R \quad \text{and} \,
  d(c_{px_n}(R),c_{px_m}(R)) \leq 1,
  \end{equation}
  since $\angle_p(x_n,x_m) \to 0$ for $n,m \to \infty$. Note that the
  constant $n_0(R)$ does not depend on $p$, but only on the values
  $d(p,x_n)$ and $\angle_p(x_n,x_m)$, since $X$ has a uniform lower
  curvature bound.

  We show now the following: {\em The geodesic arc $\sigma_{x_n x_m}$ has
  empty intersection with the open ball $B_{R-\frac{1}{2}}(p)$ for all
  $n,m \geq n_0(R)$.}

  If $\sigma_{x_n x_m} \cap B_R(p) = \emptyset$, there is nothing to
  prove. If $\sigma_{x_n x_m} \cap B_R(p) = \emptyset$, there exists a
  first $t_0 > 0$ and a last $t_1 > 0$ such that
  $$
  q_1 = \sigma_{x_n x_m}(t_0), \, q_2 = \sigma_{x_n x_m}(t_1) \in S_R(p),
  $$
  where $S_R(p)$ denotes the sphere of radius $R > 0$ around $p$ (see
  Figure \ref{fig:ballR}). Then we have
  $$
  d(q_1,q_2) = l(\sigma_{x_n x_m}) - d(x_n,q_1) - d(x_m,q_2).
  $$

  \begin{figure}[h]
  \psfrag{xn}{$x_n$} 
  \psfrag{xm}{$x_m$}
  \psfrag{q1}{$q_1$}
  \psfrag{q2}{$q_2$}
  \psfrag{SR(p)}{$S_R(p)$}
  \psfrag{spxn(R)}{$\sigma_{px_n}(R)$}
  \psfrag{spxm(R)}{$\sigma_{px_m}(R)$}
  \begin{center}
         \includegraphics[width=8cm]{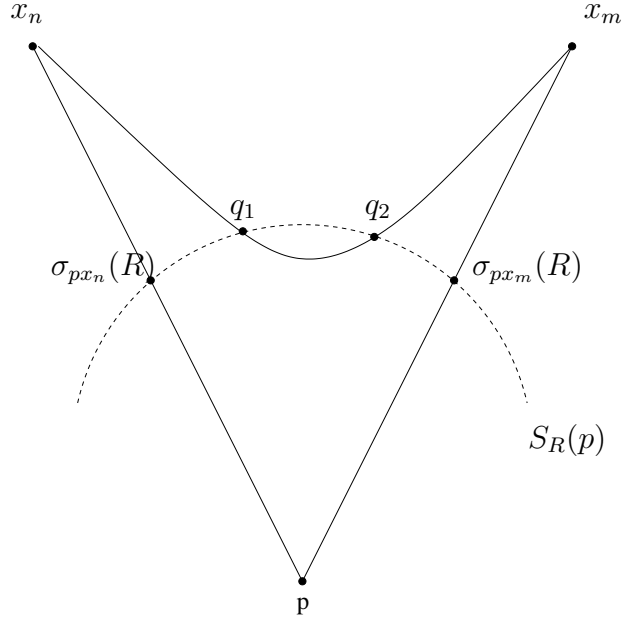}
  \end{center}
  \caption{Illustration of the proof of $(a) \Rightarrow (b)$ in Theorem
    \ref{thm:gromconv}}
  \label{fig:ballR}
  \end{figure}
  
  Using \eqref{eq:smdist}, we have 
  \begin{eqnarray*}
    l(\sigma_{x_n x_m}) & \leq & d(x_n,\sigma_{px_n}(R)) + d(\sigma_{px_n}(R), 
    \sigma_{px_m}(R)) + d(\sigma_{px_m}(R),x_m)\\
    & \leq & d(x_n, \sigma_{px_n}(R)) + d(x_m,\sigma_{px_m}(R)) + 1,
  \end{eqnarray*}
  which implies that
  \begin{equation} \label{eq:dq1q2}\begin{aligned}
  d(q_1,q_2) &\leq \left(d(x_n,\sigma_{px_n}(R)) - d(x_n,q_1)\right) \\ &\qquad +
  \left( d(x_m,\sigma_{px_m}(R)) - d(x_m,q_2) \right) + 1. \end{aligned}
  \end{equation}
  Since $d(p,x_n) = R+d(\sigma_{px_n}(R),x_n) \leq d(q_1,x_n) + R$ (by
  the triangle inequality), we obtain $d(x_n,q_1) - d(x_n,\sigma_{px_n}(R))
  \geq 0$, and similarly $d(x_m,q_2) - d(x_m,\sigma_{px_m}(R)) \geq 0$. This,
  together with \eqref{eq:dq1q2} shows $d(q_1,q_2) \leq 1$. But then the
  geodesic segment of $\sigma_{x_nx_m}$ between $q_1$ and $q_2$ cannot enter
  the ball $B_{R-\frac{1}{2}}(p)$.\\

  Therefore, we have for all $n,m \ge n_0(R)$,
  $$
  R- \frac{1}{2} \leq d(p,\sigma_{x_n x_m}) \leq (x_n,x_m)_p + 32\delta, 
  $$
  using Proposition \ref{prop:gromprod2}. This shows that 
  $$
  (x_n, x_m)_p \to \infty \quad \text{as $n,m \to \infty$.}
  $$
\end{proof}

\section{The geometric boundary of a harmonic space with
  purely exponential volume growth}
\label{sect:geombound}

This chapter provides a purely self-contained introduction into the
geometric boundary $X(\infty)$, based on equivalent geodesic rays, and
its associated cone topology for harmonic spaces with purely exponential volume growth. 

We like to mention that all results in the later chapters
\ref{chp:dirichprob}, \ref{chp:horopoly} and \ref{chp:meanvalueinf}
could have also been formulated in terms of the Busemann boundary
$\partial_B X$ instead of the geometric boundary $X(\infty)$, in which
case the current chapter as well as the following Chapter
\ref{sect:busgeombd} would be of no relevance for these later results.

We first recall the classical definition of the geometric boundary.

\begin{dfn}
  The {\em geometric boundary} $X(\infty)$ of the harmonic space $(X,g)$
  consists of all equivalence classes of geodesic rays, where two unit
  speed geodesic rays $\sigma_1, \sigma_2 : [0, \infty) \to X$ are
  equivalent if there exists $C > 0$ such that
  $$ d (\sigma_1(t), \sigma_2(t) \leq C $$ 
  for all $ t \geq 0$.  The equivalence class of a (unit speed)
  geodesic ray $\sigma$ is denoted by $[\sigma]$. The geodesic ray
  with initial vector $v \in SX$ is denoted by $\sigma_v$.
\end{dfn}

\begin{prop} \label{prop:unifdiv}
  For every $p \in X$, the map
  \begin{eqnarray*}
  \Phi_p: S_p X &\to& X(\infty),\\
  \Phi_p(v) &=& [\sigma_v]
  \end{eqnarray*}
  is injective.
\end{prop}

\begin{proof}
  This is an immediately consequence of the uniform divergence of geodesics
  (see Corollary \ref{cor:uniformdiv}).
\end{proof}

\begin{remark} Proposition \ref{prop:unifdiv} holds in the general
  context of noncompact harmonic manifolds (without the purely
  exponential volume growth condition).
\end{remark}

\begin{lemma} \label{lem:anosA}
  There exists a universal constant $A > 0$, only depending on $X$,
  such that for all unit speed geodesic rays $\sigma_1, \sigma_2 : [0,
  \infty) \to X$ with $\sigma_1(0) = \sigma_2(0)$ we have
  $$
  d(\sigma_1(t), \sigma_2(t)) \leq A d(\sigma_1(T), \sigma_2(T)) \quad 
  \forall\; 0 \leq t \leq T.
  $$
\end{lemma}

\begin{proof}
  Because of the uniform lower and upper bound on the sectional
  curvature of $X$, there exist $A_0 \ge 1$ and $\epsilon > 0$, only
  depending on $X$, such that
  $$ d(\sigma_1(t),\sigma_2(t)) \le A_0 d(\sigma_1(T),\sigma_2(T)) \quad
  \forall\; 0 \leq t \leq T \leq \epsilon. $$ 
  Using again the lower curvature bound on $X$ and Corollary
  \ref{cor:uniformdiv}, we can find for every $R > 0$ a constant
  $A_1(R) \ge 1$ such that
  $$ d(\sigma_1(t),\sigma_2(t)) \le A_1(X,R) d(\sigma_1(T),\sigma_2(T)) \quad
  \forall\; \epsilon \leq t \leq T \leq R. $$ 
  We assume now that $R > 0$ is chosen large enough that $R/2$ is greater
  than a certain universal constant $c>0$, introduced later and only
  depending on $X$. It remains to show that there exists a universal
  constant $A_2 \ge 1$, only depending on $X$, such that
  $$ d(\sigma_1(t),\sigma_2(t)) \le A_2(X,R) d(\sigma_1(T),\sigma_2(T)) \quad
  \forall\; R \leq t \leq T. $$
  Let $T > R$, and $c : [0,1] \to X$ be a geodesic connecting $\sigma_1(T)$ with
  $\sigma_2(T)$, written as
  $$
  c(s) = \exp_p r(s) v(s),
  $$
  where $p = \sigma_1(0)$, $r(0) = r(1) = T$ and $v(s) \in S_p X$ for
  all $0 \leq s \leq1$. Assume first that there is $s_0 \in (0,1)$
  such that $r(s_0)= d(c(s_0),p) \le T/2$. Then
  $d(\sigma_1(T),\sigma_2(T)) \ge T$ and we have
  $$ 
  d(\sigma_1(t),\sigma_2(t)) \le 2t \le 2T \le 2 d(\sigma_1(T),\sigma_2(T)). 
  $$ 
  So we can disregard this case and assume that $r(s_0) = d(c(s_0),p)
  > T/2$, for all $s_0 \in [0,1]$. Let $R \leq t = \delta T$, and
  $c_\delta : [0,1] \to X$ be given by $c_\delta(s) = \exp_p \delta
  r(s) v(s)$. Then
  \begin{eqnarray*}
  \left.\frac{d}{ds} \right|_{s = s_0} c_\delta(s) &=& D \exp_p (\delta
  r(s_0) v(s_0) (\delta r'(s_0) v(s_0) + \delta r(s_0) v'(s_0))\\
  &=& r'(s_0) c_{v(s_0)}' (\delta r(s_0)) + A_{v(s_0)} (\delta r(s_0)) (v'(s_0)
  \end{eqnarray*}
  Since $c_{v(s_0)}' (\delta r(s_0)) \perp A_{v(s_0)} (\delta r(s_0))
  (v'(s_0))$, we obtain
  \begin{multline*}%
    \left \| \left.\frac{d}{ds} \right |_{s=s_0} c_\delta(s) \right
    \|^2 = (r'(s_0))^2 + ||A_{v(s_0)} (\delta r(s_0)) v'(s_0)||^2\\
    \leq (r'(s_0))^2 + \|A_{v(s_0)} (\delta r(s_0)) A_{v(s_0)}^{-1}
    (r(s_0))||^2 \cdot \|A_{v(s_0)} (r(s_0)) (v'(s_0))\|^2.
  \end{multline*}
  Let $B(t) = A_{v(s_0)}(t) A_{v(s_0)}^{-1} (r(s_0))$. This is an
  orthogonal Jacobi-Tensor along $c_{v(s_0)}$. Let $w \in (c_{v(s_0)}'
  (r(s_0)))^\perp$ with $||w|| = 1$. Then $J(t) = B(t) w$ is a Jacobi
  field with $J(0) = 0, J(r(s_0)) = w$. Using the Anosov property of
  the geodesic flow and the theorem in \cite[p. 107]{Bo}, we
  condude the existence of a universal $A_2 \geq 2$, such that
  \begin{eqnarray} \label{eq:jacobiest}
  \|J(t)\| \leq A_2\; \|J(r(s_0))\| = A_2
  \end{eqnarray}
  for all $t \in [c,r(s_0)]$ with a universal constant $c > 0$, only
  depending on $X$. Therefore
  $$
   \|B(t)\| \leq A_2 
  $$
  for all $c \leq t \leq r(s_0)$. Note that $\delta r(s_0) > \delta
  \frac{T}{2} \ge \frac{R}{2}$, and that we assumed earlier that $R/2 > c$. 
  This implies that
   \begin{eqnarray*}
     \left \| \left.\frac{d}{ds} \right|_{s=s_0} c_\delta(s) \right\|^2 &\leq& A_2^2 
     \left((r'(s_0))^2 + \|A_{v(s_0)}(r(s_0)) v'(s_0)\|^2 \right)\\
     &\leq& A_2^2 \left \| \left.\frac{d}{ds} \right|_{s=s_0} c(s) \right\|^2,
   \end{eqnarray*}
  which shows that
  \begin{eqnarray*} 
    {\rm length}(c_\delta) &=& \int\limits_0^1 \left\| \left.\frac{d}{ds} 
      \right|_{s=s_0} c_\delta(s)\right\| ds \\
    &\leq& A_2 \int\limits_0^1 \left\|\left.\frac{d}{ds} \right|_{s=s_0}
      c(s)\right \| ds = A_2\; d(\sigma_1(T), \sigma_2(T)).
   \end{eqnarray*}
   Since $c_\delta$ connects $\sigma_1(\delta T)$ with
   $\sigma_2(\delta T)$, we conclude
  $$
   d(\sigma_1(t), \sigma_2(t)) \leq A_2\; 
   d(\sigma_1(T), \sigma_2(T)) \quad \text{for all $R \le t = \delta T \leq T$}.
   $$
\end{proof}

\begin{prop}
  For every $p \in X$, the map $\Phi_p : S_p X \to X(\infty)$ is bijective.
\end{prop}

\begin{proof}
  In view of Proposition \ref{prop:unifdiv}, it suffices to prove
  surjectivity of $\Phi_p$. Let $[\sigma] \in X(\infty)$ and
  $\sigma(0) = q$. Choose $t_n \to \infty$ and $v_n \in S_p X$
  and $s_n \in {\R}$ such that $\sigma(t_n) = \exp_p(s_n
  v_n)$. Obviously
  $$
  |t_n - s_n| \leq d(p,q).
  $$
  Then $d(p, \sigma(t_n)) = d(q_n, \sigma(t_n)) = s_n$, where $q_n : =
  \sigma(t_n - s_n)$. By Lemma \ref{lem:anosA}, we have
  \begin{eqnarray*}
    d(c_{v_n}(t), \sigma(t_n - s_n + t)) \leq A\; d(p,q_n) & \leq & 
    A\; (d(p,q) + |t_n - s_n|)\\
    & \leq & (A+1)\; d(p,q) 
  \end{eqnarray*}
  for all $t \in [0,s_n]$. This implies that
  \begin{eqnarray*}
    d(c_{v_n}(t), \sigma(t)) & \leq & 
    d(c_{v_n}(t), \sigma(t_n-s_n+t)) + d(\sigma(t_n-s_n+t),\sigma(t))\\
    & \leq & (A+1)\; d(p,q) + |t_n-s_n| \leq (A+2)\; d(p,q)
  \end{eqnarray*}
  for all $t \in [0,s_n]$. Since $S_p X$ compact, there exists $v_0
  \in S_p X$ such that a subsequence of $v_n$ converges to $v_0$.  We
  conclude that $d(c_{v_0}(t), \sigma(t) \leq (A+2)\; d(p,q)$ for all $t
  \geq 0$, i.e., $[\sigma] = [c_{v_0}]$. This shows surjectivity.
\end{proof}

The next proposition is an easy consequence of the uniform divergence
of geodesics: For every distance $d>0$ and every angle $\epsilon>0$
there exists a uniform radius $\widehat{R}$ such that any two points
$p,q$ at distance $\le d$ and outside any ball of radius $\widehat{R}$
will be seen from the center of this ball in an angle $\le \epsilon$.

\begin{prop} \label{prop:hatR}
  For $d > 0$ and $\epsilon > 0$ there exists $\widehat{R} =
  \widehat{R}(d,\epsilon,X) \geq 0$, such that for all $p_0 \in X$ and
  for all $ p,q \not \in B_{\widehat{R}}(p_0)$ with $d(p,q) \leq d$ we
  have
  $$
  \angle_{p_0}(p,q) \leq \epsilon.
  $$
\end{prop}

\begin{proof}
  Since $a(t) \to \infty$ as $t \to \infty$, we can choose
  $\widehat{R} > 0$ such that, for all $t \ge \widehat{R}$, we have
  $\frac{2d}{a(t)} \le \epsilon$. Let $p,q \not\in
  B_{\widehat{R}}(p_0)$ with $d(p,q) \le d$. Let $v,w \in S_{p_0}X$
  and $t_1, t_2 > \widehat{R}$ such that $p = c_v(t_1)$ and $q =
  c_w(t_2)$. We conclude from the triangle inequality that $|t_2 -
  t_1| \le d$. Let $q_0 = c_w(t_1) \not\in B_{\widehat{R}}(p_0)$. Then
  $$ d(p,q_0) \le d(p,q) + d(q,q_0) \le d + |t_2 - t_1| \le 2d. $$ 
  Then Corollary \ref{cor:uniformdiv} implies that 
  $$ \angle_{p_0}(p,q) = \angle_{p_0}(p,q_0) \le \frac{d(p,q_0)}{a(t_1)}
  \le \frac{2d}{a(t_1)} \le \epsilon, $$ 
  since $t_1 \ge \widehat{R}$. This finishes the proof.
\end{proof}

The next result states that far out points of a geodesic ray $\sigma$,
seen from another point $p$ at bounded distance $\le d$ from
$\sigma(0)$, appear under a very small angle.

\begin{prop} \label{prop:STangle}
  Let $d > 0$ be given. Let $\sigma: [0,\infty) \to X$ be a unit speed
  geodesic ray with $q:=\sigma(0)$. Let $p \in B_d(q) = \{ z
  \in X \mid d(z,q) \le d \}$. For $T > 0$, let $\sigma_T :
  [0,d(p,\sigma(T))] \to X$ be the unit speed geodesic from $p$ to
  $\sigma(T)$. Then for $\epsilon > 0$ there exists $C = C(X,d,\epsilon) > 0$
  such that for all $S,T \geq C$,
  $$
  \angle_p(\sigma_S'(0),\sigma_T'(0)) \le \epsilon.
  $$
\end{prop}

\begin{proof}
  The details of the following proof are illustrated in Figure
  \ref{fig_propangle}. By \cite{Kn3}, $X$ is $\rho$-Gromov hyperbolic
  for some $\rho >0$. Since $X$ has lower curvature bound, we can find
  $\delta > 0$ such that for all $z \in X$ and for all $ v_1,v_2 \in
  S_zX$ we have
  \begin{equation} \label{eq:smanglesmdist}
  \angle_z(v_1,v_2) \leq \delta \quad \Rightarrow \quad
  d(c_{v_1}(\rho+1),c_{v_2}(\rho+1)) \leq 1.
  \end{equation}

  \begin{figure}[h]
  \begin{center}
    \psfrag{p}{$p$} 
    \psfrag{q}{$q$}
    \psfrag{ld}{$\le d$}
    \psfrag{le}{$\le \epsilon$}
    \psfrag{R0}{$R_0$}
    \psfrag{gR}{$\ge \widehat{R}$}
    \psfrag{ld}{$\le d$}
    \psfrag{ldelta}{$\le \delta$}
    \psfrag{l1}{$\le 1$}
    \psfrag{s(T)}{$\sigma(T)$}
    \psfrag{s(S)}{$\sigma(S)$}
    \psfrag{s}{$\sigma$}
    \psfrag{sS}{$\sigma_S$}
    \psfrag{sT}{$\sigma_T$}
    \psfrag{ss(t0)}{$\sigma_S(t_0)$}
    \psfrag{Sr+1(s(T))}{$S_{\rho+1}(\sigma(T))$}
    \psfrag{l2r+1}{$\le 2\rho +1$}
        \includegraphics[width=12cm]{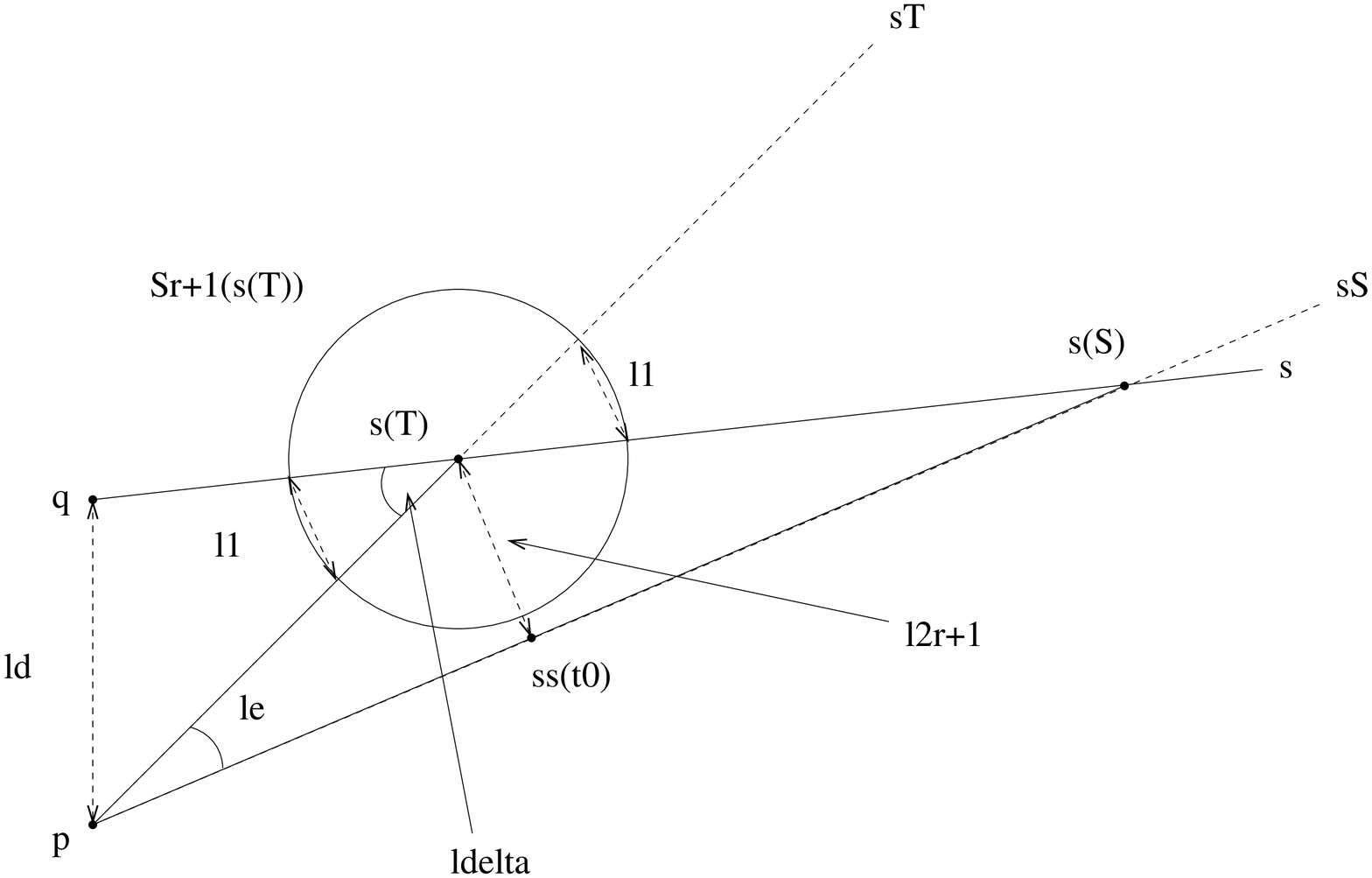}
  \end{center}
  \caption{Illustration of the proof of Proposition \ref{prop:STangle}}
  \label{fig_propangle}
  \end{figure}

  Choose
  $$
  C: = \max \{\widehat{R}(d,\delta,X), 
  \widehat{R}(2 \rho+1,\epsilon,X) +2 \rho+1\} + d
  $$
  with $\widehat{R}(d,\epsilon,X)$ as in Proposition
  \ref{prop:hatR}. Without less of generality, we can assume that $S
  \ge T$. We know that
  $$
  d(p,\sigma(T)), d(q,\sigma(T)) \geq T - d \geq \widehat{R}(d,\delta,X),
  $$
  and conclude from Proposition \ref{prop:hatR} that
  $$
  \angle_{\sigma(T)}(p,q) \leq \delta.
  $$
  Using \eqref{eq:smanglesmdist}, this implies that
  $$
  d(\sigma_T(d(p,\sigma(T)) \pm(\rho+1)), \sigma(T \pm(\rho+1))) \leq 1.
  $$
  Using \cite[Cor. 4.5]{Kn3}, we conclude that there exists $t_0 > 0$
  such that $d(\sigma(T), \sigma_S(t_0)) \leq 2 \rho+1$. (In the case
  $S \in [T,T+(\rho+1)]$, we choose $t_0 = d(p,\sigma(S)$ and have
  $\sigma_S(t_0)=\sigma(S)$.) Since $T \geq C \geq \widehat{R}(2
  \rho+1,\epsilon,X) + 2\rho+1 + d$, we have
  $$
  d(\sigma(T),p) \geq d(\sigma(T),q) - d(p,q) = T-d(p,q) \geq
  \widehat{R}(2 \rho+1, \epsilon,X) + 2\rho+1
  $$
  and
  \begin{eqnarray*}
    d(\sigma_S(t_0),p) & \geq & d(\sigma(T),p) - d(\sigma_S(t_0), \sigma(T))\\
    & \geq & \widehat{R} (2 \rho+1,\epsilon,X) + 2\rho+1 - (2\rho+1) = 
    \widehat{R}(2\rho+1,\epsilon,X).
  \end{eqnarray*}
  Using Propositon \ref{prop:hatR} again, we conclude that
  $$
  \angle_p(\sigma(T), \sigma_S(t_0)) = 
  \angle_p(\sigma_T'(0),\sigma_S'(0)) \leq \epsilon,
  $$
  finishing the proof of the proposition.
\end{proof}

This proposition has the following limit version (for $S \to \infty$).

\begin{cor} \label{cor:Tangle} Let $d > 0$ be given. Let $\sigma:
  [0,\infty) \to X$ be a unit speed geodesic ray with
  $q:=\sigma(0)$. Let $p \in B_d(q) = \{ z \in X \mid d(z,q) \le d \}$
  and $v \in S_pX$ such that $[\sigma_v] = [\sigma]$. Then for
  $\epsilon > 0$ we have
  $$
  \angle_p(v,\sigma_T'(0)) \le \epsilon
  $$
  for all $T \ge C(X,d,\epsilon)$ with $\sigma_T$ and $C(X,d,\epsilon)$ 
  as in Proposition \ref{prop:STangle}. 
\end{cor}

\begin{proof}
  We only need to show that $\lim_{T \to \infty} \sigma_T'(0) =
  v$. Then we obtain for $T \ge C(X,d,\epsilon)$:
  $$ \angle_p(v,\sigma_T'(0)) \le \limsup_{S \to \infty} 
  \angle_p(\sigma_S'(0),\sigma_T'(0)) \le \epsilon, $$ using
  Proposition \ref{prop:STangle}.

  We know from Proposition \ref{prop:STangle} that, for any $T_n \to
  \infty$, $\sigma_{T_n}'(0)$ is a Cauchy sequence in $S_pX$. Since
  $S_pX$ is compact, there exists $v_0 = \lim_{T \to \infty}
  \sigma_T'(0)$. So it remains to show that $v_0 = v$. Note that we have
  $$ T-d(p,q) \le d(p,\sigma(T)) \le T+d(p,q). $$
  Let $v_1 = -\sigma'(T)$ and $v_2 = - \sigma_T'(d(p,\sigma(T)))$. Figure
  \ref{fig_limpropangle} illustrates the following inequalities.

  \begin{figure}[h]
  \begin{center}
    \psfrag{p=sT(0)}{$p=\sigma_T(0)$} 
    \psfrag{q=s(0)}{$q=\sigma(0)$}
    \psfrag{ld}{$\le d$}
    \psfrag{s(T)=sT(dpsT)}{$\sigma(T)=\sigma_T(d(p,\sigma(T)))$}
    \psfrag{s}{$\sigma$}
    \psfrag{sT}{$\sigma_T$}
    \psfrag{s(s)}{$\sigma(s)$}
    \psfrag{sT(s)}{$\sigma_T(s)$}
    \psfrag{v1}{$v_1$}
    \psfrag{v2}{$v_2$}
        \includegraphics[width=12cm]{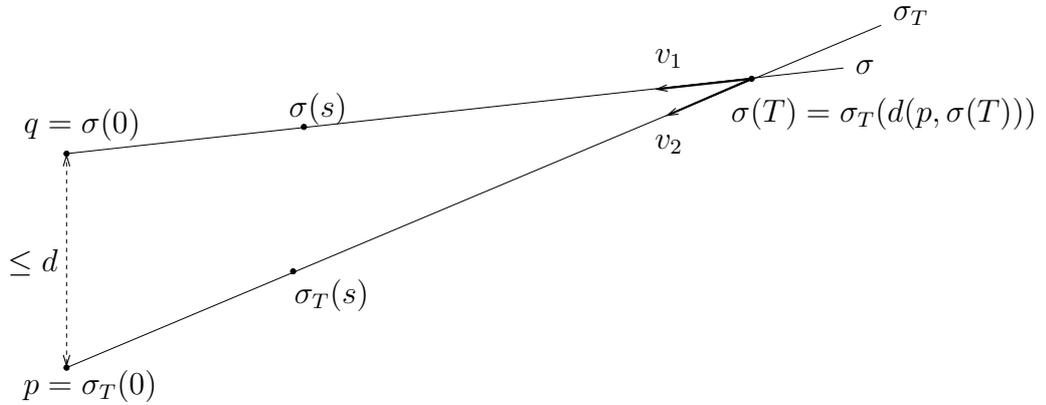}
  \end{center}
  \caption{Illustration of the proof of Corollary \ref{cor:Tangle}}
  \label{fig_limpropangle}
  \end{figure}

  Using Lemma \ref{lem:anosA}, we have for $0 \le s \le T-d(p,q)$:
  \begin{eqnarray*}
    d(\sigma(s),\sigma_T(s)) &=& 
    d(\sigma_{v_1}(T-s),\sigma_{v_2}(d(p,\sigma(T))-s)) \\
    &\le& d(\sigma_{v_1}(T-s),\sigma_{v_2}(T-s)) + \\
    && d(\sigma_{v_2}(T-s),\sigma_{v_2}(d(p,\sigma(T))-s)) \\
    &\le& A\; d(\sigma_{v_1}(T),\sigma_{v_2}(T)) + |T-d(p,\sigma(T))| \\
    &\le& A\; \left(d(q,p)+d(\sigma_{v_2}(d(p,\sigma(T))),\sigma_{v_2}(T)) 
    \right) + d(p,q) \\
    &\le& 2A\; d(p,q) + d(p,q) \le (2A+1)d.
  \end{eqnarray*}
  Taking the limit, we conclude that 
  $$ d(\sigma(s),\sigma_{v_0}(s)) \le (2A+1)d $$
  for all $s \ge 0$, i.e., $[\sigma_{v_0}] = [\sigma]$. Proposition
  \ref{prop:unifdiv} implies that $v_0 = v$. 
\end{proof}

For $v_0 \in S_xX$, $R > 0$, $\delta > 0$, we define
\begin{equation} \label{eq:Uv0Rd}
U(v_0,R,\delta) : = \{\sigma_v(t) \mid\; t > R, v \in S_xX\;
\text{with}\; \angle_x(v_0,v) < \delta\},
\end{equation}
which we consider as geometric neighbourhoods of points at the
geometric boundary.

\begin{prop} \label{prop:neighborinc}
  Let $d > 0$ be given. Let $v_0 \in S_pX$ and $v_1 \in S_qX$ with
  $[\sigma_{v_0}] = [\sigma_{v_1}]$ and $d(p,q) \le d$. Then for all
  $R,\delta > 0$ there exist $R'=R'(X,d,R,\delta)$ and
  $\delta'=\delta'(X,d,R,\delta)$ such that
  $$
  U(v_1,R',\delta') \subset U(v_0,R,\delta).
  $$
\end{prop}

\begin{proof}
  We choose
  $$ R' := \max \left\{ R+d, C(X,d,\frac{\delta}{4}),
    \widehat{R}(1,\frac{\delta}{4},X) + d \right\}, $$ 
  where $\widehat{R}$ and $C$ are defined as in Propositions
  \ref{prop:hatR} and \ref{prop:STangle}. Let $\sigma =
  \sigma_{v_1}$. Since $R' \ge C(X,d,\frac{\delta}{4})$, we conclude
  with Corollary \ref{cor:Tangle} that
  \begin{equation} \label{eq:angle1} 
  \angle_p(v_0,\sigma(R')) \le \frac{\delta}{4}. 
  \end{equation}
  Since the curvature of $X$ is bounded from below, there exists
  $\delta' > 0$, only depending on $R'$ and $X$, such that 
  $$ \sigma_v(R') \in B_1(\sigma(R')) \qquad \forall\; v \in S_qX\;
  \text{with}\; \angle_p(v,v_1) < \delta'. $$
  Figure \ref{fig_neighborinc} illustrates the following arguments.
 
  \begin{figure}[h]
  \begin{center}
    \psfrag{p}{$p$} 
    \psfrag{q}{$q$}
    \psfrag{ld'}{$\le \delta'$}
    \psfrag{ld/4}{$\le \delta/4$}
    \psfrag{s(R')}{$\sigma(R')$}
    \psfrag{B1(s(R'))}{$B_1(\sigma(R'))$}
    \psfrag{s}{$\sigma$}
    \psfrag{sv(t)}{$\sigma_v(t)$}
    \psfrag{sv(R')}{$\sigma_v(R')$}
    \psfrag{v0}{$v_0$}
    \psfrag{v1}{$v_1$}
        \includegraphics[width=12cm]{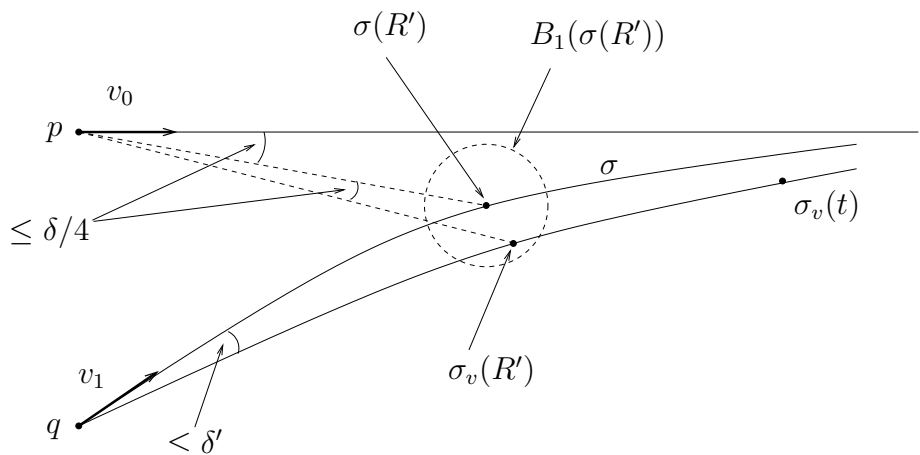}
  \end{center}
  \caption{Illustration of the proof of Proposition
    \ref{prop:neighborinc}}
  \label{fig_neighborinc}
  \end{figure}

  Let $z \in U(v_1,R',\delta')$. Then $z = \sigma_v(t)$ with 
  $\angle_p(v,v_1) < \delta'$ and $t > R'$. 

  Since $R' \ge \widehat{R}(1,\frac{\delta}{4},X)+d$, we have
  $$ 
  d(p,\sigma(R')), d(p,\sigma_v(R')) \ge \widehat{R}(1,\frac{\delta}{4},X)
  $$
  and $d(\sigma(R'),\sigma_v(R')) \le 1$. Using Proposition
  \ref{prop:hatR}, we conclude that
  \begin{equation} \label{eq:angle2} 
  \angle_p(\sigma(R'),\sigma_v(R') \le \frac{\delta}{4}. 
  \end{equation}
  Since $t \ge R' \ge C(X,d,\frac{\delta}{4})$, we deduce from Proposition
  \ref{prop:STangle}:
  \begin{equation} \label{eq:angle3} 
  \angle_p(\sigma_v(R'),\sigma_v(t)) \le \frac{\delta}{4}. 
  \end{equation}
   
  Bringing \eqref{eq:angle1}, \eqref{eq:angle2} and \eqref{eq:angle3}
  together, we obtain
  \begin{eqnarray*}
    \angle_p(z,v_0) &\le& \angle_p(\sigma_v(t),\sigma_v(R')) + 
    \angle_p(\sigma_v(R'),\sigma(R')) + \angle_p(\sigma(R'),v_0) \\
    &\le& \frac{\delta}{4} + \frac{\delta}{4} + \frac{\delta}{4} < \delta
  \end{eqnarray*}
  and 
  $$ d(z,p) \ge d(\sigma_v(t),q)-d(p,q) \ge t-d > R'-d \ge R, $$
  i.e., $z \in U(v_0,R,\delta)$. This finishes the proof.
\end{proof}  

Let $(X,g)$ be a noncompact harmonic space {\em with purely
  exponential volume growth}. The geometric compactification
$\overline{X} = X \cup X(\infty)$ is now the disjoint union of all the
points in $X$ and the equivalence classes $[\sigma]$ of unit speed
geodesic rays $\sigma: [0,\infty) \to X$. The above considerations
lead to the following natural topology on $\overline{X}$: A basis of
this compact topological space is given by the open balls
$U_\epsilon(p) = \{ x \in X \mid d(x,p) < \epsilon \}$ with $p \in X$
(neighbourhoods of finite points $p \in X$) and the sets
$U(v_0,R,\delta)$, defined in \eqref{eq:Uv0Rd} (neighbourhoods of the
infinite point $[\sigma_{v_0}] \in X(\infty)$). For all $p \in X$, the
map
\begin{eqnarray*}
\Phi_p: \{v \in T_p X \mid\; ||v|| \leq 1\} = B_p(1) &\longrightarrow& X \cup 
X(\infty),\\
\Phi_p(v) &=& 
\begin{cases}
\exp_p \frac{v}{1-\Vert v \Vert} & \text{if $\Vert v \Vert < 1$,}\\
[\sigma_v] & \text{if $\Vert v \Vert = 1$,}
\end{cases}
\end{eqnarray*}
is a homeomorphism. The same holds true for the restriction: For all
$p \in X$, the map
\begin{equation} \label{eq:idbdhom} 
\Phi_p: S_pX \to X(\infty), \quad \Phi_p(v) = [\sigma_v] 
\end{equation}
is a homeomorphism.

\section{Busemann functions and the geometric boundary}
\label{sect:busgeombd}

We begin with the following definition:

\begin{dfn}
  Let $(X,g)$ be a connected noncompact complete Riemannian
  manifold. Two unit vectors $v,w \in SX$ are {\em asymptotic
    directions}, if the corresponding geodesic rays $\sigma_v,
  \sigma_w: [0,\infty) \to X$ with $\sigma_v'(0) = v$ and
  $\sigma_w'(0) = w$ stay within bounded distance, i.e., there is a
  constant $C > 0$ such that
  $$ d(\sigma_v(t),\sigma_w(t)) \le C $$
  for all $t \ge 0$. In other words, $v$ and $w$ are asymptotic
  directions iff $\sigma_v$ and $\sigma_w$ define the same equivalence
  class.
\end{dfn}

Let $(X,g)$ be a noncompact harmonic space. For all $v \in SX$ and $t
\in \R$ let $b_{v,t}(q) = d(q,\sigma_v(t))-t$. The Busemann function $b_v$
is then defined as
$$
b_v(q) = \lim\limits_{t \to \infty} b_{v,t}(q).
$$

\begin{prop}
  Let $(X,g)$ be a noncompact harmonic space with purely exponential
  volume growth and $v \in SX$. Then the Busemann function $b_v$ is
  differentiable, and the vector field $Z(q) = -\grad b_v(q)$ is a
  vector field of asymptotic directions.
\end{prop}

\begin{proof}
  We know from Proposition \ref{prop:busprops} that $b_v$ is
  differentiable and that we have $\grad b_v = \lim_{t \to \infty}
  \grad b_{v,t}$, where the convergence is uniform on compact sets.\\

  Let $w = - \grad b_v(q) = \lim_{t \to \infty} - \grad b_{v,t}(q)$. Since we have
  $$
  d(\sigma_v(s+(t-d(q,\sigma_v(t)))), \sigma_{-\grad b_{v,t}(q)}(s))
  \leq A\; d(\sigma_v(t-d(q,\sigma_v(t))),q),
  $$
  for all $0 \leq s \leq d(q,\sigma_v(t))$, by Lemma \ref{lem:anosA},
  we conclude
  \begin{eqnarray*}
    d(\sigma_v(s),\sigma_{-\grad b_{v,t}(q)}(s)) & \leq & |t-d(q,c_v(t))| + 
    A\; d(\sigma_v(t-d(q,\sigma_v(t))),q)\\
    & \leq & |t-d(q,\sigma_v(t))| + A\left(|t-d(q,\sigma_v(t))| + 
      d(p,q)\right)\\
    & \leq & d(p,q) + A (d(p,q) + d(p,q)) = (2A+1)d(p,q),
  \end{eqnarray*}
  for all $0 \leq s \leq d(q(c_v(t)))$. Keeping $s > 0$ fixed, and
  taking the limit $t \to \infty$, we obtain 
  $$
  d(\sigma_v(s), \sigma_{-\grad b_v(q)}(s)) \leq (2A+1) d(p,q) \quad
  \forall\; s \geq 0,
  $$
  i.e., $v$ and $Z(q) = -\grad b_v(q)$ are asymptotic directions.
\end{proof}

\begin{cor} \label{cor:busediff}
  Let $(X,g)$ be a noncompact connected harmonic space with
  purely exponential volume growth. If $v,w \in SX$ are asymptotic
  directions, then $b_v - b_w$ is constant.
\end{cor}

\begin{proof}
  For all $q \in X$, the vectors $-\grad b_v(q), -\grad b_w(q) \in S_q
  X$ are asymptotic to $v$ and, therefore, asymptotic to each
  other. Because of Proposition \ref{prop:unifdiv}, we have
  $$
  -\grad b_v(q) = -\grad b_w(q) \quad \forall q \in X.
  $$
  This implies that $\grad(b_v - b_w) \equiv 0$ and, therefore, $b_v -
  b_w$ must be constant on $X$.
\end{proof}

\begin{dfn}
  Let $(X,g)$ be a noncompact connected harmonic space with purely
  exponential volume growth. For $p \in X$ and $\xi \in X(\infty)$, we
  define
  $$
  b_{p, \xi}(q) = b_v(q),
  $$
  where $v \in S_pX$ is given by $[\sigma_v] = \xi$. Note that
  $$ b_{p,\xi}(p) = 0 \quad \text{and}\, -\grad b_{p,\xi}(p) = v. $$
\end{dfn}

The next result states that the Busemann boundary and the geometric
boundary agree for noncompact harmonic spaces with purely exponential
volume growth.

\begin{theorem} \label{thm:busegeombd}
  Let $(X,g)$ be a noncompact connected harmonic space with purely
  exponential volume growth and $p_0 \in X$ be a reference point. Then
  there is a canonical homeomorphism $\partial_B^{p_0} X \to
  X(\infty)$, given by $b_v \mapsto [\sigma_v]$ for all $v \in
  S_{p_0}X$.
\end{theorem}

\begin{proof}
  We recall from Proposition \ref{prop:busboundprops}(2) that the map
  $\varphi_{p_0}: S_{p_0}X \to \partial_B^{p_0}X$, defined by
  $\varphi_{p_0}(v) = b_v$ is a homeomorphism. We saw at the end of
  Chapter \ref{sect:geombound} that the map $\Phi_{p_0}: S_{p_0}X \to X(\infty)$,
  $\Phi_p(v) = [\sigma_v]$ is a homeomorphism (see \eqref{eq:idbdhom}). The
  canonical homeomorphism introduced in the theorem is the composition
  of these two homeomorphisms.
\end{proof}

Let us, finally, return to the visibility measures $\mu_p \in
{\mathcal M}_1(\partial_B^{p_0})$, introduced in Chapter
\ref{sect:visib}. Using the canonical homeomorphism $\partial_B^{p_0}X
\to X(\infty)$ in Theorem \ref{thm:busegeombd}, we can view these as
probability measures on the geometric boundary $X(\infty)$. Then we
have $d\mu_{p_0}([\sigma_v]) = \frac{1}{\omega_n} d\theta_{p_0}(v)$
for all $v \in S_{p_0}X$ and, because of the identity
$$ d\mu_p(\xi) = e^{-h b_{p_0,\xi}(p)}\; d\mu_{p_0}(\xi) $$
for all $\xi \in \partial_B^{p_0}$, we have the identity
$$ d\mu_p([\sigma_v]) = e^{-h b_v(p)}\; d\mu_{p_0}([\sigma_v]) $$
for all $v \in S_{p_0}X$ representing $[\sigma_v] \in X(\infty)$. This
will become important in Chapter \ref{chp:dirichprob} below on the
solution of the Dirichlet problem at infinity.

\section{Solution of the Dirichlet problem at infinity}
\label{chp:dirichprob}

Let $(X,g)$ be a harmonic manifold with purely exponential volume
growth. Recall that the Busemann function associated to $v \in S_pX$
is defined as
$$ b_v(q) : = \lim\limits_{t \to \infty} d(c_v(t),q)-t. $$
For $v_0 \in S_pX$ and $\delta > 0$, we introduce the cone
$$
C(v_0, \delta) = \{c_v(t) \mid\; t \geq 0, \angle(v_0,v) \leq \delta\}.
$$
We already mentioned at the end of Chapter \ref{chp:repbdharmfunc} the
crucial condition $\lim_{x \to \xi} \mu_x = \delta_\xi$ to solve the
Dirichlet problem at infinity. This abstract condition can be deduced
from the geometric fact that any horoball $\mathcal H$, centered at
$\xi = c_{v_0}(\infty) \in X(\infty)$, ends up inside any given cone
$C(v_0,\delta)$, when being translated to the horoball $\widetilde
{\mathcal H}$ along the stable direction (see the illustration in
Figure \ref{fig:horoincone}). This is essentially the content of the
following proposition. (Note that the horoballs centered at $\xi$ can
be described by $\{ q \in X \mid b_{v_0}(q) \le -C \}$, and that these
horoballs become smaller and shrink towards the limit point $\xi$, as
$C \in \R$ increases to infinity.)

  \begin{figure}[h]
  \psfrag{H}{$\mathcal H$} 
  \psfrag{Ht}{$\widetilde{\mathcal H}$}
  \psfrag{v0}{$v_0$}
  \psfrag{Cv0d}{$C(v_0,\delta)$}
  \psfrag{X(inf)}{$X(\infty)$}
  \psfrag{xi}{$\xi$}
  \begin{center}
         \includegraphics[width=8cm]{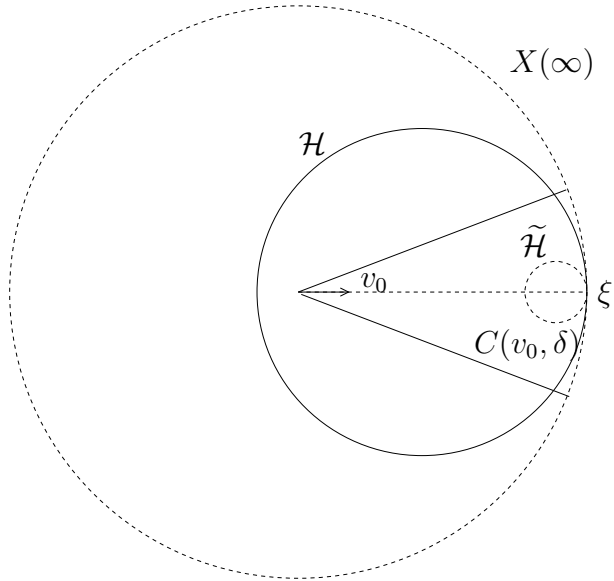}
  \end{center}
  \caption{Geometric property to guarantee the solution of the Dirichlet
    problem at infinity}
  \label{fig:horoincone}
  \end{figure}

\begin{prop} \label{prop:horoincone}
  Let $(X,g)$ be a harmonic space with purely exponential volume growth.
  Let $v_0 \in S_pX$ and $\delta > 0$. Then there exists a constant $C_1
  > 0$, depending only on $v_0$ and $\delta$, such that
  $$
  b_{v_0}(q) \geq d(p,q)-C_1 \quad \text{for all $q \in X \backslash 
  C(v_0, \delta)$.} 
  $$
\end{prop}

\begin{remark}
  Proposition \ref{prop:horoincone} does not hold if $(X,g)$ is the
  Euclidean space. In this case, every horoball is a halfspace, which
  lies never inside a given cone.
\end{remark}

\begin{proof}
  There exists a constant $C_1 > 0$ such that
  \begin{equation} \label{eq:gromprodest} 0 \leq 2(c_{v_0}(t),q)_p
    \leq C_1 \quad \forall\; t \geq 0 \quad \forall\; q \in X \backslash
    C(v_0, \delta).
  \end{equation}
  If this were false, then we could find sequences $t_n \geq 0$, $q_n
  \in X \backslash C(v_0,\delta)$, such that
  $$
  (c_{v_0}(t_n),q_n)_p \to \infty.
  $$
  This would mean, by Theorem \ref{thm:gromconv}, that $d(p,q_n) \to
  \infty$ and $\angle_p(v_0,q_n) \to 0$, which is a contradiction to
  $q_n \in X \backslash C(v_0,\delta)$.

  \eqref{eq:gromprodest} means that
  $$
  d(p,q) - (d(c_{v_0}(t),q)-t) \leq C_1 \quad \forall\; t \ge 0.
  $$
  Taking the limit $t \to \infty$, we obtain
  $$
  d(p,q) - b_{v_0}(q) = d(p,q) - \lim\limits_{t \to \infty}
  (d(c_{v_0}(t),q)-t) \leq C_1,
  $$
  finishing the proof.
\end{proof}

In fact, we need the following {\em uniform} modification of Proposition
\ref{prop:horoincone}. Note that in Proposition \ref{prop:busdest}
below, $v \in S_pX$ plays the role of $v_0$ in Proposition
\ref{prop:horoincone}, and every $x \in C(v_0,\frac{\delta}{2})$
satisfies $x \in X \backslash C(v,\frac{\delta}{2})$, because of
$\angle_p(v,v_0) \ge \delta$.

\begin{prop} \label{prop:busdest} Let $(X,g)$ be a harmonic space with
  purely exponential volume growth.  Let $v_0 \in S_pX$ and $\delta >
  0$. Then exists a $C_2>0$, depending only on $v_0$ and $\delta$,
  such that
  $$
  b_v(x) \geq d(p,x)-C_2,
  $$
  for all $x \in C(v_0,\frac{\delta}{2})$ and all $v \in S_pX$ with
  $\angle_p(v,v_0) \geq \delta$.
\end{prop}

\begin{proof}
  There exists a constant $C_2>0$ such that
  \begin{multline*}
  0 \le 2(c_v(t),x)_p \leq C_2 \quad \forall\; t \geq 0,\quad 
  \forall\; v \in S_pX\, \text{with}\, \angle_p(v,v_0) \geq \delta,\\
   \forall\; x \in C(v_0,\frac{\delta}{2}),
  \end{multline*}
  for otherwise, we could fine sequences $t_n \geq 0$, $v_n \in S_pX$
  with $\angle_p(v_n,v_0) \geq \delta$, and $x_n \in
  C(v_0,\frac{\delta}{2})$ satisfying
  $$
  (c_{v_n}(t_n),x_n)_p \to \infty.
  $$
  Using Theorem \ref{thm:gromconv}, this would imply $d(p,x_n) \to \infty$
  and $\angle_p(v_n,x_n) \to 0$. But $\angle_p(v_n,x_n) \to 0$ contradicts
  to $\angle_p(x_n,v_0) \le \frac{\delta}{2}$ and $\angle_p(v_n,v_0) \ge \delta$.

  Therefore, we have
  $$
  d(x,p) - (d(c_v(t),x)-t) \leq C_2 \quad \forall\; t \ge 0,
  $$
  which implies, taking $t \to \infty$, that
  $$
  d(x,p) - b_v(x) \leq C_2 \quad \forall\; v \in S_pX\, \text{with}\,
  \angle_p(v,v_0) \geq \delta \ \text{and}\, \forall\; x \in
  C(v_0,\frac{\delta}{2}),
  $$
  finishing the proof.
\end{proof}

Now we state our main result of this chapter, namely, the solution of
the Dirichlet problem at infinity in case of purely exponential volume
growth.

\begin{theorem} \label{thm:dirichprob} Let $(X,g)$ be a harmonic space
  with purely exponential volume growth. Let $\varphi: X(\infty) \to
  {\R}$ be a continuous function. Then there exists a unique harmonic
  function $H_\varphi: X \to {\R}$ such that
  \begin{equation} \label{eq:convdirich}
  \lim\limits_{X \to \xi} H_\varphi(x) = \varphi(\xi).
  \end{equation}
  Moreover, $H_\varphi$ has the following integral presentation:
  $$
  H_\varphi(x) = \int\limits_{X(\infty)} \varphi(\xi) d \mu_x(\xi),
  $$
  where $\{\mu_x\}_{x \in X} \subset {\mathcal M}_1(X(\infty))$ are
  the {\em visibility probability measures} (originally introduced in
  Definition \ref{def:vismeas} on $\partial_BX$, and recalled as
  measures on $X(\infty)$ at the end of Chapter \ref{sect:busgeombd}).
\end{theorem}

\begin{remark}
  Note the differences between the earlier Theorem
  \ref{thm:harmoncrep} and Theorem \ref{thm:dirichprob}. The earlier
  theorem states the rather abstract fact that every bounded harmonic
  function $F$ on $X$ can be represented by a certain integral of a
  function $\varphi$ defined on the boundary, using a rather abstract
  {\em harmonic measure}. Theorem \ref{thm:harmoncrep} makes no
  statement about the convergence of $F(x) \to \varphi(\xi)$, as $x
  \in X$ converges to $\xi$. Theorem \ref{thm:dirichprob} is
  formulated in the context of continuity, the involved measures are
  the explicitly given visibility measures $\mu_x$, and it
  additionally states the crucial convergence $F(x) \to \varphi(\xi)$.
\end{remark}

\begin{proof}
\item[(a)] We show first that $\int\limits_{X(\infty)} \varphi(\xi) d
  \mu_x(\xi)$ is a harmonic function. Let $p \in X$. Then
  \begin{multline*}
    \Delta_x \int\limits_{X(\infty)} \varphi(\xi) d \mu_x(\xi) =
    \Delta_x \int\limits_{X(\infty)} \varphi(\xi)
    \frac{d \mu_x}{d \mu_p}(\xi) d \mu_p(\xi) = \\
    \Delta_x \int\limits_{X(\infty)} \varphi(\xi) e^{-h b_{p,\xi}(x)}
    d \mu_p(\xi).
  \end{multline*}
  Let $K \subset X$ be a compact set. Then $x \mapsto \varphi(\xi)
  e^{-h b_{p,\xi}(x)}$ is bounded for all $x \in K$ and all $\xi \in
  X(\infty)$, because of $|b_{p,\xi}(x)| \leq d(p,x)$. Moreover
  $\Delta_x \varphi(\xi) e^{-h b_{p,\xi}(x)} = 0$ and
  $b_{p,\xi}(\cdot)$ is smooth, because of $\Delta_x b_{p,\xi} =
  h$. Therefore,
  $$
  \Delta_x \int\limits_{X(\infty)} \varphi(\xi) d \mu_x(\xi) =
  \int\limits_{X(\infty)} \varphi(\xi) \underbrace{\Delta_x e^{- h
      b_{p,\xi}(x)}}_{= 0} d \mu_p(\xi) = 0.
  $$

  \smallskip

  \item[(b)] Now we prove
  $$
  \lim\limits_{x \to \xi_0} \int\limits_{X(\infty)} \varphi(\xi) d
  \mu_x(\xi) = \varphi(\xi_0).
  $$
  Let $\xi_0 = [c_{v_0}]$ with $v_0 \in S_p X$. Without loss of
  generality, we can assume that $\varphi(\xi_0) = 0$ (by subtracting
  a constant if necessary). Let $\epsilon > 0$ be given. Then there
  exists $\delta > 0$, such that
  $$
  |\; \varphi([c_v])\; | \leq \epsilon \quad \forall\; v \in S_pX \;
  \text{with}\; \angle_p(v_0,v) \leq \delta.
  $$
  We split the integral representing $H_\varphi(x)$ in the following way:
  \begin{multline*}
    \omega_n |H_\varphi(x)| \leq \left| \int_{S_pX\; \backslash\;
        \{v\; \mid\; \angle(v_0,v) \leq \delta\}} \varphi([c_v])\;
      e^{-h b_v(x)}\; d\theta_p(v)\right| + \\
    \left |\int_{\{v\; \mid\; \angle(v_0,v) \leq \delta\}}
      \varphi ([c_v])\; e^{-h b_v(x)}\; d \theta_p(v) \right|.
  \end{multline*}
  Now, using Proposition \ref{prop:busdest}, we obtain
  \begin{multline*}
    \omega_n |H_\varphi(x)| \leq \Vert \varphi \Vert_\infty \int_{S_pX\;
      \backslash\; \{v\; \mid\; \angle(v_0,v) \leq \delta\}}
    e^{-h(d(p,x)-C_2)}\; d \theta_p(v) + \\
    \epsilon \int_{\{v\; \mid\; \angle(v_o,v) \leq \delta\}}
    e^{-h b_v(x)}\; d \theta_p(v) \leq\\
    \Vert \varphi \Vert_\infty\; \omega_n\; e^{hC_2}\; e^{-hd(p,x)} + \epsilon
    \underbrace{\int_{S_pX} e^{-hb_v(x)}d \theta_p(v)}_{=\int_{S_xX}
      d \theta_x(v) = \omega_n} \leq\\
    \omega_n \left(\epsilon + \Vert \varphi \Vert_\infty\; e^{hC_2}\;
      e^{-hd(p,x)}\right).
  \end{multline*}
  Since $\epsilon
  > 0$ was arbitrary and $d(p,x) \to \infty$ for $x \to \xi_0$, we
  conclude that
  $$
  |H_\varphi(x)| \to 0\quad \text{for}\, x \to \xi_0.
  $$

  \smallskip

  \item[(c)] Uniqueness of the solution follows from the maximum
  principle.
\end{proof}

Let us finish this chapter with an application of Theorem
\ref{thm:dirichprob} (see formulas \eqref{eq:appl1} and
\eqref{eq:appl2} below).

\begin{remark}
  Obviously, the harmonic function $h_v : X \to {\R}$, introduced in
  Theorem \ref{thm:hvharm}, has a continuous extension to the
  compactification $\overline{X} = X \cup X(\infty)$ via
  $$
  h_v(q) =
  \begin{cases}\mu(r) \langle v,w \rangle & \text{if}\; q = \exp_p(rw)
    \in X
  \; \text{with}\; w \in S_pX,\\
  \frac{1}{h} \langle v,w \rangle & \text{if}\; q  = [c_w] \in X(\infty)\;
  \text{with}\; w \in S_pX.
  \end{cases}
  $$
  This implies that the harmonic map $F_{E_p} : X \to
  B_{\frac{1}{h}}(0)$, introduced in Chapter \ref{sec:harmdiffeo}, has
  a extension as a homeomorphism $\overline{F}_{E_p} : \overline{X}
  \to \overline{B_{\frac{1}{h}}(0)}$ with $\overline{F}_{E_p}(p) = 0$.

  Since $h_v : \overline{X} \to {\R}$ and its restriction
  on $X$ is harmonic, we know from Theorem \ref{thm:dirichprob} that
  $$
  h_v(x) = \frac{1}{h} \frac{1}{\omega_n} \int\limits_{S_pX} e^{-h
    b_w(x)} \langle v,w \rangle d \theta_p(w).
  $$
  On the other hand, we have
  $$
  h_v(x) = \mu(d_p(x)) \cdot \langle v,w_p(x) \rangle,
  $$
  which implies that
  \begin{equation} \label{eq:appl1}  
  \frac{1}{\omega_n} \int\limits_{S_pX} e^{-h b_w(x)}
  \langle v,w \rangle d \theta_p(w) = h \cdot \mu(d_p(x)) \cdot \langle
  v,w_p(x) \rangle,
  \end{equation}
  or
  \begin{equation} \label{eq:appl2}
  \frac{1}{\omega_n} \int\limits_{S_pX}
  \underbrace{e^{-h b_w(x)} w}_{\in T_pX} d \theta_p(w) = h \cdot
  \mu(d_p(x)) \cdot w_p(x).
  \end{equation}
\end{remark}

\section{Horospheres of harmonic spaces with 
purely exponential volume growth have 
polynomial volume growth}
\label{chp:horopoly}

The Anosov property implies for all $v \in SX$ the existence of a
splitting
$$
T_vSX = E^s(v) \oplus E^u (v) \oplus E^c(v)
$$
and constants $a \ge 1$ and $b >0$ such that for all $\xi \in E^s(v)$
\begin{equation} \label{eq:anos}
\|D \phi^t(v) \xi\| \leq a \|\xi\| e^{-bt} \quad \forall\; t \ge 0.
\end{equation}
Let $W_v^s \subset SX$ be the corresponding strong stable manifold,
i.e., the integral manifold associated to the distribution $E^s$
through $v \in SX$. Its projection ${\mathcal H}_v = \pi W_v^s \subset
X$ is a horosphere orthogonal to $v$. Let $p = \pi(v)$. 
Consider a curve
$$
\xi : [0,1] \to W_v^s \quad \xi(0) = v
$$
in the strong stable manifold $W_v^s$ such that
$$
r \geq {\rm length}(\pi \circ \xi) = \int\limits_0^1 \|(\pi \circ
\xi)'(s)\| ds.
$$ 
Note that the geodesic flow $\phi^t: SX \to SX$ induces a bijection
$\phi^t: W_v^S \to W_{\phi^t v}^s$. We apply this to the curve $\xi$
(see Figure \ref{fig_horoshrink}). Then
$$
\pi \circ \phi^t \xi : [0,1] \to \mathcal{H}_{\phi^t_{v}},
$$
and, since $(\phi^t \xi)'(s)=D \phi^t(\xi(s))(\xi'(s)) \in E^s(\xi(s))$,
\begin{eqnarray*}
  {\rm length}(\pi \circ \phi^t \xi) & = & 
  \int\limits_0^1 \|(\pi \circ \phi^t \xi)'(s)\| ds\\
  & \leq & \int\limits_0^1 \| (\phi^t \xi)'(s) \| ds 
  \stackrel{\eqref{eq:anos}}{\leq} ae^{-bt} \int\limits_0^1 \| \xi'(s) \| ds.
\end{eqnarray*}

\begin{figure}[h]
  \begin{center}
    \psfrag{p=pi(v)}{$p=\pi(v)$} 
    \psfrag{v}{$v$}
    \psfrag{Hv=piWsv}{${\mathcal H}_v=\pi W_v^s$}
    \psfrag{Hptv=piWsptv}{${\mathcal H}_{\phi^t(v)}=\pi W_{\pi^t(v)}^s$}
    \psfrag{phit(v)}{$\phi^t(v)$}
    \includegraphics[width=8cm]{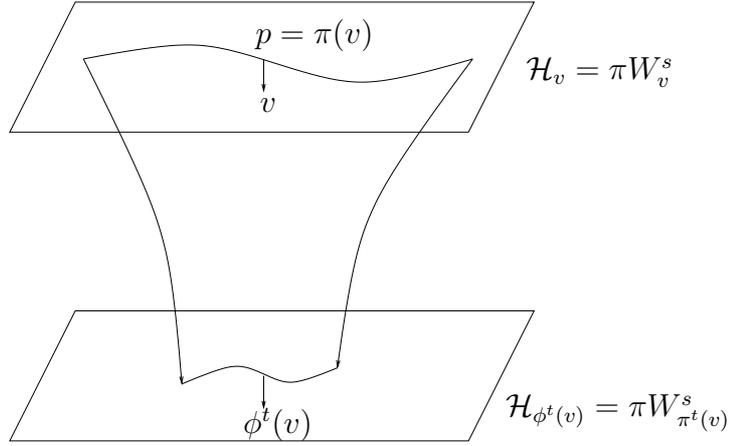}
  \end{center}
  \caption{Contraction of the geodesic flow on stable horospheres}
  \label{fig_horoshrink}
\end{figure}

Since the sectional curvature of a harmonic space is bounded, we conclude
that the second fundamental form of horospheres is bounded as well (see the
proof of part (A) of Proposition \ref{prop:curvsphhoro}). Therefore, there
exists $C >0$ such that
$$
\|\nabla_{(\pi \circ \xi)'(s)} \xi(s)\| \leq C \| \pi \circ \xi)'(s) \|.
$$
This implies
\begin{eqnarray*}
\| \xi'(s) \|^2 &= & \|(\pi \circ \xi)'(s)\|^2 + \| \frac{D}{ds} \xi(s)\|^2 = \|(\pi \circ \xi)'(s)\|^2 + \|\nabla_{(\pi \circ \xi)'(s)} \xi(s)\|^2\\
&\leq& (1+C^2) \|(\pi \circ \xi)'(s)\|^2
\end{eqnarray*}
and
$$
{\rm length}(\pi \circ \phi^t \xi) \leq ae^{-bt} \sqrt{1+C^2}\; {\rm
  length} (\pi \circ \xi) \leq ae^{-bt} \sqrt{1+C^2}\; r.
$$
Hence ${\rm length}(\pi \circ \phi^t \xi) \leq 1$  if
$$
ae^{-bt} \sqrt{1+C^2}\; r \leq 1,
$$
or equivalently
$$
e^{-bt} \leq \frac{1}{a\sqrt{1+C^2}\; r},
$$
which means
$$t \geq \frac{\log(a \sqrt{1+C^2}\; r)}{b} = : t_0.
$$
Let $B_{\mathcal{H}}v = \nabla_v \xi$, where $\xi$ is the inward unit
normal vector field of the horosphere $\mathcal H$. We know from above
that $\|B_{\mathcal{H}}\| \leq C$.

Recall the Gauss equation
$$
\langle R(X,Y)Y,X \rangle = \langle R^{\mathcal{H}} (X,Y)Y,X \rangle +
\langle X,B_{\mathcal{H}} Y \rangle^2 - \langle B_{\mathcal{H}} X,X
\rangle \langle B_{\mathcal{H}} Y,Y \rangle
$$
with $X,Y \in T_q \mathcal{H}$. If $X,Y$ are orthonormal, we have
$$
|K_{\mathcal{H}} (\Span \{X,Y\})| \leq |K(\Span \{X,Y\})| + 2
\|B_{\mathcal{H}}\|^2 \leq \tilde{C},
$$ 
for some positive constant $\tilde{C}>0$, since harmonic manifolds
have bounded sectional curvature. Therefore the Volume Comparison
Theorem yields that any ball of radius $1$ in any horosphere has an
intrinsic volume bounded by some constant $A>0$:
$$
\vol_{\mathcal{H}} (B_1(q)) \leq A \quad \forall\; \mathcal{H}\;
\text{horospheres}\; \forall\; q \in \mathcal{H}.
$$
This implies that
\begin{eqnarray*}
\vol_{\mathcal{H}_{v}} (B_r(p)) & \leq & 
\vol_{\mathcal{H}_{v}} (\phi^{-t_0} (B_1 (\pi \circ \phi^{t_0}(v))))\\
& \leq & e^{ht_0} \vol_{\mathcal{H}_{\phi^{t_0}(v)}} (B_1 (\pi \circ \phi_{t_0}(v))) 
\leq A e^{ht_0}\\
& = & A e^{ \frac{h}{b} \log (a\sqrt{1+C^2}\; r) }= A(a\sqrt{1+C^2}\; r)^{h/b}.
\end{eqnarray*}
This proves the statement in the title of this chapter.$\qquad \qquad \qquad
\qquad \square$

\section{Mean value property of harmonic functions at infinity}
\label{chp:meanvalueinf}

In this chapter, we modify the arguments given in \cite{CaSam} for
asymptotic harmonic manifolds of negative curvature. The flow of
arguments follows also the arguments given in \cite{KP}.

\begin{theorem}
  Let $(X,g)$ be a noncompact harmonic manifold of dimension $n \in \N$ with
  purely exponental volume growth. Let $\varphi : \overline{X} = X \cup
  X(\infty) \to {\R}$ be continuous, and its restriction $\varphi: X
  \to {\R}$ be harmonic. Let $\xi \in X(\infty)$ and $p_0 \in X$. Let
  $\mathcal{H} \subset X$ be the horosphere, centered at $\xi$,
  containing the point $p_0$. Let $K_j \subset \mathcal{H}$ be an
  exhaustion of $\mathcal{H}$, such that $\partial K_j$ is smooth and
  satisfying
  $$
  \frac{\vol_{n-2}(\partial K_j)}{\vol_{n-1}(K_j)} \to 0 \quad
  \text{as}\; j \to \infty.
  $$
  Then we have the following ''mean value property at infinity'':
  $$
  \lim\limits_{j \to \infty} \frac{\int_{K_j} \varphi(x)
    dx}{\vol_{n-1}(K_j)} = \varphi(\xi).
  $$
\end{theorem}

\begin{remark} Since horospheres have polynomial volume growth, the
  intrinsic balls of suitably chosen increasing radii satisfy
  $$
  \frac{\vol_{n-2}(\partial
    B_{\mathcal{H}}(r_j))}{\vol_{n-1}(B_{\mathcal{H}}(r_j))} \to 0.
  $$
  A suitable choice of sets $K_j$ are regularized spheres, as
  explained in \cite[p. 665]{KP}. But there might be many more increasing sets
  satisfying this asymptotic isoperimetric property.
\end{remark}

\begin{proof}
  Let $\xi = [c_v]$ with $v \in S_{p_0}X$ and $\mathcal{H} =
  b_v^{-1}(0)$. Let $\mathcal{H}_t = b_v^{-1}(t)$. Let $\phi_t : X \to
  X$ be the flow assiciated to $\grad b_v = \grad b_{p_0,\xi}$. Then
  $\phi_t : \mathcal{H}_0 \to \mathcal{H}_t$. Let $K_j(t) = \phi_t
  (K_j) \subset \mathcal{H}_t$. Then
  $$
  \vol_{n-1} (K_j(t)) = e^{ht} \vol_{n-1} (K_j).
  $$
  Since $X$ has a lower sectional curvature bound, there exists $C >
  0$ such that
  $$
  \vol_{n-2}(\partial K_j(t)) \leq e^{C|t|} \vol_{n-2}(\partial K_j).
  $$
  This implies that, on every compact set $I \subset [0, \infty)$, we
  have
  $$
  || \frac{\vol_{n-2}(\partial
    K_j(\cdot))}{\vol_{n-1}(K_j(\cdot))}||_{\infty,I} \to 0, \quad
  \text{as}\: j \to \infty.
  $$
  Define
  $$
  g_j(t) = \frac{\int_{K_j(t)} \varphi(x)dx}{\vol_{n-1}(K_j(t))} \quad
  \forall\; t \in {\R}.
  $$
  Since $||g_j||_\infty \leq ||\varphi||_\infty$, using diagonal
  arguments, we find a subsequence $g_{j_k}$ such that $g_{j_k}(t) \to
  g(t)$, for all rational $t$. Since $\varphi$ is uniformly
  continuous, we have $g_{j_k} \to g$ pointwise to a continuous limit.\\

  Next we show that $g$ satisfies
  \begin{equation} \label{eq:diffeq}
    g'' + hg' = 0, 
  \end{equation}
  in the distributional sense. Let $\psi \in C_0^\infty ({\R})$ be a
  test function. Then we have
  $$ \int\limits_{- \infty}^\infty g_j(t) (\psi''(t) - h
  \psi'(t))dt = \int\limits_{- \infty}^\infty \frac{\int_{K_j(t)}
    \varphi(x)dx}{\vol_{n-1}(K_j(t))} (\psi''(t) - h
  \psi'(t))dt.
  $$
  Let $\tilde{\varphi} : \mathcal{H} \times (- \infty,\infty) \to
  {\R}$ be defined as $\tilde{\varphi}(x,t) : =
  \varphi(\phi_t(x))$. The tranformation formula yields:
  \begin{multline*}
    \int\limits_{K_j(t)} \varphi(x)dx = \int\limits_{\phi_t(K_j)}
    \varphi(x)dx = \int\limits_{K_j} \varphi \circ \phi_t(x)
    \overbrace{{\rm Jac}\; \phi_t(x)}^{e^{ht}}dx = \\
    = e^{ht} \int\limits_{K_j} \tilde{\varphi} (x,t)dx.
  \end{multline*}
  Therefore, we have
  $$
  g_j(t) = \frac{1}{\vol(K_j)} \int\limits_{K_j} \varphi(\phi_t x)dx,
  $$
  and
  \begin{eqnarray*}
    g_j''(t) + h g_j'(t) & = & \frac{1}{\vol(K_j)} 
    \int\limits_{K_j} \frac{d^2}{dt^2} \varphi(\phi_t x) + h \frac{d}{dt} 
    \varphi(\phi_t x)dx\\
    & = & \frac{1}{\vol(K_j)} \int\limits_{K_j} \Delta_x \varphi(\phi_t x) - 
    \Delta_{\mathcal{H}_t} \varphi(\phi_t x)dx\\
    & = & \frac{1}{\vol(K_j(t))} \int\limits_{K_j(t)} 
    \underbrace{\Delta_x \varphi(x)}_{= 0} - \Delta_{\mathcal{H}_t} \varphi(x)dx\\
    & = & - \frac{1}{\vol(K_j(t))} \int\limits_{K_j(t)} 
    \Delta_{\mathcal{H}_t} \varphi(x)dx\\
    & = & \frac{1}{\vol(K_j(t))} \int\limits_{\partial K_j(t)} 
    \langle \grad_{\mathcal{H}_t} \varphi(x), \nu_x \rangle dx,
  \end{eqnarray*}
  where $\nu_x$ denotes the outward unit vector of $\partial K_j(t)
  \subset {\mathcal{H}_t}$. Since ${\rm supp} \psi \subset \R$ is compact,
  we have  
  \begin{multline*}
    \int\limits_{- \infty}^\infty g_j(t)(\psi''(t) - h \psi'(t))dt
    = \int\limits_{- \infty}^\infty (g_j''(t) + h g_j'(t)) \psi(t)dt \\
    = \int\limits_{- \infty}^\infty \frac{1}{\vol_{n-1}(K_j(t))}
    \int\limits_{\partial K_j(t)} \langle \grad_{\mathcal{H}(t)}
    \varphi(x), v_x \rangle dx\, \psi(t)dt.
  \end{multline*}
  Taking absolute value, we conclude:
  \begin{multline*}
    \left| \int\limits_{- \infty}^\infty g_j(t)(\varphi''(t) - h
      \varphi'(t))dt \right| \\
    \leq \int\limits_{\supp \psi} \frac{\vol_{n-2} (\partial
      K_j(t))}{\vol_{n-1}(K_j(t))} \; ||\grad_X \varphi||_\infty\;
    ||\psi||_\infty\; dt \to 0,
 \end{multline*}
  as $j \to \infty$, since, by Theorem \ref{thm:intform}:
  $$
  \langle \grad_X \varphi(p),v \rangle = \frac{1}{\vol(B_1(p))}
  \int\limits_{S_1(p)} \varphi(q) \varphi_v(q) d \mu_1(q)\; \forall\;
  v \in S_p X,
  $$
  which implies
  $$
  ||\grad_X \varphi(p)|| \leq \frac{1}{\vol_n(B_1(p))} \vol_{n-1}
  (S_1(p)) ||\varphi||_\infty,
  $$
  i.e.,
  $$
  ||\grad_X \varphi||_\infty \leq \frac{\vol_{n-1}(S_1(p))}{\vol_n(B_1(p))}
  ||\varphi||_\infty.
  $$
  By Lebesgue's dominated convergence, and since $||g||_\infty,
  ||g_j||_\infty \leq ||\varphi||_\infty$, we conclude that
  $$
  \int\limits_{- \infty}^\infty g(t)(\varphi''(t) - h \varphi'(t))dt = 0,
  $$
  i.e., the continuous function $g$ satisfies \eqref{eq:diffeq} in the
  distributional sense. Therefore, $g$ is smooth and satisfies
  \eqref{eq:diffeq} in the classical sense, which implies
  $$
  g' + hg = c,
  $$
  for some suitably chosen constant $c \in \R$. The general solution
  of $g' + hg = c$ is $g(t) = c' e^{-ht} + \frac{c}{h}$ with an
  arbitrary constant $c' \in {\R}$. Since $g$ ist bounded, we
  have $g(t) = \frac{c}{h}$.\\

  Let $p_0 \in X$, $v \in S_{p_0}X$ and $\xi = [c_v] \in
  X(\infty)$. Let $\epsilon > 0$ and $R > 0$. Recall from Proposition
  \ref{prop:horoincone}, that we find $t \ge R$ such that
  \begin{equation} \label{eq:bU}
  b_{\xi,p_0}((- \infty, -t]) \subset U(v,R,\epsilon).
  \end{equation}
  
  Continuity of $\varphi : \overline{X} \to {\R}$ implies, for every
  $\epsilon > 0$, that there exists an open neighborhood $U$ of $\xi
  \in X(\infty)$, such that
  $$
  |\varphi(x) - \varphi(\xi)| < \epsilon \quad \forall\; x \in U.
  $$
  Choose $t < 0$ negative enough, such that $\mathcal{H}_t =
  b_{p_0,\xi}^{-1}(t) \subset U$. This is possible because of
  \eqref{eq:bU}. This implies
  $$
  |g_j(t) - \varphi(\xi)| = \frac{\int_{K_j(t)} |\varphi(x) -
    \varphi(\xi)|\; dx}{\vol_{n-1}(K_j(t))} \leq \epsilon.
  $$
  Therefore, we have $|g(t) - \varphi(\xi)| \leq \epsilon$ for $t < 0$
  negative enough. Since $g(t) = \frac{c}{h}$, i.e., $g$ is constant,
  we must have $g \equiv \varphi(\xi)$, since $\epsilon > 0$ was arbitrary.\\

  For $t = 0$, we conclude
  $$
  g_j(0) = \frac{\int_{K_j} \varphi(x)dx}{\vol(K_j)} \to g(0) = \varphi(\xi),
  $$
  as $j \to \infty$. This finishes the proof. 
\end{proof}


\end{document}